\documentclass[a4paper]{amsart}

\usepackage[utf8]{inputenc}
\usepackage{mathtools}
\usepackage{amssymb,amsfonts,amsmath}
\usepackage{enumitem}
\usepackage{color}
\usepackage{mathrsfs}
\usepackage{graphicx}
\usepackage{verbatim}
\usepackage{tikz-cd}
\usepackage{tikz}
\numberwithin{equation}{section}
\newtheorem{theorem}{Theorem}[section]
\newtheorem{corollary}[theorem]{Corollary}
\newtheorem{lemma}[theorem]{Lemma}
\newtheorem{proposition}[theorem]{Proposition}
\newtheorem{question}[theorem]{Question}
\newtheorem{problem}[theorem]{Problem}

\theoremstyle{definition}
\newtheorem{definition}[theorem]{Definition}
\newtheorem{notation}[theorem]{Notation}

\theoremstyle{remark}
\newtheorem{remark}[theorem]{Remark}

\newtheorem{example}[theorem]{Example}

\newcommand{\F}{\mathbb{F}}
\newcommand{\N}{\mathbb{N}}

\newcommand{\Z}{\mathbb{Z}}

\DeclareMathOperator{\Alt}{Alt}
\DeclareMathOperator{\Aut}{Aut}

\DeclareMathOperator{\GL}{GL}
\DeclareMathOperator{\id}{id}

\DeclareMathOperator{\Homeo}{Homeo}

\DeclareMathOperator{\LEF}{LEF}

\DeclareMathOperator{\odd}{odd}

\DeclareMathOperator{\ord}{ord}

\DeclareMathOperator{\PSL}{PSL}

\DeclareMathOperator{\SAut}{SAut}
\DeclareMathOperator{\SL}{SL}

\DeclareMathOperator{\St}{St}
\DeclareMathOperator{\supp}{supp}
\DeclareMathOperator{\Sym}{Sym}
\DeclareMathOperator{\T}{T}
\DeclareMathOperator{\E}{E}
\DeclareMathOperator{\vol}{vol}
\DeclareMathOperator{\cons}{con}

\newcommand{\abs}[1]{\vert #1 \vert}
\newcommand\Set[2]{\{\,#1\mid#2\,\}}

\newcommand{\defeq}{\mathrel{\mathop{:}}=}

\renewcommand{\epsilon}{\varepsilon}

\title[Telescopes, frames, simple groups]{From telescopes to frames and simple groups}
\author[S. Kionke]{Steffen Kionke}
\author[E. Schesler]{Eduard Schesler}
\address{FernUniversit\"at in Hagen \\ Fakult\"at f\"ur Mathematik und Informatik \\
58084 Hagen}
\email{steffen.kionke@fernuni-hagen.de}
\email{eduard.schesler@fernuni-hagen.de}
\thanks{Funded by the Deutsche Forschungsgemeinschaft (DFG, German Research Foundation) - 441848266}

\begin{document}
\begin{abstract}
We introduce the notion of a telescope of groups. Very roughly a telescope is a directed system of groups that contains various commuting images of some fixed group $B$. Telescopes are inspired from the theory of groups acting on rooted trees. Imitating known constructions of branch groups, we obtain a number of examples of $B$-telescopes and discuss several applications.
We give examples of  $2$-generated  infinite amenable simple groups. We show that every finitely generated residually finite (amenable) group embeds into a finitely generated (amenable) $\LEF$ simple group. We construct $2$-generated frames in products of finite simple groups and show that there are Grothendieck pairs consisting of amenable groups and groups with property $(\tau)$.
We give examples of automorphisms of finitely generated, residually finite, amenable groups that are not inner, but become inner in the profinite completion. We describe non-elementary amenable examples of finitely generated, residually finite groups all of whose finitely generated subnormal subgroups are direct factors.
%These groups turn out to have an abundance of exotic properties.
\end{abstract}

\maketitle

\section{Introduction}
In this paper we develop the theory of telescopes of groups. In a nutshell, a $B$-\emph{telescope}
$\mathcal{S}$ consists of
a directed system $((\Omega_i)_{i\in \N},\iota_{i,j})$ of groups $\Omega_i$ together with homomorphisms $\phi_i\colon B \to \Omega_i$ such that the images of $\iota_{i,j}\circ\phi_i$ commute but are still large enough to normally generate each $\Omega_j$.
Each telescope gives rise to two groups: the group $G_{\mathcal{S}}$ and the head $Q_{\mathcal{S}}$.
The group $G_\mathcal{S}$ is a subgroup of $\prod_{i \in \N} \Omega_i$ and (often) $Q_{\mathcal{S}}$ is a factor of $G_\mathcal{S}$. Under suitable assumptions, it will turn out that both groups can have remarkable properties: $G_\mathcal{S}$ is a frame in $\prod_{i \in \N} \Omega_i$ and $Q_\mathcal{S}$ is simple. In addition, the construction of these groups is rather explicit and needs a small number of generators. This leads to a number of previously unknown examples in group theory, which we will present now.

\subsection*{Application 1: infinite amenable simple groups}
The notion of an amenable group\footnote{Von Neumann designated such group with the German word ``messbar'' (``measurable'' in English).
The fact that such groups are now called amenable goes back to Day~\cite{Day49}, apparently as a pun.}
was introduced by von Neumann~\cite{vNeumann29} in his study of the Banach-Tarski paradox~\cite{BanachTarski24}.
%This led von Neumann to initiate the study of amenable groups, i.e.\ groups that admit a finitely additive, left invariant, probability measure defined on all of its subsets.
Since then, amenable groups gained a lot of interest and many interesting examples of amenable groups were found; see~\cite{Juschenko22} for a recent overview in the discrete case.
However, it took until 2013 when the first simple examples of infinite, finitely generated amenable groups were provided by Juschenko and Monod~\cite[Theorem A]{JuschenkoMonod13}.
%A Cantor system $(\alpha,\mathfrak{C})$ is a homeomorphism $\alpha$ of the Cantor space $\mathfrak{C}$; it is called minimal if $\alpha$ admits no proper invariant closed subset.
More precisely, they showed that for every minimal homeomorphism $\alpha$ of a Cantor set $\mathfrak{C}$, the topological full group $[[\alpha]]$, i.e.\ the group of homeomorphisms of $\mathfrak{C}$ that locally act as powers of $\alpha$, is amenable.
Before that, the simplicity of the commutator subgroup of $[[\alpha]]$ was established by Matui~\cite[Theorem 4.9]{Matui06}.
Moreover, he showed that  $[[\alpha]]'$ is finitely generated if $\alpha$ is a minimal subshift~\cite[Theorem 5.4]{Matui06}.
%Meanwhile there is an abundance of examples of finitely generated simple groups known.
%Typically, they turn out to be $2$-generated.
%In fact, according to Juschenko~\cite[???]{Juschenko22}, all known examples of
Nevertheless there are still very few constructions of infinite, finitely generated amenable simple groups.
This is demonstrated by a question of Juschenko~\cite[Question B.5]{Juschenko22}, which asks whether $2$-generated, infinite simple amenable groups exist, see also~\cite[Problem 4.7]{Juschenko-workshop}.
% where this question was discussed during a workshop on amenability at the American Institute of Mathematics.
The following result answers this question affirmatively.

\begin{theorem}[Theorem~\ref{thm:2-generated-amenable}]\label{thm:2-generated-amenable-intro}
There exists an infinite $2$-generated amenable residually finite group $G$, which admits an infinite simple quotient $Q$.
\end{theorem}

The groups $G$ and $Q$ in this theorem are the group $G_{\mathcal{A}(d,r)}$ and the head $Q_{\mathcal{A}(d,r)}$ of a telescope $\mathcal{A}(d,r)$ of alternating groups, which will serve as our running example of $B$-telescopes.
It might also be worth mentioning that, while $2$-generation in Theorem~\ref{thm:2-generated-amenable-intro} needs a bit of argumentation, finite generation of $G$ and $Q$ is a direct consequence of the definition of the $B$-telescope $\mathcal{A}(d,r)$.
%Also the simplicity is fairly easy and can be proven by adapting the proof of the CSP-proof of classical branch group to the setting of $B$-telescope.
%To proof their amenability we apply a result of Juschenko, Nekrashevych, and de la Salle.
The amenability of the groups $G_{\mathcal{A}(d,r)}$ and $Q_{\mathcal{A}(d,r)}$, as well as all other amenable groups in this article, relies on a criterion of Juschenko, Nekrashevych and de la Salle~\cite{JuschenkoNekrashevychdelaSalle16}.

\subsection*{Application 2: embeddings into simple groups}
Embedding groups into simple groups is a classical theme in group theory.
It is an easy consequence of the Baer-Schreier-Ulam theorem~\cite{Baer34} that every group $G$ embeds into a simple group $Q$.
In the case where $G$ is countable, it was shown by Gorju\v{s}kin~\cite{Gorjushkin74} that $Q$ can be chosen to be $2$-generated (see also \cite{Schupp76}).
Since then, many embedding results of specific classes of groups into simple groups with additional properties have been obtained; see e.g.~\cite{BelkZaremsky22,DarbinyanSteenbock22,Osin10} for more recent results in that spirit.
%As already indicated, the lack of examples 
Following that theme, it is natural to ask which groups embed into finitely generated, simple amenable groups.
As far as we know, a complete answer to that question seems yet to be out of reach.
%However, as a consequence of the limited constructions available to produce infinite examples of finitely generated, simple amenable groups, it seems still mysterious what kind of subgroups they have.
% which groups can be embedded into finitely generated, simple amenable groups.
An interesting result concerning that question was obtained by Matte Bon~\cite[Theorem 1.1]{Bon17}, who proved that groups generated by bounded automata embed into finitely generated, simple amenable groups.
In particular this includes many of the classical examples of branch groups such as Grigorchuk's group and Gupta-Sidki groups.
Here we prove a significantly more general result by showing that every finitely generated, residually finite amenable group $G$ embeds in a finitely generated, simple amenable group $Q$.
Of course, unless $G$ is finite, the residual finiteness of $G$ gets lost in the transition to $Q$.
But it will be an immediate consequence of our construction that $Q$ is still locally embeddable into finite groups ($\LEF$, for short) in the sense of Vershik-Gordon~\cite{VershikGordon}.

\begin{theorem}[Theorem~\ref{thm:embeddings-into-simple-groups}]\label{thm:embedding-intro}
Let $H$ be a finitely generated, residually finite (amenable) group.
There is a finitely generated, residually finite (amenable) group $G$ and an infinite, simple, (amenable) $\LEF$ group $Q$ such that
\begin{enumerate}
\item there is an embedding $\iota \colon H \rightarrow G$,
\item there is a projection $\pi \colon G \rightarrow Q$,
\item the composition $\pi \circ \iota$ is injective.
\end{enumerate}
\end{theorem}

%Here we exhibit a general method for embedding finitely generated residually finite groups into finitely generated simple groups. An infinite simple group cannot be residually finite, but we can arrange that the simple groups are still \emph{locally embeddable into finite} ($\LEF$) in the sense of Vershik-Gordon~\cite{VershikGordon}.
%
%A group is $\LEF$ if every finite subset looks like a subset of some finite group. Equivalently, a group is $\LEF$ if it is isomorphic to a subgroup of an ultraproduct of finite groups.
%
%The flexibility of $B$-telescopes is demonstrated by the following embedding result.
%
%
%\begin{theorem}[Theorem~\ref{thm:embeddings-into-simple-groups}]\label{thm:embedding-intro}
%Let $H$ be a finitely generated residually finite group.
%Then $H$ embeds into a finitely generated, $\LEF$ simple group $Q$.
%If $H$ is amenable, then $Q$ is amenable.
%\end{theorem}

The group $Q$ in the theorem is the head of a telescope $\mathcal{S}$ of alternating groups. In our construction, we will embed the group $H$ into the residually finite group $G_\mathcal{S}$ such that $H$ intersects the kernel of the projection $G_{\mathcal{S}} \to Q_\mathcal{S}$ trivially.
%
%Grigorchuk and Medynets~\cite{GrigorchukMedynets12} have proved that the topological full group $[[\alpha]]$ of a minimal Cantor system $(\alpha;\mathfrak{C})$ is $\LEF$.
%the corresponding result for full groups of Cantor minimal systems was proven by Grigorchuk and Medynets~\cite{GrigorchukMedynets12}.
As a consequence of Theorem~\ref{thm:embedding-intro}, we obtain an abundance of uncountable families of non-isomorphic, finitely generated simple amenable groups.
Indeed, this follows from the fact that a countable group contains only countably many (isomorphism classes of) finitely generated groups, while many uncountable families of finitely generated, residually finite groups are known.
The existence of uncountably many isomorphism classes of finitely generated simple amenable groups was already established by Juschenko and Monod~\cite[Corollary B]{JuschenkoMonod13} as a consequence of results of Matui~\cite{Matui06}, and Giordano, Putnam, and Skau~\cite{GiordanoPutnamSkau99}.

\subsection*{Application 3: frames}
In 1937 Bernhard Neumann~\cite{Neumann37} constructed the first known family of uncountably many isomorphism classes of finitely generated groups and thereby implicitly proved the existence of groups that are finitely generated but not finitely presented.
%In 1937 Bernhard Neumann~\cite{Neumann37} proved that there are uncountably many isomorphism classes of finitely generated groups and thereby implicitly proved the existence of groups that are finitely generated but not finitely presented.
More precisely, Neumann introduced a family of groups that is parametrized by infinite subsets $S$ of the set $\N_{\geq 3,\odd}$ of odd natural numbers of cardinality at least $3$.
For each such subset $S$, he defined the group
\[
G_S \defeq \langle (\sigma_n),(\tau_n) \rangle \leq  \prod \limits_{n \in S} \Alt(n),
\]
where $\sigma_n = (1,2,3)$ and $\tau_n = (1,\ldots,n)$, and showed that $G_{S} \cong G_{T}$ if and only if $S = T$.
%More precisely, the pairwise non-isomorphic groups introduced by Neumann are parametrized by infinite subsets $S$ of the set $\N_{\geq 3,\odd}$ of odd natural numbers of cardinality at least $3$.
%For each such subset $S$ he defines the group
%\[
%G_S \defeq \langle (\sigma_n),(\tau_n) \rangle \leq  \prod \limits_{n \in S} \Alt(n),
%\]
%where $\sigma_n = (1,2,3)$ and $\tau_n = (1,\ldots,n)$.
Since then, the groups $G_S$ remained a rich source for observing interesting phenomena in group theory, see e.g.~\cite{Pyber01,Pyber04,LubotzkyPyberShalev96} for classical applications and~\cite{BieriCornulierGuyotStrebel14, Bou-Rabee16} for rather recent ones.
In~\cite{LubotzkyPyberShalev96} the groups $G_S$ were shown by Lubotzky, Pyber, and Shalev to provide exotic types of subgroup growth.
To this end, they proved that the profinite completion of $G_S$ is given by $\widehat{G_S} \cong \widehat{\Z} \times \prod \limits_{n \in S} \Alt(n)$.
%It was shown by Lubotzky, Pyber, and Shalev~\cite{LubotzkyPyberShalev96} that the profinite completion of $G_S$ is given by $\widehat{G_S} \cong \widehat{\Z} \times \prod \limits_{n \in S} \Alt(n)$.
Using a construction of a similar spirit, they also showed that every group of the form $\widehat{\Z} \times \prod \limits_{n \in S} \PSL_n(\F_q)$ with $S \subseteq \N_{\geq 2}$ arises as the profinite completion of a finitely generated group.
Another related construction was introduced by Pyber~\cite{Pyber04} in order to realize $\widehat{\Z} \times \prod \limits_{n \geq 5} \Alt(n)^{f(n)}$ as the profinite completion of a finitely generated group, where $f$ is subject to some mild growth condition.
The factor $\widehat{\Z}$ appearing in these profinite completions can essentially be traced back to the element $(\tau_n)$ in Neumann's group $G_S$, which can be mapped onto $\Z$.
Since $\Z$ has only one subgroup for each finite index, the factor $\widehat{\Z}$ does not harm in the analysis of subgroup growth.
However, it makes these examples unsuitable for the construction of groups with exotic represention growth; compare~\cite[Proposition 2]{BassLubotzkyMagidMozes02}.
%, from which it follows that the factor makes the study of representation growth impossible.
It is therefore natural to ask whether one can remove that factor.
For instance Pyber asked in~\cite{Pyber01}  whether the product $\prod \limits_{n=2}^{\infty} \Alt(2n+1)$ is the profinite completion of a finitely generated group.
According to~\cite[p.\ $2$]{KassabovNikolov06}, a first step in this direction was achieved by Kassabov, who showed that $\prod \limits_{n=5}^{\infty} \Alt(n)$ contains a finitely generated dense subgroup $G$ whose profinite completion is given by $\widehat{G} \cong \widehat{H}^6 \times \prod \limits_{n=5}^{\infty} \Alt(n)$, where $H = \SL_3(\F_p[t,t^{-1}])$.
Kassabov's construction is based on his ideas from~\cite{Kassabov07b, Kassabov07a} and was modified later by him and Nikolov~\cite{KassabovNikolov06} in a way that resolved the question of which products of finite simple groups arise as profinite completions of finitely generated groups.
They introduced the following useful notion.

\begin{definition}[see \cite{KassabovNikolov06}]\label{def:frame-intro}
Let $P = \prod \limits_{n=1}^{\infty} S_n$ be a product of finite groups.
A finitely generated subgroup $G \leq P$ is a \emph{frame for $P$} if
\begin{enumerate}
\item $\bigoplus \limits_{n=1}^{\infty} S_n$ is contained in $G$,
\item the natural map $\widehat{G} \rightarrow P$ is an isomorphism.
\end{enumerate}
\end{definition}

The main difficulty in the work of Kassabov and Nikolov was to establish the existence of some frame in some infinite product of finite alternating groups.
They proved that for every odd prime $p$ there is a $10$-generated frame for $\prod \limits_{n=3}^{\infty} \Alt\big(\frac{p^{3n}-1}{p-1}\big)$, see \cite[Thm.~3.3]{KassabovNikolov06}, which, to the best of our knowledge, was the only known direct construction of a frame so far.
%in some infinite product of finite groups.
Despite of the difficulty to obtain such a frame, it turns out that frames are surprisingly stable and can be merged to provide frames in other products of finite groups, see~\cite[Lemma 2.2]{KassabovNikolov06}.
Using this, Kassabov and Nikolov~\cite[Theorem 1.4]{KassabovNikolov06} deduced that
%By gluing together $2$ such frames, see Lemma~\ref{lem:gluing}, they obtained the second point.
%They deduced that
a product $P = \prod_{n \in \N} S_n$ of non-abelian finite simple groups admits a frame if
%and only if
%[label = $\bullet$]
\begin{enumerate}
\item $P$ is \emph{topologically finitely generated}, i.e.\ contains a finitely generated dense subgroup, and
\item $(S_n)_n$ has \emph{essentially unbounded rank}, which means that each alternating group $\Alt(k)$ embeds into all but finitely many $S_n$.
%which means, for every $k \geq 5$ the group $\Alt(k)$ embeds into all but finitely many $S_n$.
\end{enumerate}

In general, there is no need for a topologically finitely generated profinite group to arise as the profinite completion of a residually finite, finitely generated group.
For instance, the profinite group $\prod _{p \in \mathbb{P}} \PSL_n(\F_p)$ is a counterexample; see~\cite[Section 6.2]{Segal07}.
An interesting related phenomenon is that
%the reverse inequality is not true in general, i.e.\ that
a finitely generated, residually finite group may need more generators than each of its finite quotients.
In fact it was shown by Noskov~\cite{Noskov83} that for a finitely generated, residually finite group $G$, the difference between the minimal cardinality of a generating set of $G$, denoted by $d(G)$, and the minimal cardinality of a generating set of a dense subgroup in $\widehat{G}$, denoted by $\delta(\widehat{G})$, can get arbitrarily large.
Finitely presented examples were given by Wise~\cite{Wise03}.
% and $\delta(\widehat{G}) = 3$.
%In the case where $P = \widehat{G}$ for some finitely generated group $G$, then we clearly have $d(\widehat{G}) \leq d(G)$.
In the above result of Kassabov and Nikolov, the frame in $P$ can be chosen to be $22(\delta(P)+1)$-generated~\cite[p. 3]{KassabovNikolov06}, where the bound results from their $10$-generated frame in $\prod \limits_{n=3}^{\infty} \Alt\big(\frac{p^{3n}-1}{p-1}\big)$.
In contrast to that, this product clearly contains $2$-generated, dense subgroups.
%For example, every product of distinct finite simple groups is topologically $2$-generated.
Regarding this, Kassabov and Nikolov asked whether $\delta(P)$-generated frames exist \cite[Question~1.6]{KassabovNikolov06}.
Here we give an affirmative answer for some infinite products:

\begin{theorem}[Theorem~\ref{thm:frame-minimal-Alt} and Theorem~\ref{thm:frame-minimal-SL}]\label{thm:frames-intro}
\mbox{ }
\begin{enumerate}
\item Let $d \geq 5$. Then $\prod_{j=1}^\infty \Alt(d^j)$ admits a $2$-generated frame.
\item Let $\F_q$ be a finite field and let $d \geq 4$. Then $\prod_{j=1}^\infty \SL_{d^j}(\F_q)$ admits a $2$-generated frame.
\end{enumerate}
\end{theorem}
The frames in this theorem are the groups $G_{\mathcal{A}(d,r)}$ and $G_{\mathcal{SL}_d(\F_q)}$ of telescopes of alternating, respectively special linear groups.
Telescopes provide frames under rather mild assumptions, as will be explained in Section~\ref{sec:profinite-completion}.
Readers familiar with the theory of branch groups will observe that the proof resembles the proof of the congruence subgroup property of some classical perfect branch groups.
% that act on the $d$-adic rooted tree.
% is given by the iterated permutational wreath product $\wr_X^{\infty} \Alt(X)$, where $X$ is a set with $d$ elements.
By combining this proof with the above mentioned gluing part of the proof of Kassabov and Nikolov,
%that the pure existence of a frame as in Theorem~\ref{thm:frames-intro}
%the frames in Theorem~\ref{thm:frames-intro} can be used to
one obtains a rather elementary proof of their result that every topologically finitely generated product $P = \prod \limits_{n \in \N} S_n$ admits a frame if $(S_n)_n$ has essentially unbounded rank.
%Combined with the above mentioned insight of Kassabov and Nikolov that the pure existence of such a frame can be used to deduce that every topologically finitely generated product $P = \prod \limits_{n \in \N} S_n$ admits a frame if $(S_n)_n$ has essentially unbounded rank, this provides a rather elementary proof of that fact.
%
%Along the lines of \cite{KassabovNikolov06} this gives a new proof that a product $P \defeq \prod \limits_{n \in \N} S_n$ of finite simple groups is the profinite completion of a finitely generated group if and only if $P$ is topologically finitely generated and $(S_n)_{n \in \N}$ has essentially unbounded rank.

Regarding the second point of Theorem~\ref{thm:frames-intro}, we will see in Section~\ref{sec:the-alg-construction} and~\ref{sec:profinite-completion} that the underlying telescope $\mathcal{E}_d(R)$ is still reasonable when we replace $\F_q$ with an arbitrary associative unital ring $R$.
Then the profinite completion of the associated group $G_{\mathcal{E}_d(R)}$ is given by
$\widehat{G_{\mathcal{E}_d(R)}} \cong \prod_{j=1}^\infty \widehat{\E_{d^j}(R)}$, where $\E_n(R)$ denotes the subgroup of $\GL_n(R)$ generated by elementary matrices. If $\E_d(R)$ is finitely generated, then this can be used to show that the minimal number of generators needed to generated $\bigoplus \limits_{n=1}^{m} \E_{d^n}(R)$ is bounded above by some constant only depending on $d$ and $R$ and that $\E_{d^m}(R)$ does not admit quotients of a given finite cardinality when $m$ is chosen large enough.
%Althought this is probably known, we would like to mention that the latter implies that $\E_{m}(R)$ does not admit small finite quotients when $m$ is large enough.
%see that if $\E_{d}(R)$ is finitely generated for some $d \geq 4$, then the value $d(\bigoplus \limits_{n=1}^{m} \E_{d^n}(R))$ is uniformly bounded above by some constant only depending on $d$ and $R$.

\subsection*{Application 4: amenable-($\tau$) Grothendieck pairs}
It is a classical question in group theory of what can be said about a finitely generated, residually finite group $G$ if one knows all of its finite quotients, or equivalently, if one knows its profinite completion $\widehat{G}$ (compare \cite{DixonFormanekPolandRibes82}).
It is well-known that in general $G$ is not determined by its profinite completion.
This is not even true for virtually cyclic groups; see~\cite{Baumslag74}.
Even though there are some interesting examples of groups known to be profinitely rigid (see e.g.~\cite{BMRS20}),
this seems to be an exceptional behavior.
Nevertheless some properties -- called \emph{profinite} properties -- are detected by the profinite completion. Simple examples of profinite properties are being abelian, nilpotent or solvable.
However, many other properties are known to be not profinite; e.g. finiteness properties \cite{Lubotzky14} and higher $\ell^2$-Betti numbers \cite{KammeyerSauer}.
But how different can groups be that share the same profinite completion?
Here we investigate this question for two fairly oppositional properties in group theory: being amenable and having property $(\tau)$.
The latter property was introduced by Lubotzky and Zimmer~\cite{LubotzkyZimmer89} as a slightly weaker variant of Kazhdan's property $(\T)$.
Neither property $(\tau)$ nor amenability are profinite properties. In the case of property $(\tau)$ this was proven by Kassabov~\cite{Kassabov08},
%http://helper.ipam.ucla.edu/publications/eg2008/eg2008_7560.pdf
for amenability this is a result of the authors~\cite{KionkeSchesler21}.
%Among the non-isomorphic pairs of groups that have the same profinite completion, 
As a common strengthening of these results we will combine them in a \emph{Grothendieck pair}.
% consisting of an amenable group and a group with property $(\tau)$.
Recall that a pair $(G,H)$ of finitely generated, residually finite groups is called a Grothendieck pair if $G$ is a proper subgroup of $H$ and the inclusion $\iota \colon G \rightarrow H$ induces an isomorphism $\widehat{\iota} \colon \widehat{G} \rightarrow \widehat{H}$.
In the case where $G$ and $H$ are finitely presented, it was asked by Grothendieck~\cite{Grothendieck70} in 1970 whether such pairs exist.
Even in the case of finitely generated groups it took until $1986$ when the first examples of Grothendieck pairs were constructed by Platonov and Tavgen \cite{PlatonovTavgen}.
The finitely presented case was settled almost 20 years later by Bridson and Grunewald~\cite{BridsonGrunewald}.
By combining strong results of Ershov, Jaikin-Zapirain, and Kassabov~\cite{ErshovJaikinKassabov17} on groups satisfying property $(\tau)$ with the properties of our telescope groups $G_{\mathcal{S}}$, we will show the following.
\begin{theorem}[Theorem~\ref{thm:amenable-vs-tau}]\label{thm:amenable-vs-tau-intro}
There exists a Grothendieck pair $(G,H)$ consisting of an amenable group $G$ and a group $H$ with property $(\tau)$.
\end{theorem}
Again, the group $G$ in this theorem is the group $G_{\mathcal{A}(d,r)}$ of the telescope $\mathcal{A}(d,r)$ that we have mentioned already.
% while the group $H$  of alternating groups that, which will serve as our running example of $B$-telescopes.

\subsection*{Organization of the article}
In Section~\ref{sec:the-alg-construction}
we give the definition of a telescope and introduce the notion of flexibility. We also present our main examples: the telescopes $\mathcal{A}(d,r)$ of alternating groups and the telescopes $\mathcal{SL}_d(F)$ of special linear groups.
Section~\ref{sec:profinite-completion} contains a number of useful results on the structure of flexible $B$-telescopes for a perfect group $B$.
In particular, we will deduce our results concerning frames in this section.
Section~\ref{sec:normal-subs} is concerned with normal and subnormal subgroups of the group $G_\mathcal{S}$ of a telescope $\mathcal{S}$.
We will see that under suitable assumptions $G_\mathcal{S}$ is a $t$-group, i.e., every subnormal subgroup is normal. As an application we will show that all finitely generated subnormal subgroups of $G_{\mathcal{A}(d,r)}$ are direct factors.
Still concerned with direct factors, Section~\ref{sec:direct-factors-prof-completions} contains an easy construction of a finitely generated, residually finite group in which direct factors cannot be recognized in profinite completions.
In the course of that we provide an amenable counterexample to a question of Goldstein and Guralnick.
In Section~\ref{sec:actions} we briefly introduce actions of telescopes, which we consider as a useful way to think about them.
However, the results in this section will only be used to give an alternative proof of the simplicity of the head $Q_{\mathcal{A}(d,r)}$.
That proof, as well as a short self-contained proof of the simplicity of $Q_{\mathcal{A}(d,r)}$, will be given in Section \ref{sec:simplicity-alt}.
In Section~\ref{sec:simplicity-psl} we will show that the head of a projective variant of the telescopes $\mathcal{SL}_d(\F_q)$ is simple as well.
The necessary background and the proof of the embedding result Theorem~\ref{thm:embedding-intro} is contained in Section~\ref{sec:embedding-groups}.
In Section~\ref{sec:grothendieck-pair} we observe that the group $G_{\mathcal{A}(d,r)}$ can be combined with known groups with property $(\tau)$ in order to obtain the Grothendieck pair in Theorem~\ref{thm:amenable-vs-tau-intro}.
Finally, in Section~\ref{sec:questions} we introduce some open problems involving $B$-telescopes.

\section{Definitions and Examples}\label{sec:the-alg-construction}
\subsection{Notation}
Let $G$ be a group. Throughout we will use the following notation. Let $g,h \in G$. We write $g^h = h^{-1}gh$ for conjugation and the commutator of $g,h$ is defined to be $[g,h] = ghg^{-1}h^{-1}$.
\subsection{Definition of $B$-telescopes}
\begin{definition}\label{def:telescope}
Let $B$ be a group.
A $B$-\emph{telescope}
$\mathcal{S}=((\Omega_i)_{i\in \N},  (\phi_i)_{i \geq 2})$ consists of
a directed system $((\Omega_i)_{i\in \N},\iota_{i,j})$ of groups $\Omega_i$ and homomorphisms $\iota_{i,j} \colon \Omega_i \rightarrow \Omega_j$ for all $i < j$
%for all $i < j$ such that $\iota_{j,k} \circ \iota_{i,j} = \iota_{i,k}$ for all
%\[ \Omega_1 \stackrel{\iota_2^1}{\hookrightarrow} \Omega_2 \stackrel{\iota_3^2}{\hookrightarrow} \Omega_3 \stackrel{\iota_4^3}{\hookrightarrow} \dots \]
%of groups
together with homomorphisms $\phi_i\colon B \to \Omega_i$ for each $i \geq 2$ whose image $B_i\subseteq \Omega_i$ satisfies
\begin{equation}\label{eq:generation}
     \Omega_i = \langle \{ hB_i h^{-1} \mid h \in \iota_{i-1,i}(\Omega_{i-1}) \} \rangle
 \end{equation}
and
\begin{equation}\label{eq:commutator}
[\iota_{i,j}(B_{i}),B_j] = 1
\end{equation}
 for all $2 \leq i < j$.
 For convenience we often write $\iota_{i,i}$ to denote the identity. 
\end{definition}
\begin{figure}[h]
\begin{tikzpicture}
\usetikzlibrary{arrows}
\usetikzlibrary{shapes}
\foreach \y in {1,2,3,4,5}
{
   \draw[fill=gray, opacity=0.2] (2*\y-2,0) ellipse (14pt+\y pt and 14pt + 7.5*\y pt);
 \node[align=center] (O\y) at (2*\y -2 ,1 + 0.2*\y) {$\Omega_\y$};
 }
 \node[align=center] (O6) at (9 ,1 + 1.1) {$\dots$};
 \foreach \y/\z in {1/2,2/3,3/4,4/5,5/6}
    \draw[->] (O\y) to node[above]{$\iota_{\y,\z}$} (O\z);
    
 \draw[fill=brown, opacity=0.2] (2,-3.5) circle (8pt);
\node[align=center] (B) at (2,-3.5) {$B$};

 \foreach \x in {2,4,6,8}
 {
 \draw[dotted, brown,thick] (\x+ 0.27,1.4 - \x * 0.35) -- (9, 1.4 - \x * 0.35);
   \draw[fill=brown, opacity=0.2] (\x,0.7) circle (8pt);
   }

 \foreach \x in {4,6,8}
   \draw[fill=brown, opacity=0.2] (\x,0) circle (8pt);

   \foreach \x in {6,8}
   \draw[fill=brown, opacity=0.2] (\x,-0.7) circle (8pt);
   
\draw[fill=brown, opacity=0.2] (8,-1.4) circle (8pt);

\foreach \x in {2,3,4,5}
{
     \node[align=center] (B\x) at (2*\x-2,2.1-0.7*\x) {$B_\x$};
     \draw[->, out=115-12*\x , in=270-16*\x] (B) to node[shift={(-0.5+0.\x,0.\x-0.2)}]{$\phi_\x$} (B\x);
}
\end{tikzpicture}
\caption{Diagram of a $B$-telescope}
\end{figure}
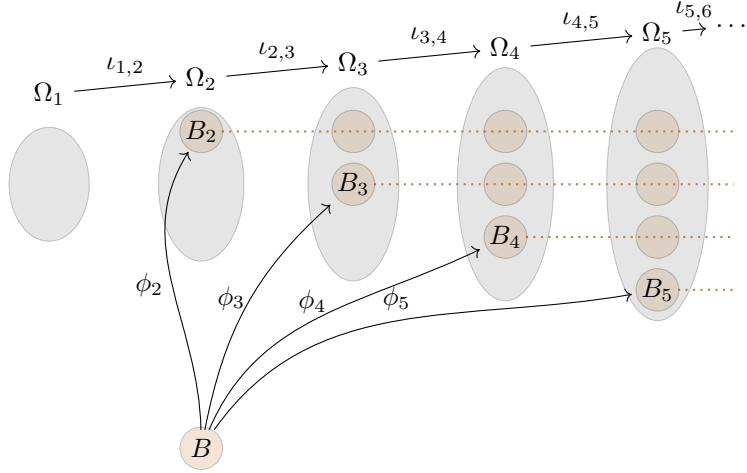
Let $\mathcal{S}=((\Omega_i)_{i\in \N},  (\phi_i)_{i \geq 2})$ be a $B$-telescope.
The groups we are interested in are subgroups of the direct product $P_{\mathcal{S}} = \prod_{i=1}^\infty \Omega_i$.
The canonical projection from $P_{\mathcal{S}}$ to $\Omega_i$ will be denoted by $\pi_i$.
For each $i \in \N$ we define 
\[
	\Delta_i \colon \Omega_i \to P_{\mathcal{S}} \quad \text{ with } \quad \omega \mapsto (\underbrace{1,\dots,1}_{i-1},\omega,\iota_{i,i+1}(\omega),\iota_{i,i+2}(\omega), \dots).
\]
We note that $\Delta_i$ is an injective homomorphism.

For an element $g \in B$ and $n \geq 1$ we define $\tilde{g}^{[n]} \in P_{\mathcal{S}}$ as the element given by $\pi_i(\tilde{g}^{[n]}) = 1$ for $i \leq n$ and by $\pi_{i}(\tilde{g}^{[n]}) = \prod \limits_{k=n+1}^{i} \iota_{k,i}(\phi_{k}(g))$ for $i > n$.
From condition \eqref{eq:commutator} we note that the terms in this product commute pairwise.
We write $\tilde{g}$ for $\tilde{g}^{[1]}$.

\begin{remark}\label{rem:tau_n-is-a-morphism}
The map $\tau_n \colon B \rightarrow P_{\mathcal{S}},\ g \mapsto \tilde{g}^{[n]}$ is a homomorphism. This directly follows from our assumption that each $\phi_{i}$ is a homomorphism and condition \eqref{eq:commutator}.\end{remark}

The image of $B$ under $\tau_n$ will be denoted by $C^{[n]}$.
The group
 \[
G_{\mathcal{S}} = \langle \Delta_1(\Omega_1), C^{[1]}\rangle \subseteq P_\mathcal{S}\]
will be called the \emph{group} of the telescope.
Here we are mostly concerned with the structure of the group $G_{\mathcal{S}}$.
However, it will be convenient to study also another group, defined from $\mathcal{S}$.
The group 
\[
\widetilde{G}_{\mathcal{S}} = \langle \{\Delta_i(\Omega_i)\mid i \geq 1\} \cup C^{[1]}\rangle \subseteq P_\mathcal{S}\]
will be called the \emph{huge group} of the telescope. We will see later that under suitable assumptions on the telescope, the group and the huge group agree.
We note that $\widetilde{G}_{\mathcal{S}}$ contains the direct sum 
$\bigoplus_{i=1}^\infty \Omega_i$ as a normal subgroup (a proof is given in Corollary~\ref{cor:first-levels} below). The quotient
\[
	Q_\mathcal{S} = \widetilde{G}_{\mathcal{S}}/\bigoplus_{i=1}^\infty \Omega_i
\]
is called the \emph{head} of the telescope.

\begin{notation}
Let $\mathcal{S} = ((\Omega_i)_{i\in \N},(\phi_i)_{i \geq 2})$ be a $B$-telescope.
For every two numbers $i,j \in \N$ with $i \leq j$ we write $B_{i,j}$, respectively $\Omega_{i,j}$, for the image of $B_i$, respectively $\Omega_i$, under $\iota_{i,j}$.
For $a \in \Omega_{i}$ and $j > i$, we define $a_{i,j} = \iota_{i,j}(a)$.
\end{notation}

We will frequently use the following simple observation, which is a direct consequence of the axioms of a $B$-telescope.

\begin{remark}\label{rem:normalform-1}
Let $\mathcal{S} = ((\Omega_i)_{i\in \N},(\phi_i)_{i \geq 2})$ be a $B$-telescope. For all $b \in B$ and $n \in \N$
\[
\tilde{b}^{[n]} = \Delta_{n+1}(\phi_{n+1}(b))\tilde{b}^{[n+1]} = \tilde{b}^{[n+1]} \Delta_{n+1}(\phi_{n+1}(b)).
\]
%Let $k \in \N$.
%Every element $g \in G_\mathcal{S}$ can be written in the form
%\[
%	g = f\Delta_k(\omega) c_1^{\delta_1} c_2^{\delta_2} \cdots c_{n}^{\delta_n} 
%\]
%for some $n \in \N$ where $\omega = \pi_k(k)$, $f\in \bigoplus_{i=1}^{k-1} \Omega_i$,
%$\delta_i \in \Delta_{k}(\Omega_k)$ and $c_i \in C^{[k]}$ for all $i \in \{1,2,\dots,n\}$.
%Here $n$ depends only on the word length of $g$ w.r.t. the generating set $\Delta_1(\Omega_1)$ and is independent of $k$.
\end{remark}

%\begin{proof}
%Equation \eqref{eq:delta-tilde} follows from \eqref{eq:commutator} and the definition of $\tilde{b}^{[n]}$.
%
%Let $g \in G_\mathcal{S}$ we prove the assertion by induction on the word length of $g$ with respect to the generating set $\Delta_1(\Omega_1) \cup C^{[1]}$.
%For the induction step assume that 
%$g = xg'$ with $x \in \Delta_1(\Omega_1) \cup C^{[1]}$.
%By induction, we may write $g' = \Delta_{k}(\omega') c_1^{\delta_1} c_2^{\delta_2} \cdots c_{n}^{\delta_n}$ with $\omega' = \pi_k(g')$ as in the statement of the lemma.
%If $x = \Delta_1(\eta) \in \Delta_1(\Omega_1)$, then $\Delta_1(\eta)f\Delta_k(\omega') = f'\Delta_k(\eta\omega)$ for some $f' \in \bigoplus_{i=1}^{k-1}$ and this proves the claim.
%If $x = \tilde{b}^{[1]} \in C^{[1]}$, then we write $\tilde{b}^{[1]} = f'\Delta_{k}(\eta)\tilde{b}^{[1]}$ using equation \eqref{eq:delta-tilde} and we obtain
%\[
%	g =xg' = f''\Delta_k(\eta\omega') (\tilde{b}^{[k]})^{\delta}c_1^{\delta_1} c_2^{\delta_2} \cdots c_{n}^{\delta_n}
%\]
%with $\delta = \Delta_k(\omega')$ for a suitable $f'' \in \bigoplus_{i=1}^{k-1}\Omega_i$.
%It follows from the proof that $n$ only depends on the number of letters from $C^{[1]}$.
%\end{proof}
%%%%%%%%%%%%%%%%%

\begin{definition}\label{def:flexible-telescope}
Let $\mathcal{S} = ((\Omega_i)_{i\in \N},  (\phi_i)_{i \geq 2})$ be a $B$-telescope.
We say that $\mathcal{S}$ is \emph{flexible}, if for all $i \in \N$ there is an element $\alpha_i \in \Omega_i$ such that the following hold
\begin{enumerate}[label=(F\arabic*)]
\item\label{it:a-b} $[\alpha_{i,i+1},B_{i+1}] = 1$ for all $i \geq 1$,
\item\label{it:ba-b} $[B_{k,m}^{\alpha_{i,m}},B_{\ell,m}] = 1$ for all $k,\ell \geq i+2$, and $m \geq k,\ell$,
\item\label{it:ba-ba} $[B_{k,m}^{\alpha_{i,m}}, B_{\ell,m}^{\alpha_{j,m}}] = 1$ for all $k \geq i+2$, $\ell \geq j+2$, $i>j$ and $m \geq k,\ell$,
\end{enumerate}
where $\alpha_{i,j} = \iota_{i,j}(\alpha_{i})$.
\end{definition}
\begin{remark}
If $\mathcal{S}$ is flexible, then Condition \ref{it:ba-b} also holds for $\ell = i+1$ due to \ref{it:a-b} and \eqref{eq:commutator}. Similarly, condition \ref{it:ba-ba} holds for $k\geq i+1$ and $\ell \geq j+1$ using conditions \ref{it:a-b}, \ref{it:ba-b} and \eqref{eq:commutator}.
\end{remark}
\begin{definition}
Let $\mathcal{S} = ((\Omega_i)_{i \in \N},(\phi_i)_{i \geq 2})$ be a $B$-telescope.
For each $n \in \N_0$ we consider the system $\mathcal{S}_{+n} \defeq ((\Omega_{i+n})_{i \in \N},(\phi_{i+n})_{i \geq 2})$, which will be called the \emph{$n$-fold shifted $B$-telescope of $\mathcal{S}$}.
\end{definition}
From the axioms it follows immediately that $\mathcal{S}_{+n}$ is indeed a $B$-telescope and that $\mathcal{S}_{+n}$ is flexible, if $\mathcal{S}$ is flexible. 
\begin{lemma}\label{lem:quotient-telescope}
Let $\mathcal{S} = ((\Omega_i)_{i \in \N},(\phi_i)_{i \geq 2})$ be a $B$-telescope. Suppose that  $M_i$ is a normal subgroup of $\Omega_i$ and that $\iota_{i,j}(M_i) \subseteq M_j$. Then $\mathcal{T} = ((\Omega_i/M_i)_{i \in \N},(\overline{\phi}_i)_{i \geq 2})$ is a $B$-telescope for $\overline{\phi}_n(b) = \phi_n(b)M_i$.
If $\mathcal{S}$ is flexible, then $\mathcal{T}$ is flexible.
\end{lemma}
\begin{proof}
By the assumption $\iota_{i,j}(M_i) \subseteq M_j$, the transition maps $\iota_{i,j}$ induce well-defined maps $\overline{\iota}_{i,j}\colon \Omega_i/M_i \to \Omega_j/M_j$.
Clearly, if $\Omega_i$ is generated by the $\Omega_{i-1,i}$-conjugates of $B_i = \phi_i(B)$, then the quotient $\Omega_i/M_i$ is generated by the $(\Omega_{i-1,i} / M_{i-1})$-conjugates of $\overline{\phi}_i(B)$.
The remaining axioms of $B$-telescopes and flexibility rely on commutation relations, which clearly survive in the quotients $\Omega_i/M_i$.
\end{proof}
%
%\begin{definition}
%Let $\mathcal{S}=((\Omega_i)_{i\in \N},  (\phi_i)_{i \geq 2})$ be a $B$-telescope.
%Given $i_0 \in \N$, we call a family of subgroups $H_j \subseteq \Omega_j$ for $j \geq i_0$ a \emph{subtelescope} if
%%A \emph{subtelescope} $\mathcal{H}$ of $\mathcal{S}$ consists of a family of subgroups $H_j \subseteq \Omega_j$ for $j \geq i_0$ and some $i_0 \in \N$, such that
%\begin{enumerate}
%\item $B_j \subseteq H_j$ for all $j \geq i_0+1$,
%\item $\iota_{i,j}(H_i) \subseteq H_j$ for all $j > i \geq i_0$,
%\end{enumerate}
%and $\mathcal{H} =((H_i)_{i \geq i_0}, (\phi_{i})_{i \geq i_0+1})$ is a $B$-telescope.
%\end{definition}

%%%% E X A M P L E S %%%%%
\subsection{Elementary examples}
We illustrate the definition with a number of examples. We begin with almost trivial examples, which are useful to understand some of the assumptions appearing later.
\begin{example}[$\{1\}$-telescopes]
Assume that $B = \{1\}$ is trivial.
Consider a $\{1\}$-telescope $\mathcal{S} =((\Omega_i)_{i\in \N},  (\phi_i)_{i \geq 2})$.
Then condition \eqref{eq:generation} 
implies that $\Omega_k = \{1\}$ for all $k \geq 2$. In particular, a $\{1\}$-telescope is determined uniquely by the group $\Omega_1$. Moreover, the group of $\mathcal{S}$ is (isomorphic to) $\Omega_1$.
Clearly, a $\{1\}$-telescope is always flexible with $\alpha_i = 1$ for all $i$.
\end{example}
\begin{example}[$B$-telescopes of length $2$]
Let $\Omega_2$ be a group with two subgroups $\Omega_1$ and $B$ such that the $\Omega_1$-conjugates of $B$ generate $\Omega_2$. Define $\Omega_k = \{1\}$ for all $k \geq 3$. Let $\iota_{1,2}\colon \Omega_1 \to \Omega_2$ and $\phi_2\colon B \to \Omega_2$ denote the inclusion maps. If we choose $\iota_{i,j}$ and $\phi_j$ to be trivial for all $j \geq 3$, then  $\mathcal{S} =((\Omega_i)_{i\in \N},  (\phi_i)_{i \geq 2})$ is a $B$-telescope. 
It is flexible with $\alpha_i = 1$ for all $i$. 
The group $G_\mathcal{S}$ is the subgroup of $\Omega_1\times \Omega_2$ defined by 
\[
  G_\mathcal{S} = \langle \{(g,g) \mid g \in \Omega_1\} \cup \{(1,b) \mid b \in B\}\rangle \subseteq \Omega_1\times \Omega_2.
\] 
Conjugating elements of the form $(1,b)$ with elements of the form $(g,g)$ one can apply condition \eqref{eq:generation} to see that $G_{\mathcal{S}} = \Omega_1 \times \Omega_2$.
\end{example}
\begin{example}[A $B$-telescope for abelian $B$]
Let $B$ be an abelian group.
Define $\Omega_i = B$ for all $i$. We set $\iota_{i,j}$ and $\phi_j$ to be the identity maps for all $1 \leq i < j$.
Then $\mathcal{S} =((\Omega_i)_{i\in \N},  (\phi_i)_{i \geq 2})$ is a flexible $B$-telescope with $\alpha_i = 1$ for all $i$.
The group $G_\mathcal{S}$ is isomorphic to $B \times B$, since it is the subgroup of $\prod_{i=1}^\infty B$ generated by the elements $\Delta_1(\Omega_1) = \{ (b,b,b\dots) \mid b \in B\}$ and $C^{[1]} = \{(1,b,b^2,b^3,b^4\dots) \mid b \in B\}$.
\end{example}
\begin{example}[A non-flexible $B$-telescope]
Let $B$ be any group.
Define $\Omega_i = \{1\}$ for all odd $i$ and define $\Omega_i = B$ for all even $i$. The maps $\iota_{i,j}$ have to be trivial for all $i < j$. In addition, we define $\phi_j\colon B \to \Omega_j$ to be the identity maps for all even numbers $j$ and to be trivial otherwise.
Then $\mathcal{S} =((\Omega_i)_{i\in \N},  (\phi_i)_{i \geq 2})$ is a $B$-telescope.
The group $G_\mathcal{S}$ is isomorphic to $B$, since $\Delta_1(\Omega_1)$ is trivial and $C^{[1]} = \{(1,b,1,b,1,b,1,\dots) \mid b \in B\} \cong B$.
This telescope is not flexible if $B$ is non-abelian. Indeed, since $\Omega_1$ is trivial the only possible choice is $\alpha_1 = 1$. However, $B_4 = \Omega_4$ is non-abelian and we have $[B_4^{\alpha_{1,4}},B_4] \neq 1$; this contradicts \ref{it:ba-b}.
\end{example}
\subsection{A family of $\Alt(rd)$-telescopes}\label{ex:Alt}
We will now introduce the main example motivating our definition of $B$-telescopes. It will serve as a running example and will be used to in our main applications. We encourage the reader to familiarize with this example. 

\smallskip

\noindent We fix an integer $d \geq 5$ and an integer $2 \leq r \leq d-3$. Our aim is to construct a flexible $\Alt(rd)$-telescope $\mathcal{A}(d,r) = ((\Omega_{\ell})_{\ell \in \N},(\phi_{\ell})_{\ell \geq 2})$, where $\Omega_{\ell} \cong \Alt(d^\ell)$ is a finite alternating group.
To this end, we fix a $d$-element set $X \defeq \{x_1,\ldots,x_d\}$, which we will think of as an alphabet.
Let $X^{\ast}$ denote the set of all words over $X$ including the empty word $\emptyset$.
We will identify the set of words of length $\ell \in \N_0$ over $X$ with the direct product $X^{\ell}$.
Let $\Omega_{\ell} = \Alt(X^{\ell})$ denote the alternating group acting on $X^{\ell}$.
For every $k \geq \ell$ we consider the embedding $\iota_{\ell,k} \colon \Omega_{\ell} \rightarrow \Omega_{k}$ given by
\[
\iota_{\ell,k}(\sigma)(vw) = \sigma(v)w
\]
for all $v \in X^{\ell}$ and $w \in X^{k-\ell}$.
Observe that $\iota_{\ell,k}(\sigma)$ is indeed an even permutation for every $\sigma \in \Omega_{\ell}$.
Let $B \defeq \Alt(\{x_1,\ldots,x_r\} \times X)$, which we think of as the subgroup of $\Omega_2$ fixing all words $x_ix_j$ with $i > r$.
For each $\ell \geq 2$, we define a homomorphism $\phi_{\ell} \colon B \rightarrow \Omega_{\ell}$ as follows.
Given $\sigma \in B$, $v \in X^{\ell-2}$ and $x,y \in X$, let
\[
\phi_{\ell}(\sigma)(vxy) =
\begin{cases}
v\sigma(xy),& \text{if } v = x_d^{\ell-2}\\
vxy,& \text{otherwise.}
\end{cases}
\]
Following Definition~\ref{def:telescope}, we write $B_{\ell}$ for the image of $B$ under $\phi_{\ell}$.

\begin{remark}\label{rem:B_i-is-a-certain-alternating-group}
Note that $B_{\ell}$ can be described as the alternating group on the set of all words of length $\ell$ over $X$ that start with $x_d^{\ell-2} x_i$ for some $i \leq r$.
In other words we have
\begin{equation}\label{eq:B_i-is-a-certain-alternating-group}
B_{\ell} = \Alt(\{x_d^{\ell-2}\} \times \{x_1,\ldots,x_r\} \times X).
\end{equation}
\end{remark}

We will show that $\mathcal{A}(d,r)=((\Omega_{\ell})_{\ell \in \N}, (\phi_{\ell})_{\ell \geq 2})$ is a flexible $B$-telescope.
As usual, we will write $B_{\ell,k}$, respectively $\Omega_{\ell,k}$, to denote the image of $B_{\ell}$, respectively $\Omega_{\ell}$, under $\iota_{\ell,k}$. To establish the various commutator relations, we will use the fact that permutations with disjoint support commute.

\begin{definition}\label{def:support}
Given a set $X$ and a subset $S \subseteq \Sym(X)$, we define the \emph{support} of $S$ in $X$ by $\supp_X(S) = \Set{x \in X}{\sigma(x) \neq x \text{ for some } \sigma \in S}$.
\end{definition}

For the flexibility of  $\mathcal{A}(d,r)$ we need elements $\alpha_i \in \Omega_i$ as in Definition~\ref{def:flexible-telescope}.
We define for each $\ell \in \N$ the element $\alpha_{\ell} \in \Omega_{\ell}$ as the $3$-cycle
\[
\alpha_{\ell} \defeq (x_d^{\ell-1}x_{d-2},x_d^{\ell-1}x_{d-1},x_d^{\ell}) \in \Omega_{\ell}.
\]
Recall that for $k \geq \ell$, we denote the image of $\alpha_{\ell}$ under $\iota_{\ell,k}$ by $\alpha_{\ell,k}$.

\begin{lemma}\label{lem:supports}
Given $i,j,k \in \N$, the following hold.
\begin{enumerate}
\item If $j \geq i \geq 2$, then $\supp_{X^j}(B_{i,j}) = \{x_d^{i-2}\} \times \{x_1,\ldots,x_r\} \times X^{j+1-i}$.
\item If $j \geq i$, then $\supp_{X^j}(\alpha_{i,j}) = \{x_d^{i-1}\} \times \{x_{d-2},x_{d-1},x_d\} \times X^{j-i}$.
\item If $k \geq j \geq i+2$, then
\[
\supp_{X^k}(B_{j,k}^{\alpha_{i,k}})
= \{x_d^{i-1}x_{d-1}x_d^{j-2-i}\} \times \{x_1,\ldots,x_r\} \times X^{k+1-j}.
\]
\end{enumerate}
\end{lemma}
\begin{proof}
The first two claims directly follow from the definitions of $B_{i,j}$ and $\alpha_{i,j}$.
By applying these two claims we obtain
\begin{align*}
\supp_{X^{k}}(B_{j,k}^{\alpha_{i,k}})
& = \alpha_{i,k}^{-1}(\supp_{X^{k}}(B_{j,k})) \\
& = \alpha_{i,k}^{-1}(\{x_d^{j-2}\} \times \{x_1,\ldots,x_r\} \times X^{k+1-j}) \\
& = \{x_d^{i-1}x_{d-1}x_d^{j-2-i}\} \times \{x_1,\ldots,x_r\} \times X^{k+1-j},
\end{align*}
which gives us the last claim.
\end{proof}
%%%%%%
\begin{proposition}\label{prop:B-telescope}
$\mathcal{A}(d,r)$ defined above is a flexible $B$-telescope.
\end{proposition}
\begin{proof}
Let $\ell \geq 2$.
We begin by to verifying~\eqref{eq:generation}, i.e.\ that $\Omega_{\ell} = \langle \{ \sigma B_{\ell} \sigma^{-1} \mid \sigma \in \Omega_{\ell-1,\ell} \} \rangle$.
In view of Remark~\ref{rem:B_i-is-a-certain-alternating-group} we have
\[
\sigma B_{\ell} \sigma^{-1}
= \Alt(\sigma(\{x_d^{\ell-2}\} \times \{x_1,\ldots,x_r\} \times X)).
\]
In the case where $\sigma = \iota_{\ell-1,\ell}(\sigma')$ for some $\sigma' \in \Omega_{\ell-1}$ this gives us
\begin{equation}\label{eq:conjugating-with-Omega-i-1}
\sigma B_{\ell} \sigma^{-1}
= \Alt(\sigma'(\{x_d^{\ell-2}\} \times \{x_1,\ldots,x_r\}) \times X).
\end{equation}
Let $v_1,\ldots,v_{d^{\ell-1}}$, respectively $w_1,\ldots,w_{d^{\ell}}$, be the lexicographic enumeration of $X^{\ell-1}$, respectively $X^{\ell}$.
%, corresponding to some order on $X$.
Then for every $v_k \in X^{\ell-1}$ we have
\begin{equation}\label{eq:consecutive}
\{v_k\} \times X = \{w_{j+1},\ldots,w_{j+d}\}
\end{equation}
for some $j$.
From our assumption that $d \geq 5$ it follows that $\Omega_{\ell-1}$ acts $2$-transitively on $X^{\ell-1}$.
Thus for every $1 \leq k < d^{\ell-1}$ there is some $\sigma_k \in \Omega_{\ell-1}$ with $\sigma_k(x_d^{\ell-2}x_1) = v_k$ and $\sigma_k(x_d^{\ell-2}x_2) = v_{k+1}$.
By setting $M_k \defeq \{v_k,v_{k+1}\} \times X$, it follows from~\eqref{eq:conjugating-with-Omega-i-1} that
\begin{equation}\label{eq:consecutive-2}
\Alt(\{w_{j+1},\ldots,w_{j+2d}\}) = \Alt(M_k) \leq \sigma_k B_{\ell} \sigma_k^{-1}
\end{equation}
for an appropriate $j$.
In view of~\eqref{eq:consecutive}, each intersection $M_k \cap M_{k+1} = \{v_{k+1}\} \times X$ consists of $d$ consecutive words in $X^{\ell}$.
Together with~\eqref{eq:consecutive-2} and our assumption that $2d \geq 10$, this tells us that every $3$ consecutive words $w_j,w_{j+1},w_{j+2}$ lie in some $M_k$.
Thus every $3$-cycle of the form $(w_j,w_{j+1},w_{j+2})$ is contained in some $\Alt(M_k)$.
Together these $3$-cycles are well-known to generate $\Alt(\Omega_{\ell})$, see e.g.~\cite[Lemma 5.1]{Matui06}.

To establish condition \eqref{eq:commutator}, we fix two integers $i,j$ with $2 \leq i < j$.
Then Lemma~\ref{lem:supports} gives us
\[
\supp_{X^j}(B_{i,j}) = \{x_d^{i-2}\} \times \{x_1,\ldots,x_r\} \times X^{j+1-i}
\]
and
\[
\supp_{X^j}(B_j) = \{x_d^{j-2}\} \times \{x_1,\ldots,x_r\} \times X.
\]
Since $r \leq d-3<d$, it follows that $\supp_{X^j}(B_{i,j})$ and $\supp_{X^j}(B_j)$ and disjoint; we conclude
$[B_{i,j},B_j] = 1$ in $\Omega_j$.

Finally, we verify the flexibility conditions \ref{it:a-b}, \ref{it:ba-b} and \ref{it:ba-ba} from Definition~\ref{def:flexible-telescope}.
Suppose that $i \geq 1$.
Then another application of Lemma~\ref{lem:supports} gives us
\[
\supp_{X^{i+1}}(\alpha_{i,i+1}) = \{x_d^{i-1}\} \times \{x_{d-2},x_{d-1},x_d\} \times X
\]
and
\[
\supp_{X^{i+1}}(B_{i+1}) = \{x_d^{i-1}\} \times \{x_1,\ldots,x_r\} \times X,
\]
which are disjoint using $r \leq d-3$.
As a consequence, we see that $\alpha_{i,i+1}$ commutes with $B_{i+1}$ and we obtain \ref{it:a-b}.
In order to verify  \ref{it:ba-b}, let $\ell \geq 2$, $j \geq i+2$, and $k \geq j,\ell$.
In this case Lemma~\ref{lem:supports} gives us
\[
\supp_{X^k}(B_{j,k}^{\alpha_{i,k}})
= \{x_d^{i-1}x_{d-1}x_d^{j-2-i}\} \times \{x_1,\ldots,x_r\} \times X^{k+1-j}
\]
and
\[
\supp_{X^k}(B_{\ell,k}) = \{x_d^{\ell-2}\} \times \{x_1,\ldots,x_r\} \times X^{k+1-\ell},
\]
from which we deduce that $\supp_{X^k}(B_{j,k}^{\alpha_{i,k}})$ and $\supp_{X^k}(B_{\ell,k})$ are disjoint.
Hence $B_{j,k}^{\alpha_{i,k}}$ commutes with $B_{\ell,k}$.

To verify \ref{it:ba-ba}, let $j > i$, $k \geq i+2$, $\ell \geq j+2$, and let $m \geq k,\ell$ be given.
Lemma~\ref{lem:supports} tells us that
\[
\supp_{X^m}(B_{k,m}^{\alpha_{i,m}})
= \{x_d^{i-1}x_{d-1}x_d^{k-2-i}\} \times \{x_1,\ldots,x_r\} \times X^{m+1-k}
\]
and
\[
\supp_{X^m}(B_{\ell,m}^{\alpha_{j,m}})
= \{x_d^{j-1}x_{d-1}x_d^{\ell-2-j}\} \times \{x_1,\ldots,x_r\} \times X^{m+1-\ell}.
\]
Since $i < j$, this implies that the supports of $B_{k,m}^{\alpha_{i,m}}$ and $B_{\ell,m}^{\alpha_{j,m}}$ are disjoint in $X^m$.
Thus $B_{k,m}^{\alpha_{i,m}}$ commutes with $B_{\ell,m}^{\alpha_{j,m}}$, which completes the proof.
\end{proof}
\begin{remark}
If one considers the sets $X^\ell$ as level sets of a regular rooted tree, then 
the construction of $\mathcal{A}(d,r)$ is reminiscent of the construction of spinal groups acting on rooted trees. The automorphisms $\Delta_1(\Omega_1)$ act as ``rooted'' tree automorphisms permuting the trees hanging below the vertices in the first level. The elements $\tilde{b}$ appear to be directed along the spine $x_d^\infty$. These elements however are no automorphisms of the tree.
\end{remark}
\subsection{A family of $\E_{2d}(R)$-telescopes}\label{ex:elementary}
Let $R$ be an associative unital ring. For a finite set $Y$ we write $\mathrm{M}_Y(R)$ for the ring of $Y\times Y$-indexed matrices with entries in $R$.  Let $\E_Y(R) = \langle \{e_{y,z}(r) \mid y,z \in Y, r \in R\} \rangle$ denote the subgroup generated by elementary matrices in $\mathrm{GL}_Y(R)$. 
For $Y' \subseteq Y$ we
say that a matrix $A = (a_{u,v})_{u,v \in Y} \in \mathrm{GL}_{Y}(R)$  is \emph{supported in} $Y'$, if 
$a_{y,y} = 1$ for all $y \not\in Y'$ and $a_{x,y} = 0$ for $x \neq y$ if $x$ or $y$ lies in $Y \setminus Y'$.
If $A_1$ and $A_2$ are matrices supported on two disjoint subsets of $Y$, then $A_1$ and $A_2$ commute.
We will think of $\E_{Y'}(R)$ as the subgroup of $\E_Y(R)$ generated by elementary matrices supported in $Y'$.
Observe that for $y\neq z \in Y$ the matrix
\begin{equation}\label{eq:swap}
	s_{y,z} = e_{y,z}(-1)e_{z,y}(1)e_{y,z}(-1)
\end{equation}
is supported in $\{y,z\}$ and the corresponding $(2 \times 2)$-block of $s_{y,z}$ is
\[
	\begin{pmatrix} 0 & -1 \\ 1 & 0 \end{pmatrix}.
\]
The group generated by elements of the form $s_{y,z}$ is called the group of signed permutations.
In particular, the group of signed permutations acts like the symmetric group on the standard basis vectors up to sign.

\smallskip

\noindent Let $d \geq 4$ and fix a set $X=\{x_1,\dots,x_d\}$ with $d$ elements. We will think of $X^k$ as the set of words of length $k$ in the letters $x_1,\dots,x_d$. We consider the block diagonal inclusion $\iota_n\colon \mathrm{M}_{X^n}(R) \to \mathrm{M}_{X^{n+1}}(R)$ taking $A = (a_{u,v})_{u,v \in X^n}$ to the matrix $\tilde{A} = (\tilde{a}_{u',v'})_{u',v'\in X^{n+1}}$ where 
$\tilde{a}_{ux,vy}  = a_{u,v} $ if $x=y$ and $\tilde{a}_{ux,vy} = 0$ if  $x \neq y$ for all $u,v \in X^n$ and $x,y \in X$.
The sequence of groups $\Omega_n \defeq \E_{X^n}(R)$ with the block-diagonal inclusions
\[
\iota_{n} \colon \Omega_n \rightarrow \Omega_{n+1}\]
will be equipped with the structure of a $B$-telescope for $B \defeq \E_{\{x_1,x_2\}\times X}(R)$.
The morphism $\phi_n \colon B \rightarrow \Omega_n$ for $n \geq 2$ corresponds to the obvious bijection
\[
j_n \colon \{x_1,x_2\}\times X \rightarrow \{x_d^{n-2}\} \times  \{x_1, x_2\}  \times X
\]
that maps $u$ to $x_d^{n-2}u$.
We write $\mathcal{E}_d(R) = ((\Omega_n)_{n \in \N}, (\phi_n)_{n \geq 2})$.
As before $B_n$ denotes the image of $B$ in $\Omega_n$.

\begin{proposition}\label{prop:B-telescope-E}
Let $R$ be an associative unital ring.
Then $\mathcal{E}_d(R)$ is a flexible $\E_{2d}(R)$-telescope for $d \geq 4$.
\end{proposition}
\begin{proof}
Let $m \geq n \geq 2$.
We observe that all elements of $B_{n,m}$ are supported in 
\begin{equation}\label{eq:support-Bn-Ed}
Z(n,m) = \{x_d^{ n-2}x_i v \mid i \in \{1,2\}, v \in X^{m-n+1}\}.
\end{equation}
Since $d > 2$ the sets $Z(n,m)$ and $Z(m,m)$ are disjoint for $m > n$ and so $B_{n,m}$ and $B_m$ commute.
Next we verify that
\[
\Omega_n = \langle \{h B_n h^{-1} \mid h \in \Omega_{n-1,n} \} \rangle,
\]
where $\Omega_{n-1}$ is identified with its image in $\Omega_n$.
Since $d \geq 3$, we note that \eqref{eq:swap} shows that $\Omega_{n-1}$ acts $2$-transitively on the set of standard basis vectors in $R^{X^{n-1}}$ up to sign. Suppose $h \in \E_{X^{n-1}}(R)$ is a signed permutation that maps $\{x_d^{n-2}\} \times \{x_1, x_2 \}$ to a two element subset $U \subseteq X^{n-1}$. Then $h B_n h^{-1} = E_{U\times X}(R)$.
Since each pair of elements of $X^n$ is contained in a set of the form $U \times X$, we can deduce that $\Omega_n$ is generated by the conjugates $h B_n h^{-1}$.

To show that $\mathcal{E}_d(R)$ is flexible, we define the elements $\alpha_n \in \E_{X^n}(R)$ by
\[
	\alpha_n = s_{x_d^{n-1}x_{d-1},x^n_d}
\]
(see \eqref{eq:swap}).
We note that $\alpha_n$ is supported on 
$x_d^{n-1} \times \{x_{d-1},x_d\}$.
Let $n \geq 1$.
The elements of $B_{n+1}$ are supported on $Z(n+1,n+1)$ (see \eqref{eq:support-Bn-Ed}).
Since $\alpha_{n,n+1}$ is supported on the set $\{x_d^{n-1}\} \times \{x_{d-1},x_d\} \times X$, which is disjoint from $Z(n+1,n+1)$ for $d \geq 4$, we see that
axiom \ref{it:a-b} is satisfied.
Now we consider axioms \ref{it:ba-b} and \ref{it:ba-ba}. Let $i \geq 1$ and let $k \geq i+2$. We observe that $B_{k,m}^{\alpha_{i,m}}$ is supported on 
\[
Z^{(i)}(k,m) = \{x_d^{ i-1}x_{d-1}x_d^{k-i-2} x_t  x \mid t \in \{1,2\}, x \in X^{m-k+1}\}.
\]

\noindent Let $k, \ell \geq i+2$ and let $m \geq k,\ell$. Then the sets $Z^{(i)}(k,m)$ and $Z(\ell,m)$ are disjoint (recall that $d \geq 4$). Hence $[B_{k,m}^{\alpha_{i,m}},B_{\ell,m}] = 1$.
Finally, let $i > j \geq 1$ be given and let $k\geq i+2$ and $\ell \geq j+2$ and $m \geq k, \ell$.
The sets $Z^{(i)}(k,m)$ and
$Z^{(j)}(\ell,m)$ are disjoint and this implies \ref{it:ba-ba}.
We conclude that the $\E_{2d}(R)$-telescope $\mathcal{E}_d(R)$ is flexible.
\end{proof}
\begin{remark}\label{rem:E-telescopes}
\begin{enumerate}[label=(\alph*)]
\item Using a similar variation as in Section \ref{ex:Alt} it is possible to define flexible $\E_{rd}(R)$-telescopes (with $2 \leq r \leq d-2$) for all associative unital rings. 
\item The group $E_{m}(R)$ is perfect for $m \geq 3$; see \cite[1.2.15]{HahnOMeara}.
\item Suppose that $R$ is a commutative ring. If $R$ is a field, a Euclidean ring or the ring of integers in a global field, then $\E_{m}(R) = \SL_m(R)$ for $m \geq 3$; see \cite[4.3.9]{HahnOMeara}.
\end{enumerate}
\end{remark}

The latter case of Remark~\ref{rem:E-telescopes} will be reflected in our notation by writing $\mathcal{SL}_d(R)$ instead of $\mathcal{E}_d(R)$.

\begin{example}\label{ex:PSL}
Let $F$ be a field and let $d \geq 4$.
Consider the telescope
$\mathcal{E}_d(F) = ((\Omega_n)_{n \in \N}, (\phi_n)_{n \geq 2})$.
In this case $\Omega_n=\SL_{d^n}(F)$.
Let $Z_n \subseteq \Omega_n$ denote the center.
The homomorphism
\[
	\iota_{n,n+1} \colon \SL_{d^n}(F) \to \SL_{d^{n+1}}(F)
\]
is a blockdiagonal inclusion and takes scalar matrices to scalar matrices; i.e. $\iota_{i,j}(Z_i) \subseteq Z_j$.
In particular, the system $(Z_n)_{n\in \N}$ satisfies the assumption of Lemma~\ref{lem:quotient-telescope} and we can deduce that
$\mathcal{PSL}_d(F) = ((\Omega_n/Z_n)_{n\in \N}, (\overline{\phi}_n)_{n\geq 2})$ is a flexible $\SL_{2d}(F)$-telescope.
\end{example}

%%%%%%%%%%%
\section{Perfect telescopes and the profinite completion}\label{sec:profinite-completion}
Here we will study flexible telescopes $\mathcal{S}$ for perfect groups $B$. We will see that the subgroup structure of $G_\mathcal{S}$ can be described in quite some detail. This will allow us to show that the group and the huge group of $\mathcal{S}$ agree and that the profinite completion of $G_\mathcal{S}$ is the direct product $\prod_{i=1}^\infty \widehat{\Omega}_i$ of profinite completions of the groups $\Omega_i$ in the telescope.

\subsection{The structure of telescopes with $B$ perfect}
Throughout the section, $B$ denotes a perfect group and we fix a flexible $B$-telescope $\mathcal{S} = ((\Omega_i)_{i\in \N},  (\phi_i)_{i \geq 2})$. We observe that this implies that $\Omega_i$ is perfect for all $i\geq 2$; indeed, by \eqref{eq:generation} $\Omega_i$ is generated by perfect subgroups.
We discuss a number of preliminary results. The following elementary but useful observation is based on a standard trick in the theory of branch groups; see \cite{Neumann86,Segal01}.

\begin{lemma}\label{lem:segal-trick}
Let $G$ be a group.
Suppose that for every $n \in \N$ there are $n$ distinct subgroups $H_{1,n},\ldots,H_{n,n}$ in $G$ that pairwise commute, i.e.\ $[H_{i,n},H_{j,n}] = 1$ for $i \neq j$.
Then every normal subgroup $N$ of finite index in $G$ contains the normal closure of $H_{i,j}'$ for suitable $i,j$.
\end{lemma}
\begin{proof}
Let $N$ be a normal subgroup of finite index in $G$ and let $m$ be the number of subgroups $H$ with $N \leq H \leq G$.
By assumption we can find $m+1$ distinct subgroups $H_{i,m+1}$ in $G$ with $[H_{i,m+1},H_{j,m+1}] = 1$ for $i \neq j$.
The pigeonhole principle tells us that there are distinct $i,j$ with $NH_{i,m+1} = NH_{j,m+1}$.
We deduce
\[
H_{i,m+1}' \leq [NH_{i,m+1},NH_{i,m+1}] = [NH_{i,m+1},NH_{j,m+1}] \leq N
\]
and $N$ contains the normal closure of $H_{i,m+1}'$.
\end{proof}
\begin{lemma}\label{lem:G-subgroups}
The group $G_\mathcal{S} = \langle \Delta_1(\Omega_1), C^{[1]}\rangle$
contains $\Delta_i(\Omega_i)$ and $C^{[i]}$ for all $i \in \N$.
\end{lemma}
\begin{proof}
We write $G = G_{\mathcal{S}}$.
We proceed by induction on $i$. For $i = 1$ the assertion follows from the definition of $G$.
Let $i \geq 1$ and assume that $\Delta_i(\Omega_i) \subseteq G$ and  $C^{[i]} \subseteq G$. 
Let $h,k \in B$ and $\alpha_i \in \Omega_i$ be as in Definition~\ref{def:flexible-telescope}.
Then $\tilde{h}^{[i]}, \tilde{k}^{[i]},\Delta_i(\alpha_i) \in G$.
Consider the element $g = [\tilde{h}^{[i]}, \Delta_i(\alpha_i)\tilde{k}^{[i]}\Delta_i(\alpha_i)^{-1}]$.
For every $j \leq i$ we have $\pi_j(\tilde{h}^{[i]}) = 1$ and $\pi_j(\tilde{k}^{[i]}) = 1$, which gives us $\pi_j(g) = 1$.
By applying \ref{it:a-b} we obtain
\begin{align*}
\pi_{i+1}(g)
&= [\pi_{i+1}(\tilde{h}^{[i]}), \pi_{i+1}(\Delta_i(\alpha_i)) \pi_{i+1}(\tilde{k}^{[i]}) \pi_{i+1}(\Delta_i(\alpha_i)^{-1})]\\
&= [\phi_{i+1}(h), \alpha_{i,i+1} \phi_{i+1}(k) \alpha_{i,i+1}^{-1}]\\
&= [\phi_{i+1}(h), \phi_{i+1}(k)]\\
&= \phi_{i+1}([h, k]).
\end{align*}
Let us now consider $\pi_j(g) \in \Omega_{j}$ for $j \geq i+2$.
Using \ref{it:a-b}, \ref{it:ba-b} and \eqref{eq:commutator}, we obtain
\begin{align*}
\pi_j(g_i)
&=\ [\pi_j(\tilde{h}^{[i]}), \pi_j(\Delta_i(\alpha_i)) \pi_j(\tilde{k}^{[i]}) \pi_j(\Delta_i(\alpha_i)^{-1})]\\
&=\ \Bigl[\prod_{\ell=i+1}^j \iota_{\ell,j}(\phi_{\ell}(h)), \prod_{\ell=i+1}^j \iota_{\ell,j}(\phi_{\ell}(k))^{\alpha_{i,j}}\Bigr]\\
&=\ \Bigl[\prod_{\ell=i+1}^j \iota_{\ell,j}(\phi_{\ell}(h)), \iota_{i+1,j}(\phi_{i+1}(k))\Bigr]\\
&=\ [\iota_{i+1,j}(\phi_{i+1}(h)),\iota_{i+1,j}(\phi_{i+1}(k))]\\
&=\ \iota_{i+1,j}(\phi_{i+1}([h,k])).
\end{align*}
Together these equations show that $g = \Delta_{i+1}(\phi_{i+1}([h,k]))$.
Since $B$ is perfect, we deduce that $G$ contains $\Delta_{i+1}(B_{i+1})$.
By  \eqref{eq:generation} the group $\Omega_{i+1}$ is generated by the $\Omega_{i,i+1}$-conjugates of $B_{i+1}$. As $\Delta_i(\Omega_i) \subseteq G$ and $\Delta_{i+1}(B_{i+1}) \subseteq G$, it follows that $\Delta_{i+1}(\Omega_{i+1})$ lies in $G$.
Since $\tilde{g}^{[i]} = \Delta_{i+1}(\phi_{i+1}(g))\tilde{g}^{[i+1]}$ holds for all $g \in B$, we conclude that $C^{[i+1]}$ is contained in $G$. 
\end{proof}

\begin{corollary}\label{cor:first-levels}
Under our assumptions we have $G_\mathcal{S} = \widetilde{G}_\mathcal{S}$.
In particular, the group $G_{\mathcal{S}}$ contains the infinite direct sum $\bigoplus \limits_{i \in \N} \Omega_i$.
\end{corollary}
\begin{proof}
Let $i \in \N$ and let $\omega \in \Omega_i$.
From Lemma~\ref{lem:G-subgroups} we know that $\Delta_i(\omega)$ and $\Delta_{i+1}(\iota_{i,i+1}(\omega^{-1}))$ lie in $G_{\mathcal{S}}$.
Thus $g \defeq \Delta_i(\omega) \Delta_{i+1}(\iota_{i,i+1}(\omega^{-1}))$ lies in $G$.
By definition we have $\pi_j(g) = 1$ for $j < i$ and $\pi_i(g) = \omega$.
For $j > i$ we obtain
\[
\pi_j(g)
= \iota_{i,j}(\omega) \iota_{i+1,j}(\iota_{i,i+1}(\omega^{-1}))
= \iota_{i,j}(\omega) \iota_{i,j}(\omega^{-1})
= 1.
\]
Together this shows that $(\omega_j)_{j \in \N} \in \bigoplus \limits_{j \in \N} \Omega_j$ with $\omega_i = \omega$ and $\omega_j = 1$ for $j \neq i$ lies in $G_{\mathcal{S}}$, which proves the claim.
\end{proof}
\begin{lemma}\label{lem:level-subgroups}
Let $n \geq 1$ and define $K_{n-1} = \ker(G_{\mathcal{S}} \to \Omega_1\times \dots \times \Omega_{n-1})$.
Then $K_{n-1} = \langle \Delta_{n}(\Omega_{n}), C^{[n]}\rangle$.
\end{lemma}
\begin{proof}
From Lemma~\ref{lem:G-subgroups} we obtain the inclusion $\langle \Delta_n(\Omega_n), C^{[n]} \rangle \leq K_{n-1}$.
To establish the converse inclusion 
we proceed by induction on $n$.
For $n = 1$ we have $K_0 = G_{\mathcal{S}}$ and the assertion follows from the definition of $G_{\mathcal{S}}$.
Assume that $K_{n-1} = \langle \Delta_{n}(\Omega_{n}), C^{[n]}\rangle$ for some $n \in \N$.
Let $x \in K_n$.
Since $K_n \leq K_{n-1}$ we can write
\[
x = \Delta_n(\omega_1) c_1 \Delta_n(\omega_2) c_2 \cdots \Delta_n(\omega_r) c_r
\]
for certain $\omega_1,\dots,\omega_r \in \Omega_n$ and $c_1,\dots,c_r \in C^{[n]}$.
Projecting $x$ onto $\Omega_n$ we get 
	$1 = \omega_1 \omega_2 \cdots \omega_r$
and hence
\[
x = \Delta_{n+1}(\iota_{n,n+1}(\omega_1)) c_1 \Delta_{n+1}(\iota_{n,n+1}(\omega_2)) c_2 \cdots \Delta_{n+1}(\iota_{n,n+1}(\omega_r)) c_r.
\]
By writing $c_i = \Delta_{n+1}(h_i)c'_i$ with $c'_i \in C^{[n+1]}$ and $h_i \in B_{n+1}$ and obtain
\[
x = \Delta_{n+1}(\iota_{n,n+1}(\omega_1)h_1) c'_1 \Delta_{n+1}(\iota_{n,n+1}(\omega_2)h_2) c'_2 \cdots \Delta_{n+1}(\iota_{n,n+1}(\omega_r)h_r) c'_r,
\]
which is an element in $\langle \Delta_{n+1}(\Omega_{n+1}),C^{[n+1]} \rangle$.
Since $x \in K_n$ was arbitrary, we deduce that $K_n = \langle \Delta_{n+1}(\Omega_{n+1}),C^{[n+1]} \rangle$.
\end{proof}
Lemma~\ref{lem:level-subgroups} allows us to relate the group of a flexible $B$-telescope to the group a shifted telescope.
\begin{corollary}\label{cor:group-associated-to-shifted-B-telescope}
Let $\mathcal{S}$ be flexible $B$-telescope and let $n \in \N$.
If $B$ is perfect, then $G_{\mathcal{S}_{+n}} = K_{n}$, where $K_{n} = \ker(G_{\mathcal{S}} \to \Omega_1\times \dots \times \Omega_{n})$.
For every $n \in \N$ we have
\[
G_{\mathcal{S}} = G_{\mathcal{S}_{+n}} \oplus \bigoplus \limits_{i=1}^{n} \Omega_i \subseteq \prod \limits_{i=1}^{\infty} \Omega_i = P_{\mathcal{S}}.
\]
\end{corollary}
\begin{proof}
By definition we have $G_{\mathcal{S}_{+n}} = \langle \Delta_{n+1}(\Omega_{n+1}) \cup C^{[n+1]} \rangle$.
On the other hand, our assumptions allow us to apply Lemma~\ref{lem:level-subgroups} to deduce that $K_n = \langle \Delta_{n+1}(\Omega_{n+1}) \cup C^{[n+1]} \rangle$, which proves the first claim.
From Corollary~\ref{cor:first-levels} we know that $G_{\mathcal{S}}$ contains the infinite direct sum $\bigoplus \limits_{i \in \N} \Omega_i$.
Thus $K_n$ coincides with the image of the canonical map $G_{\mathcal{S}} \rightarrow \prod \limits_{i>n} \Omega_i$.
This gives us
\[
G_{\mathcal{S}}
= K_{n} \oplus \bigoplus \limits_{i=1}^{n} \Omega_i
= G_{\mathcal{S}_{+n}} \oplus \bigoplus \limits_{i=1}^{n} \Omega_i,
\]
which proves the corollary.
\end{proof}
\begin{lemma}\label{lem:normal-closure}
Let $n \geq 1$. Then $K_n$ is the normal closure of $C^{[n]}$ in $G_{\mathcal{S}}$.
\end{lemma}
\begin{proof}
Let $N \trianglelefteq G_{\mathcal{S}}$ be a normal subgroup that contains $C^{[n]}$. 
Let $h,k \in B$, then $\tilde{h}^{[n]}, \tilde{k}^{[n]} \in C^{[n]}$.
Recall from Lemma~\ref{lem:G-subgroups} that $G_{\mathcal{S}}$ contains $\Delta_i(\Omega_i)$ and $C^{[i]}$ for all $i \in \N$.
As in Lemma~\ref{lem:G-subgroups} it follows from \ref{it:a-b} and \ref{it:ba-b} that
\[
[\tilde{h}^{[n]}, \Delta_n(\alpha_n)\tilde{k}^{[n]}\Delta_n(\alpha_n)^{-1}] = \Delta_{n+1}(\phi_{n+1}([h,k])),
\]
which is contained in $N$.
Since $B$ is perfect, we deduce that $N$ contains $\Delta_{n+1}(B_{n+1})$.
As a normal subgroup of $G_{\mathcal{S}}$, the group $N$ contains the $\Delta_n(\Omega_n)$-conjugates of $\Delta_{n+1}(B_{n+1})$, which generate $\Delta_{n+1}(\Omega_{n+1})$ by~\eqref{eq:generation}.
Thus $\Delta_{n+1}(\Omega_{n+1})$ lies in $N$.
Since $C^{[n+1]}$ is clearly contained in $\langle \Delta_{n+1}(\Omega_{n+1}), C^{[n]}\rangle$, it remains to apply Lemma \ref{lem:level-subgroups}, which gives us $K_n = \langle \Delta_{n+1}(\Omega_{n+1}), C^{[n+1]}\rangle \subseteq N$.
\end{proof}
\begin{lemma}\label{lem:normal-closure-of-Delta} 
Let $N$ be the normal closure of $\Delta_i(\Omega_i)$ in $G_{\mathcal{S}}$.
For $i \geq 2$ we have $N = K_{i-1}$.
For $i=1$ we always have $N \supseteq K_{2}$ and
if $\Omega_2$ is simple and $\Omega_3$ is non-trivial, then $N = G_{\mathcal{S}}$ .
\end{lemma}
\begin{proof}
Let $N$ be the normal closure of $\Delta_i(\Omega_i)$ in $G_{\mathcal{S}}$.
Clearly, $K_{i-1}$ contains $N$. Let $j = \max(2,i)$. Let $h,k \in B$ be arbitrary. Then
\begin{align*}
 x \defeq [\tilde{h}^{[j]},\Delta_{i}(\alpha_{1,i})] &= \tilde{h}^{[j]}\Delta_{i}(\alpha_{1,i})(\tilde{h}^{[j]})^{-1}\Delta_{i}(\alpha_{1,i})^{-1}
 \end{align*}
 lies in $N$. By \ref{it:ba-b} and \ref{it:ba-ba} we have
 \[
 	[\tilde{k}^{[j]},x] = [\tilde{k}^{[j]},\tilde{h}^{[j]}] \in N.
 \]
Since $B$ and consequently $C^{[j]}$ are perfect, we deduce that $C^{[j]} \subseteq N$. Lemma~\ref{lem:normal-closure} implies that
$K_{j} \subseteq N$. If $i \geq 2$, then $j=i$ and we deduce that $K_{i-1} = \Delta_i(\Omega_i)K_{i} \subseteq N$.
 
If $i = 1$ we assume that $\Omega_2$ is simple and $\Omega_3$ is non-trivial. By the above argument $N$ contains $K_{2} = \ker(G_{\mathcal{S}} \to \Omega_1\times\Omega_2)$, since $N$ contains $\Delta_1(\Omega_1)$, it projects onto the first factor $\Omega_1$. By simplicity of $\Omega_2$ it remains to show that $\pi_2(N)$ is non-trivial. We claim that the element $\pi_2(\Delta_1(\alpha_1)) = \alpha_{1,2}$ is non-trivial.
Indeed, since $\Omega_3$ is non-trivial \eqref{eq:generation} implies that the group $B_3$ is non-trivial and perfect. The element $\alpha_{1,3} = \iota_{2,3}(\alpha_{1,2})$ satisfies $[B_3^{\alpha_{1,3}},B_3] = 1$ by \ref{it:ba-b}, hence $\alpha_{1,3} \neq 1$ and therefore $\alpha_{1,2} \neq 1$.
\end{proof}

%%%%%%%%%%%%%%%%%%%%%%%
\begin{proposition}
Let $\mathcal{S}$ be a flexible $B$-telescope for a perfect group $B$.
Assume that $\Omega_i$ is non-trival for infinitely many $i \in \N$. Then
\begin{enumerate}[label=(\alph*)]
\item $G_\mathcal{S}$ is not linear.
\item $G_\mathcal{S}$ is $\ell^2$-acyclic.
\end{enumerate}
\end{proposition}
\begin{proof}
By Corollary~\ref{cor:first-levels} $G_\mathcal{S}$ contains the infinite direct sum $\bigoplus_{i=1}^\infty \Omega_i$. Since $B$ is perfect, the groups $\Omega_i$ are perfect for $i\geq 2$. If infinitely many $\Omega_i$ are non-trivial, then $G_\mathcal{S}$ cannot be linear by \cite[Corollary 6]{Abert06} (see also \cite{KS2023}).

Assume now that $\Omega_i$ is a non-trivial finite group for an infinite set of indices $i \in S \subseteq  \N$.
Then $\bigoplus_{i\in S} \Omega_i$ is an infinite amenable normal subgroup of $G_\mathcal{S}$ and it follows from \cite[Thm.~7.2]{Lueck:l2-invariants} that $G_\mathcal{S}$ is $\ell^2$-acyclic.
We now assume that $\Omega_i$ is infinite for almost all $i$ and
we prove that the $k$-th $\ell^2$-Betti number of $G_\mathcal{S}$ vanishes using induction on $k$.
We observe that if $\Omega_i$ is infinite, then $G_\mathcal{S}$ is infinite (by Corollary~\ref{cor:first-levels}) and thus the $0$-th $\ell^2$-Betti number of $G_\mathcal{S}$ vanishes, i.e., $b_0^{(2)}(G_\mathcal{S}) = 0$. 
Now let $k \geq 1$ and let $n \in \N$ be minimal such that $\Omega_n$ is infinite.
Then by Corollary \ref{cor:group-associated-to-shifted-B-telescope} $G_\mathcal{S}$ contains a finite index subgroup $H$ isomorphic to 
\[
	G_{\mathcal{S}_{+n}} \times \Omega_n.
\]
By induction hypothesis, the K\"unneth formula and the induction formula for $\ell^2$-Betti numbers \cite[Theorem 6.54 (5), (6)]{Lueck:l2-invariants} we obtain
\[
	b_k^{(2)}(G_\mathcal{S}) = \frac{1}{|G_\mathcal{S}:H|} \sum_{i+j = k} b_i^{(2)}(G_{\mathcal{S}_{+n}})b_{j}^{(2)}(\Omega_n)
	= \frac{1}{|G_\mathcal{S}:H|} b_k^{(2)}(G_{\mathcal{S}_{+n}})b_{0}^{(2)}(\Omega_n) = 0.
\]
\end{proof}
%%%%%%%%%%%%%%%%%%%%%%%
\subsection{The profinite completion}
Corollary~\ref{cor:first-levels} has the following direct consequence.
\begin{corollary}\label{cor:dense}
Let $\mathcal{S} = ((\Omega_i)_{i\in \N},  (\phi_i)_{i \geq 2})$ be a flexible $B$-telescope for a perfect group $B$.
The group $G_{\mathcal{S}} = \langle \Delta_1(\Omega_1), C^{[1]}\rangle$ projects onto $\Omega_1\times \dots \times \Omega_n$ for all $n \geq 1$. In particular, the image of $G_\mathcal{S}$ is dense in $\prod_{i=1}^\infty \widehat{\Omega}_i$.
\end{corollary}
This result allows us to deduce the following theorem.
\begin{theorem}\label{thm:completion}
Let $B$ be a perfect group and let $\mathcal{S} = ((\Omega_i)_{i\in \N},  (\phi_i)_{i \geq 2})$ be a flexible $B$-telescope.
Then the inclusion $G_{\mathcal{S}} \to P_\mathcal{S}$ induces an isomorphism
\[
	\widehat{G}_{\mathcal{S}} \stackrel{\cong}{\longrightarrow} \prod_{i=1}^\infty \widehat{\Omega}_i.
\]
If $\Omega_1$ and $B$ are finitely generated, then $G_\mathcal{S}$ is finitely generated.
If all $\Omega_i$ are residually finite, then $G_\mathcal{S}$ is residually finite.
\end{theorem}
\begin{proof}
The finite index subgroups of $G_{\mathcal{S}}$ that contain $K_n$ correspond to the finite index subgroups of $\Omega_1\times \dots \times \Omega_n$ by Corollary \ref{cor:dense}.
Given a finite index normal subgroup $N$ of  $G_{\mathcal{S}}$, it is sufficient to prove that $N$ contains $K_n$ for some $n$ (one might call this the ``congruence subgroup property'').

Let $n \geq 2$ be given. The subgroups $\Delta(\alpha_i^{-1})C^{[n]}\Delta(\alpha_i)$ commute pairwise for all $i \in \{1,2,\dots,n-2\}$ by assumption \ref{it:ba-ba}.
Thus we can deduce from Lemma \ref{lem:segal-trick} that $N$ contains the derived subgroup of some $\Delta(\alpha_i^{-1})C^{[n]}\Delta(\alpha_i)$.
Since $N$ is normal and $C^{[n]}$ is perfect - being an image of the perfect group $B$ - we deduce that $C^{[n]} \subseteq N$.
In this case Lemma \ref{lem:normal-closure} implies $K_n \subseteq N$.
\end{proof}

Let $P = \prod_{n=1}^{\infty} S_n$ be a product of finite groups. We recall that a frame in $P$ is a
finitely generated subgroup $G \leq P$ such that it contains
$\bigoplus_{n=1}^{\infty} S_n$ and the induced map $\widehat{G} \rightarrow P$ is an isomorphism (see Definition~\ref{def:frame-intro}).
Consequently, flexible $B$-telescopes of finite groups give rise to surprisingly simple examples of frames. We will see next that the frames provided by our main example can be generated by the minimal number of generators.
Let $d \geq 5$ and $2 \leq r \leq d-3$ be given. We consider the flexible $\Alt(rd)$-telescope $\mathcal{A}(d,r)$ defined in \S~\ref{ex:Alt}.
%%%%
\begin{theorem}\label{thm:frame-minimal-Alt}
The group $G_{\mathcal{A}(d,r)}$ is $2$-generated and a frame in $\prod \limits_{\ell=1}^{\infty} \Alt(d^\ell)$.
\end{theorem}
%%%%
\begin{proof}
Let $\Gamma = G_{\mathcal{A}(d,r)}$ be the group of $\mathcal{A}(d,r)$.
From Proposition~\ref{prop:B-telescope} we know that $\mathcal{A}(d,r)$ is flexible.
Since the group $B$ is the perfect group $\Alt(rd)$, we can conclude from Theorem~\ref{thm:completion} that the inclusion $\Gamma \rightarrow \prod \limits_{\ell=1}^{\infty} \Omega_{\ell}$ induces an isomorphism $\widehat{\Gamma} \rightarrow \prod \limits_{\ell=1}^{\infty} \Omega_{\ell}$.
Corollary~\ref{cor:first-levels} tells us that the direct sum $\bigoplus \limits_{\ell=1}^{\infty} \Omega_{\ell}$ is contained in $\Gamma$, hence it is a frame.

It remains to show that $\Gamma$ is $2$-generated. We observe that it is sufficient to prove that the subgroup $K_{n}=G_{\mathcal{S}_{+n}}$ is $2$-generated for some $n$. Indeed, it is known that $\bigoplus_{j=1}^n \Alt(d^j)$ is $2$-generated and since $\bigoplus_{j=1}^n \Alt(d^j)$ and $K_n$ don't share a non-trivial factor, the product $K_n\times\bigoplus_{j=1}^n \Alt(d^j)$ does not admit a proper subdirect subgroup.

Recall that $X =\{x_1,\dots,x_d\}$. Let $n \in \N$ be such that $d^n \geq \max(18,3(r+1),2rd)$. Note that $n \geq 2$. We fix a prime number $p$ satisfying 
\[
d^n/2 < p < (1+\frac{1}{3})d^n/2  = \frac{2}{3}d^n \leq d^n - (r+1),
\]
which exists by the main theorem of \cite{Nagura52} .
Let $\sigma_1$ be a $p$-cycle in $\Alt(X^n)$ that fixes $\{x_d^{n-1}\}\times\{x_1,\dots,x_r, x_d\}$ and moves $\{x_1^{n-1}\}\times X$. Let $\sigma_2$ be a $p$-cycle that moves every fixed point of $\sigma_1$ and fixes $\{x_1^{n-1}\}\times \{x_1,\dots,x_r, x_d\}$.
Since $p > d^n/2$, the supports of $\sigma_1$ and $\sigma_2$ cannot be disjoint and by construction they generate a transitive subgroup $T \leq \Alt(d^n)$.
Using an old theorem of C.~Jordan one can show that any transitive subgroup $T \subseteq \Alt(d^n)$ containing a $p$-cycle is equal to  the alternating group $\Alt(d^n)$; see \cite[Theorem 4.2]{HandbookComp}.
Let $\tau_1, \tau_2$ be any generating set for $\Alt(rd)$. We observe that $p > d^n/2 \geq rd$ and hence the order of each $\tau_i$ is coprime to $p$. This implies that $\tau_1^p$ and $\tau_2^p$ generate $\Alt(rd)$.

We define $a = \Delta_{n}(\sigma_1)\tilde{\tau}_1^{[n]}$ and we note that by construction $\Delta_{n}(\sigma_1)$ and $\tilde{\tau}_1^{[n]}$ commute. Taking $a^p$ and $a^{\ord(\tau_1)}$ we see that $\langle a \rangle$ contains $\Delta_{n}(\sigma_1)$ and $\tilde{\tau}_1^{[n]}$ . Let $\omega\in \Alt(X^{n-1})$ be such that $\omega(x_1^{n-1}) = x_d^{n-1}$.
Consider the elements
\begin{equation}\label{eq:generator-c}
c = \Delta_{n-1}(\omega)^{-1}\tilde{\tau}_2^{[n]}\Delta_{n-1}(\omega) = \Delta_{n}(\omega_{n-1,n})^{-1}\tilde{\tau}_2^{[n]}\Delta_{n}(\omega_{n-1,n})
\end{equation}
and $b = \Delta_{n}(\sigma_2)c$. Again $\Delta_n(\sigma_2)$ and $c$ commute and so $\langle b \rangle$ contains $\Delta_n(\sigma_2)$ and $c$.
We conclude that the group $H =\langle a,b \rangle$ contains $\tilde{\tau}^{[n]}_1$ , $c$ and $\Delta_n(\langle \sigma_1,\sigma_2 \rangle ) = \Delta_n(\Alt(d^n))$;  in particular, it contains $\Delta_n(\omega_{n-1,n})$. Thus by \eqref{eq:generator-c} $H$ also contains $\tilde{\tau}_2^{[n]}$ and since $\tau_1$ and $\tau_2$ generate $\Alt(rd)$, we deduce that $H = \langle \Delta_n(\Alt(X^n)), C^{[n]} \rangle = K_{n-1}$ by Lemma~\ref{lem:level-subgroups}. This proves that $K_{n-1}$ is $2$-generated and completes the proof of the theorem.
\end{proof}
%%%%%%%%%
An analogous result applies to the $\SL_{2d}(\F_q)$-telescope defined in \S~\ref{ex:elementary}.
In order to prove it, we need a result about generating sets of $\SL_n(\F_q)$.
\begin{lemma}\label{lem:generators-SL}
Let $n \geq 11$, $\ell$ a prime number and let $\F_{q}$ be a finite field with $q = \ell^k$. Suppose that $p$ is an odd prime number satisfying $n/2 < p < n-5$. Then $\SL_n(\F_q)$ is generated by a pair of matrices  $U,V \in \SL_n(\F_q)$ with orders
	$\ord(U) = 2p$ and $\ord(V) = 3p$.
The generators $U$ and $V$ can be chosen such that each of them is supported on at most $p+5$ standard basis vectors.
\end{lemma}
\begin{proof}
The group $\SL_5(\F_q)$ is $(2,3)$-generated, i.e., there are generators $A,B \in \SL_5(\F_q)$ of order $2$ and $3$; see \cite{Tchakerian05}.
Let $U$ be a block diagonal matrix
\[
	U = \left(\begin{array}{c|ccc|c}C_1 & & & & \\
	\hline
	& 1 & & &  \\
	& & \ddots & &  \\
	& & & 1 &  \\
	\hline
	& & & & A \end{array}\right)
\] 
whose upper left $(p\times p)$-block $C_1$ is a $p$-cycle permutation matrix. Let $V$ be a block diagonal matrix 
\[
	V = \left(\begin{array}{c|ccc|c} B & & & &\\
	\hline
	& 1 & & & \\
	& & \ddots & & \\
	& & & 1 &  \\
	\hline
	& & & & C_2\end{array}\right)
\] 
 whose lower right $(p\times p)$-block $C_2$ is a $p$-cycle permutation matrix.
Since $p > n/2 > 5$ we have $\ord(U)=2p$ and $\ord(V) = 3p$.
The group $\langle U, V \rangle$ contains the permutation matrices $U^2$, $V^{3}$ which generate a transitive subgroup of $\Alt(n)$ that contains a $p$-cycle, so by \cite[Theorem 4.2]{HandbookComp} they generate the group of alternating permutation matrices. Since $\Alt(n)$ acts $n-2$-transitively (and $n-2>5$), we can pick a permutation matrix $P \in \langle U, V \rangle$ such that $PU^pP^{-1}$ has the $5\times 5$-block $A$ in the upper left position. It follows that $PU^pP^{-1}$ and $V^p$ generate a copy of $\SL_5(\F_q)$ in the upper left block, that contains in particular elementary matrices. By conjugating these elementary matrices with permutation matrices, we find all elementary matrices in $\langle U, V \rangle$ and we deduce that
$\SL_n(\F_q) = \langle U, V \rangle$. Note that by construction each of $U$ and $V$ is supported on at most $p+5$ standard basis vectors.
\end{proof}
%%%%
\begin{theorem}\label{thm:frame-minimal-SL}
Let $d \geq 4$ and let $\F_q$ be a finite field. The group $G_{\mathcal{SL}_d(\F_q)}$ is a $2$-generated frame for $\prod\limits_{\ell=1}^{\infty}\SL_{d^\ell}(\F_q)$.
\end{theorem}
\begin{proof}
Let $\Gamma=G_{\mathcal{SL}_d(\F_q)}$ denote the group of the $\SL_{2d}(\F_q)$-telescope.
From Proposition~\ref{prop:B-telescope-E} we know that $\mathcal{SL}_d(\F_q)$ is flexible.
Moreover the group $B = \SL_{2d}(\F_q)$ is perfect since $2d \geq 3$. Again Theorem~\ref{thm:completion} implies that the inclusion $\Gamma \rightarrow \prod \limits_{\ell=1}^{\infty} \Omega_{\ell}$ induces an isomorphism $\widehat{\Gamma} \rightarrow \prod \limits_{\ell=1}^{\infty} \Omega_{\ell}$ and Corollary~\ref{cor:first-levels} shows that it is a frame.
 
 Let $q = \ell^k$ for a prime number $\ell$. We will show that $\Gamma$ is $2$-generated. As in the proof of Theorem \ref{thm:frame-minimal-Alt} it is sufficient to prove that the level subgroup $K_n$ is $2$-generated for some $n$.
 Choose $n$ such that $d^n \geq 24$.
 By the main theorem of \cite{Nagura52} there is a prime number $p$ such that $d^n/2 < p < \frac{2}{3}d^n \leq d^n -8$. Let $U,V \in \SL_{d^n}(\F_q)$ be generators of order $2p$ and $3p$ as in Lemma~\ref{lem:generators-SL}. 
Since $2d\geq 8$ it follows from \cite{Pellegrini17} that $B = \SL_{2d}(\F_q)$ can be generated by two elements $s,t$ where $\ord(s) = 2$ and $\ord(t) = 3$. 
 By Lemma~\ref{lem:generators-SL}  there are $d^n-p-5 \geq 3$ standard basis vectors, that are not contained in the support of $U$; i.e., we may conjugate $\tilde{t}^{[n]}$ with an element $\omega \in \Delta_n(\Omega_n)$ such that $\delta_1\tilde{t}^{[n]}\delta_1^{-1}$ and $\Delta_n(U)$ commute.
 Similarly, we may conjugate $\tilde{s}^{[n]}$ with an element $\delta_2 \in \Delta_n(\Omega_n)$ such that $\delta_2\tilde{s}^{[n]}\delta_2^{-1}$ and $\Delta_n(V)$ commute.
 We deduce that 
 \begin{align*}
 \langle \Delta_n(U)\omega\tilde{t}^{[n]}\omega^{-1}, \Delta_n(V)\tilde{s}^{[n]}\rangle &\supseteq \langle \Delta_n(U)^{3},\Delta_n(V)^2,\delta_1\tilde{t}^{[n]}\delta_1^{-1},\delta_2\tilde{s}^{[n]}\delta_2^{-1} \rangle\\ 
 &\supseteq \langle \Delta_n(\Omega_n), C^{[n]}\rangle = K_{n-1}. \qedhere
 \end{align*}
 \end{proof}
 %%%%%
The construction of telescopes of matrix groups is not limited to fields and is able to provide more exotic examples.
Let $R$ be an associative unital ring and let
 $\mathcal{E}_d(R)$ denote the telescope defined in Example~\ref{ex:elementary}.
The profinite completion of the associated group $G_{\mathcal{E}_d(R)}$ is given by
$\widehat{G_{\mathcal{E}_d(R)}} \cong \prod_{j=1}^\infty \widehat{\E_{d^j}(R)}$, where $\E_n(R)$ denotes the subgroup of $\GL_n(R)$ generated by elementary matrices.
%, see Proposition~\ref{prop:B-telescope-E} and Theorem~\ref{thm:completion}.
Now $G_{\mathcal{E}_d(R)}$ is finitely generated if and only if $\E_{d}(R)$ is finitely generated. In this case, we deduce that the minimal number of generators needed to generated $\bigoplus \limits_{n=1}^{m} \E_{d^n}(R)$ is bounded above by some constant only depending on $d$ and $R$.
Although this is probably known, we would like to mention that the latter shows, without using any arithmetic, that $\E_{d^m}(R)$ does not admit quotients of a given finite cardinality when $m$ is chosen large enough.

\medskip 

Let us consider a concrete instance.
Given $d \geq 4$, we consider the flexible $\SL_{2d}(\Z)$-telescope $\mathcal{SL}_d(\Z)$ described in \S \ref{ex:elementary}.
Since $\SL_{2d}(\Z)$ is perfect and finitely generated, it follows from Theorem \ref{thm:completion} that $G_{\mathcal{SL}_d(\Z)}$ is a finitely generated group whose profinite completion is isomorphic to
\[
\widehat{G_{\mathcal{SL}_d(\Z)}} \cong \prod_{k=1}^\infty \widehat{\SL_{d^\ell}(\Z)} \cong \prod_{k=1}^\infty \prod_{p \text{ prime }} \SL_{d^k}(\Z_p),
\]
where the latter isomorphism is a consequence of the congruence subgroup property of $\SL_n(\Z)$ for $n\geq 3$.
\begin{corollary}\label{cor:products-sln-4-generated}
The group $\prod \limits_{i=1}^{k} \SL_{d^{i}}(\Z)$ is $4$-generated for every $d \geq 4$ and every $k \in \N$.
\end{corollary}
\begin{proof}
It is a classical result of Trott~\cite{Trott62} that $\SL_n(\Z)$ can be generated by $2$ elements for $n \geq 3$. Hence $G_{\mathcal{SL}_d(\Z)}$ is $4$-generated and maps onto $\prod \limits_{i=1}^{k} \SL_{d^{i}}(\Z)$ by Corollary~\ref{cor:group-associated-to-shifted-B-telescope}.
\end{proof}
\subsection{Simplicity of the head}
 Let $\mathcal{S}=((\Omega_i)_{i \in \N},  (\phi_i)_{i \geq 2})$ be a flexible $B$-telescope for a perfect group $B$. Recall that $Q_\mathcal{S} = G_\mathcal{S}/\bigoplus_{i\in \N} \Omega_i$ is called the \emph{head} of $\mathcal{S}$.
 It follows from Theorem \ref{thm:completion} that the head $Q_{\mathcal{S}}$ does not admit any finite quotients.
 By Corollary \ref{cor:group-associated-to-shifted-B-telescope} there is a canonical isomorphism between the head $Q_\mathcal{S}$ and the head $Q_{\mathcal{S}_{+n}}$ of any shifted $B$-telescope.
We will later see examples where the head is a simple group. The following criterion gives one way to verify this.
\begin{lemma}\label{lem:simplicity-criterion}
Let $B$ be a perfect finite group and let $\mathcal{S} = ((\Omega_i)_{i\in \N},  (\phi_i)_{i \geq 2})$ be a flexible $B$-telescope.
Assume that $\Omega_1$ is finitely generated. We consider the group 
\[
	\mathcal{N}^\infty = \bigcup_{j=1}^\infty \bigcap_{k=j}^{\infty} N_{G_\mathcal{S}}(C^{[k]}).
\]
Assume that every maximal subgroup $M\subseteq G_\mathcal{S}$ containing $\mathcal{N}^\infty$ satisfies $\mathrm{core}(M)=\mathrm{core}(\mathcal{N}^\infty)$. Then $G_\mathcal{S}/\mathrm{core}(\mathcal{N}^\infty)$ is a simple group.
\end{lemma}
 We note that $\bigoplus_{i\in \N} \Omega_i$ is always contained in $\mathcal{N}^\infty$ (and thus in $\mathrm{core}(\mathcal{N}^\infty)$).
 \begin{proof}
 Let $N \trianglelefteq G_\mathcal{S}$ be a normal subgroup that contains $\mathrm{core}(\mathcal{N}^\infty)$ properly. By assumption $N$ is not contained in a maximal subgroup $M$ containing $\mathcal{N}^\infty$.
 Since $N\mathcal{N}^\infty$ is larger than $\mathcal{N}^\infty$ and not contained in any maximal subgroup containing $\mathcal{N}^\infty$,
 we have $N\mathcal{N}^\infty = G_\mathcal{S}$, i.e., the restriction of the canonical projection $p \colon \mathcal{N}^\infty \to G_\mathcal{S}/N$ is onto.
 By assumption $G_{\mathcal{S}}$ is finitely generated; we pick a finite generating set $g_1,\dots, g_r$ and choose inverse images $n_1,\dots, n_r \in  \mathcal{N}^\infty$ under $p$.
 By definition of $\mathcal{N}^\infty$ there is a $j \geq 1$ such that $n_1,\dots,n_r$ normalize $C^{[j]}$.
 In turn, the image $p(C^{[j]})$ in $G_\mathcal{S}/N$ is a normal subgroup.
 By Lemma \ref{lem:normal-closure} the normal closure of $C^{[j]}$ generates $G_\mathcal{S}$ with $\bigoplus_{i\in\N} \Omega_i \subseteq N$ and hence $p(C^{[j]}) =G_\mathcal{S}/N$.
Since $C^{[j]}$ is  finite, it follows that $G_\mathcal{S}/N$ is finite. However, $Q_{\mathcal{S}}$ doesn't admit non-trivial finite quotients, thus 
$N = G_{\mathcal{S}}$ and $G/\mathrm{core}(\mathcal{N}^\infty)$ is simple. 
 \end{proof}
%%%%%%%
\section{Normal subgroups of $G_{\mathcal{S}}$}\label{sec:normal-subs}

Let $\mathcal{S} = ((\Omega_i)_{i \in \N},(\phi_i)_{i \geq 2})$ be a flexible $B$-telescope for a perfect group $B$.
The aim of this section is to study the structure of normal and subnormal subgroups of $G_{\mathcal{S}}$ for telescopes of finitely generated non-abelian simple groups.

\subsection{Normal subgroups of $G_{\mathcal{S}}$}

\begin{proposition}\label{prop:structure-of-normal-subgroups}
Let $B$ be a finitely generated, perfect group.
Suppose that there is a flexible $B$-telescope $\mathcal{S} = ((\Omega_i)_{i \in \N},(\phi_i)_{i \geq 2})$ and  assume that each $\Omega_i$ is a finitely generated non-abelian simple group. If the head $Q_\mathcal{S}$ of $G_\mathcal{S}$ is simple, then
 for each normal subgroup $N$ of $G_{\mathcal{S}}$ exactly one of the following holds.
\begin{enumerate}
\item There is a subset $I \subseteq \N$ such that $N = \bigoplus \limits_{i \in I} \Omega_i$.
\item There is some $n \in \N$ and a finite subset $I \subseteq \{1,\ldots,n\}$ such that $N = G_{\mathcal{S}_{+n}} \oplus \bigoplus \limits_{i \in I} \Omega_i$.
\end{enumerate}
\end{proposition}
\begin{proof}
Let $N$ be a normal subgroup of $G_{\mathcal{S}}$.
Suppose first that $N \subseteq \bigoplus \limits_{i \in \N} \Omega_i$.
Then it clearly follows that $N = \bigoplus \limits_{i \in I} \Omega_i$, where $I \subseteq \N$ consists of those numbers $i$ for which the image of the canonical projection $\pi_i \colon N \rightarrow \Omega_i$ is non-trivial.
We may therefore assume that $N$ is not contained in $\bigoplus \limits_{i \in \N} \Omega_i$.
For each $k \in \N$ let $N_k$ denote the normal subgroup of $G_{\mathcal{S}}$ that is generated by $N$ and $\bigoplus \limits_{i=1}^{k} \Omega_i$.
Then $N_{\infty} \defeq \bigcup \limits_{k=1}^{\infty} N_k$ is a normal subgroup of $G_{\mathcal{S}}$ that properly contains $\bigoplus \limits_{i \in \N} \Omega_i$.
Since the head $Q_\mathcal{S}$ is assumed to be simple, it follows that $N_{\infty} = G_{\mathcal{S}}$.
Using our assumption that $B$ and $\Omega_1$ are finitely generated, it follows that $G_{\mathcal{S}}$ is finitely generated.
We therefore obtain $N_{\infty} = N_{k_0}$ for some $k_0 \in \N$.
Hence $G_{\mathcal{S}} = \langle N \cup \bigoplus \limits_{i=1}^{k_0} \Omega_i \rangle$ and it follows from the simplicity of $\Omega_i$ that $K_{k_0}$ is contained in $N$.
Together with Corollary~\ref{cor:group-associated-to-shifted-B-telescope} this gives us
\[
G_{\mathcal{S}_{+k_0}}
= K_{k_0}
\subseteq N
\subseteq G_{\mathcal{S}_{+k_0}} \oplus \bigoplus \limits_{i=1}^{k_0} \Omega_i.
\]
As a consequence, there is a finite set $I \subseteq \{1,\ldots,k_0\}$ such that $N = G_{\mathcal{S}_{+k_0}} \oplus \bigoplus \limits_{i \in I} \Omega_i$, which proves the claim.
\end{proof}

\begin{corollary}\label{cor:fg-direct-factors}
Let $B$ be a finitely generated, perfect group.
Suppose that there is a flexible $B$-telescope $\mathcal{S} = ((\Omega_i)_{i \in \N},(\phi_i)_{i \geq 2})$ and that each $\Omega_i$ is a finitely generated non-abelian simple group. If the head $Q_\mathcal{S}$ is simple, 
then the following are equivalent for a subgroup $N \leq G_{\mathcal{S}}$.
\begin{enumerate}
\item $N$ is a finitely generated normal subgroup of $G_{\mathcal{S}}$.
\item $N$ is a direct factor of $G_{\mathcal{S}}$.
\end{enumerate}
\end{corollary}
\begin{proof}
Since $B$ and $\Omega_1$ are finitely generated it follows that $G_{\mathcal{S}}$ is finitely generated.
Thus every direct factor of $G_{\mathcal{S}}$, being a quotient of $G_{\mathcal{S}}$, is a finitely generated normal subgroup, which gives us $\textit{(2)} \Rightarrow \textit{(1)}$.
Suppose now that $N$ is a finitely generated, normal subgroup of $G_{\mathcal{S}}$.
%From Corollary~\ref{cor:t-group} we know that $N$ is a normal subgroup of $G_{\mathcal{S}}$.
In view of Proposition~\ref{prop:structure-of-normal-subgroups} there are two cases to consider.
Suppose first that there is a subset $I \subseteq \N$ such that $N = \bigoplus \limits_{i \in I} \Omega_i$.
Since $N$ is finitely generated it follows that $I$ is finite.
Thus there is some $n \in \N$ such that $I$ is contained in $I_n \defeq \{1,\ldots,n\}$.
We can therefore use Corollary~\ref{cor:group-associated-to-shifted-B-telescope} to deduce that
\[
G_{\mathcal{S}}
= G_{\mathcal{S}_{+n}} \oplus \bigoplus \limits_{i \in I_n} \Omega_i
= G_{\mathcal{S}_{+n}} \oplus N \oplus \bigoplus \limits_{i \in I_n \setminus I} \Omega_i,
\]
which proves the claim in this case.
Suppose next that there is some $n \in \N$ and a finite subset $I \subseteq I_n$ such that $N = G_{\mathcal{S}_{+n}} \oplus \bigoplus \limits_{i \in I} \Omega_i$.
Then another application of Corollary~\ref{cor:group-associated-to-shifted-B-telescope} gives us
\[
G_{\mathcal{S}}
= G_{\mathcal{S}_{+n}} \oplus \bigoplus \limits_{i \in I_n} \Omega_i
= N \oplus \bigoplus \limits_{i \in I_n \setminus I} \Omega_i,
\]
which proves the corollary.
\end{proof}

\subsection{Residually finite $t$-groups}

A classical theme in group theory is the study of groups whose subnormal subgroups are subject to certain restrictions.
A natural but limiting restriction is to ask that all subnormal subgroups are normal, or equivalently that normality is a transitive relation.
Reflecting the latter characterization, groups satisfying this property are called $t$-groups.
Since their introduction by Best and Taussky~\cite{BestTaussky41}, $t$-groups gained a considerable amount of interest and many efforts have been made in understanding their structure and finding equivalent or related conditions, see e.g.~\cite{Gaschutz57,Robinson64,BrewsterSehgal87,DardanoDeMari21}.
Despite of many structural properties that can be gained for general $t$-groups and related groups, essentially all examples studied in the literature appear to be solvable.
One reason being that solvable groups guarantee to have many subnormal subgroups.
Another one was pointed out by Robinson~\cite[Section 13.4]{Robinson96} in the case of finite groups.
Using the affirmative solution of the famous Schreier conjecture that the outer automorphism group of a finite simple group is solvable,
%which is based on the classification of non-abelian finite simple groups,
he showed that every finite, perfect $t$-group that has no non-trivial abelian normal subgroups is already a product of non-abelian finite simple groups.
Nevertheless, the following result shows that the class of finitely generated $t$-groups also contains examples of infinite groups that have many normal subgroups and yet are far from being solvable.

\begin{theorem}\label{thm:t-group}
Let $B$ be a finitely generated, perfect group.
Suppose that there is a flexible $B$-telescope $\mathcal{S} = ((\Omega_i)_{i \in \N},(\phi_i)_{i \geq 2})$ and that each $\Omega_i$ is a finitely generated non-abelian simple group. If the head $Q_\mathcal{S}$ is simple, then
 $G_{\mathcal{S}}$ is a $t$-group.
\end{theorem}
\begin{proof}
Let $N$ be a normal subgroup of $G_{\mathcal{S}}$ and let $M$ be a normal subgroup of $N$.
We have to show that $M$ is normal in $G_{\mathcal{S}}$.
Suppose first that $N = \bigoplus \limits_{i \in I} \Omega_i$ for some subset $I \subseteq \N$.
Since each $\Omega_i$ is simple it follows that $M = \bigoplus \limits_{i \in J} \Omega_i$ for some subset $J \subseteq I$, which shows that $M$ is normal in $G_{\mathcal{S}}$.
In view of Proposition~\ref{prop:structure-of-normal-subgroups} we can therefore assume that there is some $n \in \N$ and a finite subset $I \subseteq \{1,\ldots,n\}$ such that $N = G_{\mathcal{S}_{+n}} \oplus \bigoplus \limits_{i \in I} \Omega_i$.
Using again the fact that each $\Omega_i$ is non-abelian and simple, we deduce that $M = M_1 \oplus \bigoplus \limits_{i \in J} \Omega_i$, where $J \subseteq I$ and $M_1$ is a normal subgroup of $G_{\mathcal{S}_{+n}}$.
It therefore remains to show that $M_1$ is normal in $G_{\mathcal{S}}$.
From the assumptions on $\mathcal{S}$ it directly follows that $\mathcal{S}_{+n}$ is a flexible $B$-telescope. The head $Q_{\mathcal{S}_{+n}}$ is canoncially isomorphic to $Q_\mathcal{S}$ and is hence simple.
We can therefore apply Proposition~\ref{prop:structure-of-normal-subgroups} to deduce that either $M_1 = \bigoplus \limits_{i \in J} \Omega_i$ for an appropriate subset $J \subseteq \N_{\geq n+1}$ or $M_1 = G_{\mathcal{S}_{+m}} \oplus \bigoplus \limits_{i \in J} \Omega_i$ for some $m \in \N$ and $J \subseteq \{n+1,\ldots,m\}$.
Now the claim follows since in both cases $M_1$ is normal in $G_{\mathcal{S}}$.
\end{proof}

Consider the $\Alt(rd)$-telescope $\mathcal{A}(d,r)$.
Using Theorem~\ref{thm:frame-minimal-Alt}, we can now deduce the following exotic property of $G_{\mathcal{A}(d,r)}$.

\begin{corollary}\label{cor:fg-subnormal-2-generated}
Let $N$ be a subnormal subgroup of $G_{\mathcal{A}(d,r)}$.
If $N$ is finitely generated, then $N$ is $2$-generated.
%Every finitely generated subnormal subgroup of $G_{\mathcal{A}(d,r)}$ is $2$-generated.
\end{corollary}
\begin{proof}
%Let $N$ be a subnormal subgroup of $G_{\mathcal{A}(d,r}$.
From Corollary~\ref{cor:fg-direct-factors} we know that $N$ is a direct factor of $G_{\mathcal{A}(d,r)}$.
As such $N$ is a quotient of $G_{\mathcal{A}(d,r)}$ and we can deduce from Theorem~\ref{thm:frame-minimal-Alt} that $N$ is $2$-generated.
\end{proof}

%From Proposition~\ref{prop:structure-of-normal-subgroups} and the notion of shifted $B$-telescopes, it is now easy to deduce that $G_{\mathcal{S}}$ is one of the rare infinite examples of residually finite, finitely generated $\T$-groups, i.e.\ groups in which normality is a transitive relation.
%In fact the authors are not aware of other finitely generated, residually finite, infinite $\T$-groups that are not virtually nilpotent.

\section{Direct factors of profinite completions}\label{sec:direct-factors-prof-completions}
Let $G$ be a finitely generated, residually finite group and let $N$ be a finitely generated, normal subgroup of $G$.
Let $\iota \colon N \rightarrow G$ denote the inclusion map and let $\widehat{\iota} \colon \widehat{N} \rightarrow \widehat{G}$ be the induced map on the profinite completions.
In this setting, Goldstein and Guralnick~\cite[Question 3.1]{GoldsteinGuralnick06} raised the following question.

\begin{question}[Goldstein and Guralnick]\label{quest:Goldstein-Guralnick}
Suppose that the image of $\widehat{N}$ under $\widehat{\iota}$ is a direct factor of $\widehat{G}$.
Does it follow that $N$ is a direct factor of $G$?
\end{question}

In the case where $G$ is virtually polycyclic, a positive solution of Question~\ref{quest:Goldstein-Guralnick} was provided by Nikolov and Segal~\cite[Theorem 1]{NikolovSegal07}.
In the same work they also constructed a counterexample for the general case, which is based on the methods in~\cite{KassabovNikolov06}.
Finitely presented examples were given by Bridson~\cite{Bridson09}.
Here we want to present a new counterexample that, as we will see later, is amenable.
To the best of our knowledge, all previously known counterexamples to Question~\ref{quest:Goldstein-Guralnick} contain a non-abelian free group.
Our approach is based on the two following observations. 

\begin{lemma}\label{lem:inner-not-inner}
Let $N$ be a finitely generated residually finite group. If $\tau$ is a finite order automorphism of $N$ that is not inner, but the induced automorphism $\widehat{\tau}$
on $\widehat{N}$ is inner, then $G = N \rtimes \langle \tau \rangle$ provides a counterexample to Question \ref{quest:Goldstein-Guralnick}.
\end{lemma}
\begin{proof}
Since $N$ has finite index in $G$, we have $\widehat{G} = \widehat{N} \rtimes \langle \widehat{\tau} \rangle$.
By assumption $\widehat{\tau}$ is given by conjugation with $\gamma \in \widehat{N}$ and hence $\widehat{G}$ decomposes as a direct product
$\widehat{G} = \widehat{N} \times \langle (\gamma^{-1},\tau) \rangle$.
On the other hand, $\tau$ is not inner and hence there is no element $(n,\tau) \in G$ that centralizes $N$. 
\end{proof}
Despite of the fact that the existence of a single frame in an infinite product of finite non-abelian simple groups was difficult to achieve, it turns out that the property of being a frame is quite stable under combination.
This is demonstrated by the following very useful lemma, see~\cite[Lemma 2.2]{KassabovNikolov06}.
\begin{lemma}[Kassabov, Nikolov]\label{lem:gluing}
Let $\prod \limits_{n \in \N} C_n$ be a product of finite groups.
For each $n \in \N$ let $A_n,B_n \leq C_n$ be subgroups with $\langle A_n,B_n \rangle = C_n$.
Suppose that $G$ is a frame for $\prod \limits_{n \in \N} A_n$ and that $H$ is a frame for $\prod \limits_{n \in \N} B_n$.
Then $K = \langle G \cup H \rangle$ is a frame in $\prod \limits_{n \in \N} C_n$.
\end{lemma}

Let $d \geq 5$ and let $r \leq d-3$ be given.
Consider the telescope $\mathcal{A}(d,r)$ and the associated frame $G_{\mathcal{A}(d,r)} \subseteq \prod_{\ell \in \N } \Alt(X^\ell) = P$, where $X = \{x_1,x_2,\dots,x_d\}$.
We define $\delta \in \Alt(X)$ as $\delta = (x_1, x_d)(x_2,x_{d-1})$ and put further $\epsilon_\ell = \delta \times \cdots \times \delta \in \Alt(X^\ell)$.
Then $\epsilon = (\epsilon_1,\epsilon_2,\epsilon_3,\dots) \in P$ is an element of order two, which will be used now in order to apply Lemma~\ref{lem:inner-not-inner}.
%Let us now formulate the main result of this section.

\begin{theorem}\label{thm:counterexample-goldberg-guralnick}
The group  $N = \langle G_{\mathcal{A}(d,r)}, G_{\mathcal{A}(d,r)}^{\varepsilon} \rangle \subseteq P$
 is a $4$-generated residually finite group with $\widehat{N} = P$. Let $\tau$ be the automorphism of $N$ induced by conjugation with $\epsilon$. Then $\tau$ is not inner, but the induced automorphism $\widehat{\tau}$ on $P$ is inner.
In particular, $N \rtimes \langle \tau \rangle$ is a counterexample to Question~\ref{quest:Goldstein-Guralnick}
\end{theorem}
\begin{proof}
By Theorem \ref{thm:frame-minimal-Alt} the group $G_{\mathcal{A}(d,r)}$ is a $2$-generated frame, so $N$ can be generated with $4$ elements. The group $N$ is residually finite, because it is a subgroup of $P$.
It follows from Lemma \ref{lem:gluing} that $N$ is a frame in $P$ and hence $\widehat{N} \cong P$ via the natural inclusion.
%Recall from Section \ref{sec:simplicity-alt} that 
Call a point $w \in X^{i}$ \emph{consistent} for an element $g \in P$, if $\pi_{i+n}(g)(wy) = \pi_i(g)(w)y$ for all $n \in \N$ and all $y \in X^n$.
Note that $\epsilon$ does not admit any consistent points.
The ratio of the number of consistent points in $X^{i}$ by $d^{i}$ will be called the \emph{consistency of $g$} at level $i$ and is denoted by $\cons_i(g)$. It is easy to check that consistency is invariant under conjugation with $\epsilon$.
We observe that for all $g,h \in P$
 \[
 	\cons_i(gh) \geq \cons_i(g)+ \cons_i(h) -1
 \]
  for all $i$. Indeed, if $w \in X^{i}$ is consistent for $h$ and $h(w)$ is consistent for $g$, then $w$ is consistent for $gh$.
By construction the elements in $\Delta_i(\Alt(X^{i}))$ are consistent at all points in $X^{i}$.
For $c \in \Alt(rd)$ the element $\tilde{c}^{[1]}$ is consistent at all points in $X^{i}$ except for the $r+1$ points $x_d^{i}, x_d^{i-1}x_1, \dots, x_d^{i-1}x_r$. 
 In particular, $\lim_{i \to \infty} \cons(\tilde{c}^{[1]}) = 1$ and by induction on the word length the formula above shows that this holds true for all elements in $G_{\mathcal{A}(d,r)}$ and all elements in $N$.
 We conclude that $\epsilon$ does not lie in $N$. The group $N$ is dense in $P$ and $P$ has trivial center, therefore this implies that $\tau$ is not inner.
\end{proof}
If one starts with another non-trivial element $\delta$, one can achieve that $\tau$ has any given finite order.
The amenability of the group $N \rtimes \langle \tau \rangle$ in Theorem~\ref{thm:counterexample-goldberg-guralnick} is a consequence of a criterion of Juschenko, Nekrashevych and de la Salle~\cite{JuschenkoNekrashevychdelaSalle16} and will be verified in a more general context in \S~\ref{subsec:amenability-of-embeddings}, which is why we omit it here. 
\begin{remark}
One can show that $N$ is the group $G_{\mathcal{S}}$ for a $B$-telescope $\mathcal{S}$ with $B = \Alt(rd)\times \Alt(rd)$.
\end{remark}

\section{Actions of telescopes}\label{sec:actions}
\begin{definition}\label{def:action}
Let $\mathcal{S} = ((\Omega_i)_{i \in \N},  (\phi_i)_{i \geq 2})$ be a $B$-telescope.
 Let $(X_i)_{i\in \N}$ be a directed sequence of sets with compatible maps
\[
	\iota_{i,j}\colon X_i \to X_j
\]
for all $i < j$, such that
each $X_i$ is an $\Omega_i$-set and the map $\iota_{i,j}$ satisfies
\begin{equation}\label{eq:compatibility}
	\iota_{i,j}(\omega.x) = \iota_{i,j}(\omega).\iota_{i,j}(x)
\end{equation}
for all $i<j$, all $\omega \in \Omega_i$, and all $x \in X_i$. On the right hand side $\iota_{i,j}(\omega)$ refers to the transition maps in the $B$-telescope and $\iota_{i,j}(x)$ refers to the transition maps of the system $((X_i),(\iota_{i,j})_{i<j})$.
The image of $X_i$ in $X_j$ will be denoted by $X_{i,j}$.
We say that $\mathcal{S}$ \emph{acts} on $(X_i)_{i\in \N}$ if there is a $\kappa\geq 2$ such that $B_{i+n}$ acts trivially on $X_{i,i+n}$ for all $n \geq \kappa$ and every $i \in \N$.
\end{definition}

For simplicity $\iota_{i,i}$ denotes the identity map on $X_i$.
We define $X_\infty = \varinjlim_{i\in \N} X_i$. The image of $x \in X_i$ in $X_\infty$ will be denoted by $[x]$ and we write $X_{i,\infty}$ for the subset $\{[x] \mid x \in X_i\} \subseteq X_\infty$. Observe that $X _\infty = \bigcup_{i=1}^\infty X_{i,\infty}$.
We say that a sequence $(y_i)_i$ in a set $Y$ \emph{stabilizes} if there is some $t\geq 1$ such that $y_i = y_t$ for all $i \geq t$. In this case we write $\lim_{i\to \infty} y_i$ to denote the stable value $y_t$. 
\begin{lemma}\label{lem:action}
Suppose that $\mathcal{S}$ acts on $(X_i)_{i\in \N}$. Then for every  $\gamma \in \widetilde{G}_\mathcal{S}$ and every $x \in X_i$ there is some $t\geq i$ such that 
\begin{equation}\label{eq:stability}
\pi_{k}(\gamma). \iota_{i,k}(x) = \iota_{t,k}(\pi_{t}(\gamma). \iota_{i,t}(x))
\end{equation}
 for all $k \geq t$. In particular, the sequence 
$[\pi_{i+k}(\gamma). \iota_{i,i+k}(x)]$ stabilizes
and
\[
	\gamma.[x] = \lim_{k\to \infty} [\pi_{i+k}(\gamma). \iota_{i,i+k}(x)]
\]
defines an action of the huge group $\widetilde{G}_\mathcal{S}$ on $X_\infty$.
\end{lemma}
\begin{proof}
Suppose that $\gamma_1, \gamma_2 \in \widetilde{G}_\mathcal{S}$ satisfy \eqref{eq:stability} for all $i \geq 1$ and all $x \in X_i$.  Then $\gamma_1\gamma_2$ has the same property. Indeed, let $x \in X_i$ be given and let $t \geq i$ be such that $\iota_{t,k}(\pi_{t}(\gamma_2). \iota_{i,t}(x)) = \pi_{k}(\gamma_2). \iota_{i,k}(x)$ and  for all $k \geq t$. Put $y = \pi_{t}(\gamma_2). \iota_{i,t}(x)$.
Let $s \geq t$ be such that $\iota_{s,k}(\pi_{s}(\gamma_1). \iota_{t,s}(y)) = \pi_{k}(\gamma_1). \iota_{t,k}(y)$ for all $k \geq s$. Then
\begin{align*}
   \iota_{s,k}(\pi_{s}(\gamma_1). \iota_{t,s}(y)) = \pi_{k}(\gamma_1). \iota_{t,k}(y) = \pi_k(\gamma_1).\pi_k(\gamma_2).\iota_{i,k}(x)
   = \pi_k(\gamma_1\gamma_2).\iota_{i,k}(x)
\end{align*}
holds for all $k \geq s$. In particular, it is sufficient to prove \eqref{eq:stability} for a set of generators of $\widetilde{G}_\mathcal{S}$. Once this has been verified, the same calculation implies directly that $\gamma.[x] = \lim_{k\to \infty} [\pi_{i+k}(\gamma). \iota_{i,i+k}(x)]$ defines an action of $\widetilde{G}_\mathcal{S}$ on $X_\infty$.

Let  $\omega \in \Omega_i$ and let $x \in X_i$.
By the compatibility conditions \eqref{eq:compatibility} we have $\iota_{i,i+k}(\omega).\iota_{i,i+k}(x) = \iota_{i,i+k}(\omega.x)$ for all $k \geq 0$.
Let $g \in B$. By assumption $\phi_{k}(g)$ acts trivially on $X_{i,k}$ if $k \geq i+\kappa$.
In particular,
\begin{align*}
        \iota_{i+\kappa,k}(\pi_{i+\kappa}(\tilde{g}).\iota_{i,i+\kappa}(x)) &= \iota_{i+\kappa,k}\Bigl(\prod_{j=2}^{i+\kappa}\iota_{j,i+\kappa}(\phi_j(g)).\iota_{i,i+\kappa}(x)\Bigr)\\
        &= \prod_{j=2}^{i+\kappa}\iota_{j,k}(\phi_j(g)).\iota_{i,k}(x) & \text{\small (by \eqref{eq:compatibility})} \\
        &=\prod_{j=2}^{k}\iota_{j,k}(\phi_j(g)).\iota_{i,k}(x)\\
        &= \pi_{k}(\tilde{g}).\iota_{i,k}(x)
\end{align*}
for all $k \geq i+\kappa$.
Now the Lemma follows since by definition $\widetilde{G}_\mathcal{S} = \langle \{\Delta_i(\Omega_i) \mid i \in \N \} \cup C^{[1]} \rangle$.
\end{proof}

The elements $\Delta_i(\omega)$ with $\omega \in \Omega_i$ act in the obvious way by
\[
	\Delta_i(\omega).[x] = [\iota_{i,j}(\omega).x]
\]
for $x \in X_j$ with $j\geq i$.
Let $g \in B$ and let $x \in X_i$, then $\tilde{g}^{[n]}$ acts by
\[
	\tilde{g}^{[n]}.[x] :=  [\pi_{i+\kappa}(\tilde{g}^{[n]}).\iota_{i,i+\kappa}(x)].
\]
 \begin{remark}
 In general, the action of $\widetilde{G}_\mathcal{S}$ on $X_\infty$ is not faithful.
 The elements of $\bigoplus_{i \in \N} \Omega_i$ are contained in $\widetilde{G}_\mathcal{S}$ and act trivially on $X_\infty$. Therefore the action always factors through the head $Q_\mathcal{S}$.
In particular, if $\mathcal{S}$ is a flexible $B$-telescope for a perfect group $B$, then 
by Corollary \ref{cor:group-associated-to-shifted-B-telescope} the image of $G_\mathcal{S}$ in $\Sym(X_\infty)$ agrees with the image of the group of any shifted telescope $\mathcal{S}_{+n}$.
 \end{remark}
 
We observe that  every element of $\widetilde{G}_\mathcal{S}$ acts locally like an element $\Delta_j(\Omega_j)$ for some $j \in \N$.
More precisely, for every  $\gamma \in \widetilde{G}_\mathcal{S}$ and every $x \in X_i$, there is 
some $j \geq i$ and an element $\omega \in \Omega_j$ such that 
$\Delta_j(\omega)^{-1}\gamma$ stabilizes $[x] \in X_\infty$.
% \begin{proof}
% By definition of the action, we have
% $\gamma.[x] = \lim_{k\to \infty} [\pi_{i+k}(\gamma). \iota_{i,i+k}(x)]$ and so there is some $j \geq i$ such that $y = \gamma.[x] = [\pi_j(\gamma)\iota_{i,j}(x)]$. Define $\omega = \pi_j(\gamma) \in \Omega_j$, then 
% \[\Delta_j(\omega)^{-1}\gamma.[x] = [\omega^{-1}.\pi_j(\gamma)\iota_{i,j}(x)] = [\iota_{i,j}(x)] = [x].\qedhere\]
% \end{proof}

\begin{example}\label{ex:subset-action}
Natural constructions of group actions can be extended to $B$-telescopes.
We illustrate this with an easy example.
Let $\mathcal{S}$ denote a $B$-telescope that acts on $((X_i)_{i\in \N},(\iota_{i,j})_{i<j})$ with injective transition maps. Assume $|X_1| \geq k$. Let $Y_i$ denote the set of $k$-element subsets of $X_i$. The action of $\Omega_i$ on $X_i$ induces an action of $\Omega_i$ on $Y_i$.
Moreover, the injective inclusion maps $\iota_{i,j}\colon X_i \to X_j$ induce injective maps $\iota'_{i,j} \colon Y_i \to Y_j$ that are compatible with the actions of $\mathcal{S}$.
We observe that $Y_{i,j}$ is the set of $k$-element subsets of $X_{i,j}$. By assumption $B_{i+n}$ acts trivially on $X_{i,i+n}$ for $n\geq \kappa$ and thus $B_{i+n}$ acts trivially on $Y_{i,i+n}$. In total, we have an action of $\mathcal{S}$ on $(Y_i)_{i \in \N}$.
\end{example}
\begin{lemma}\label{lem:primitive-action}
Let $\mathcal{S}= ((\Omega_i)_{i \in \N},  (\phi_i)_{i \geq 2})$ be a flexible $B$-telescope for a perfect group $B$. Assume that $\mathcal{S}$ acts on $(X_i)_{i \in \N}$. 
Suppose there is $n_0\geq 1$ such that the action of $\Omega_i$ on $X_i$ is primitive (resp.\ $k$-transitive) for all $i \geq n_0$. Assume that $X_\infty$ has at least three elements (resp.\ at least $k$ elements). Then the action of $G_{\mathcal{S}}$
on $X_\infty$ is primitive (resp. $k$-transitive).
\end{lemma}
\begin{proof}
Assume there is some $k \geq 1$ such that the action of $\Omega_i$ on $X_i$ is $k$-transitive for all $i \geq n_0$.
By assumption $X_\infty$ has at least $k$ elements. Let $[x_1],[x_2],\dots,[x_k]$ and $[y_1],[y_2],\dots,[y_k]$ be two $k$-tuples of distinct elements in $X_\infty$.
Without loss of generality, we may assume that $x_1,\dots,x_k, y_1,\dots,y_k \in X_i$ for some $i \geq n_0$. Since the images in $X_\infty$ are distinct, we have $x_s \neq x_t$ and $y_s \neq y_t$ for all $s \neq t$.
Using the $k$-transitivity of the action of $\Omega_i$ on $X_i$, we see that there is an element $g \in \Omega_i$ with $g.x_s = y_s$ for all $s \in \{1,2,\dots,k\}$.
From the definition of the action we conclude that $\Delta_i(g)[x_s] = [y_s]$.
Since the groups $\Delta_{i}(\Omega_i)$ are contained in $G_\mathcal{S}$ by Lemma~\ref{lem:G-subgroups}, it follows that the action of $G_{\mathcal{S}}$ on $X_\infty$ is $k$-transitive.

Suppose next that $X_\infty$ contains at least $3$ elements and that the action of $\Omega_i$ on $X_i$ is primitive for $i \geq n_0$.
Assume that $Y \subsetneq X_\infty$ is a subset with $|Y| \geq 2$. We will show that $Y$ cannot be a block for the action of $G_{\mathcal{S}}$.
Since $X_\infty = \bigcup_{i=1}^\infty X_{i,\infty}$ there is some $i \in \N$ such that $Y_i = Y \cap X_{i,\infty}$ has at least two elements and is different from $X_{i,\infty}$. Since the action of $\Omega_i$ is primitive and $|X_i| \geq 3$, the equivariant map $X_i \to X_{i,\infty}$ is injective. By primitivity there are elements $g \in \Omega_i$ and $y,z \in Y_i$ with
\[
	\Delta_i(g).y \not \in Y_i \quad \text{ and } \quad \Delta_i(g).z \in Y_i \subseteq Y.
\]
Since $X_i$ is invariant under the action of $\Delta_i(\Omega_i)$, we conclude that $\Delta_i(g).y \not\in Y$. Again Lemma~\ref{lem:G-subgroups} implies that the groups $\Delta_{i}(\Omega_i)$ are contained in $G_\mathcal{S}$ and consequently $Y$ cannot be a block for the action of $G_\mathcal{S}$ on $X_\infty$, i.e., the action is primitive.
\end{proof}
We say that $\mathcal{S}$ acts \emph{primitively} on $(X_i)_{i \in \N}$ if $|X_\infty| \geq 3$ and the action of $\Omega_i$ on $X_i$ is primitive for all but finitely many $i$.
Observe that in this case for large $i$ the set $X_{i,\infty}$ is an $\Omega_i$-equivariant image of the primitive $\Omega_i$-set $X_i$. Therefore, it is either trivial or the map $X_i \to X_{i,\infty}$ is injective. Using $|X_\infty| \geq 3$ it follows that $|X_i| \geq 3$ and that the inclusions $\iota_{i,i+1}\colon X_i \to X_{i+1}$ are injective for all sufficiently large $i$. In particular, the action of $\Omega_i$ on $X_i$ is transitive for all large $i$.
\begin{example}\label{ex:alt-action}
 Let $d \geq 5$ and let $2 \leq r \leq d-3$ be given.
We describe an action of the telescope $\mathcal{A}(d,r) = ((\Alt(d^{i})_{i\in \N}), (\phi_n)_{n\geq 2})$ defined in Section \ref{ex:Alt}. Recall that $X = \{x_1,\dots,x_d\}$ is a set with $d$ elements. We think of $X^{i}$ as the set of words of length $i$ in the letters $X$. For $j > i$ we define the transition maps
\[
	\iota_{i,j} \colon X^i \to X^{j} \quad u \mapsto ux_{d}^{j-i},
\] 
i.e., $\iota_{i,j}$ appends $x_{d}^{j-i}$ to the word $u$.
Note that $X_{i,j} = X^{i} \times \{x_d^{j-i}\}$.
For all $\sigma \in \Alt(X^{i})$ and $u \in X^{i}$ we have
\[
	\iota_{i,j}(\sigma).\iota_{i,j}(u) = \iota_{i,j}(\sigma).ux_{d}^{j-i} = \sigma(u)x_{d}^{j-i} = \iota_{i,j}(\sigma.u),
\]
which shows that these transition maps are compatible in the sense of \eqref{eq:compatibility} with the transition maps in $\mathcal{A}(d,r)$.
We claim that this is an action of $\mathcal{A}(d,r)$ in the sense of Definition \ref{def:action}.
Recall that the support of
$B_{i+n}$ is $\{x_d^{i+n-2}\}\times \{x_1,\dots,x_r\} \times X$; see Lemma~\ref{lem:supports}. In particular, using $r \leq d-3 < d$ we see that for $n \geq 2$ the support of $B_{i+n}$ is disjoint from $X_{i,i+n} = X^{i} \times \{x_{d}\}^{n}$ and $B_{i+n}$ acts trivially on $X_{i,i+n}$.

\smallskip

\noindent The transition maps $X_i \to X_j$ are injective and the cardinality of these sets is unbounded, hence $X_\infty$ is an infinite set.
Since the action of $\Alt(d^{i})$ on $d^{i}$ points is $d^{i}-2$ transitive, it follows from Lemma \ref{lem:primitive-action}
that the action of $G_\mathcal{S}$ on $X_\infty$ is highly transitive, i.e., $k$-transitive for every $k$.
 In particular, $\St_{G_\mathcal{S}}(x)$ is a maximal subgroup for every $x \in X_\infty$. 
 
Using the criterion of Nash-Williams one can show that the action of $G_{\mathcal{A}(d,r)}$ on $X_\infty$ is recurrent.
\end{example}

 %
%\begin{proposition}
%Let $d \geq 5$ and let $2 \leq r \leq d-3$ be given.
%The action of $G_{\mathcal{A}(d,r)}$ on $X_\infty$ is recurrent.
%\end{proposition}
%%
%\begin{proof}
%We use the criterion of Nash-Williams \cite{Nash-Williams} in the form given in \cite[Theorem 5.8]{Juschenko22}.
%Fix finite generating sets of $\Delta_1(\Omega_1)$ and $C^{[1]}$.
%We use the sequence of sets $X_{i,\infty}$; their union is $X_\infty$. We think of $X_{i,\infty}$ as infinite words $\xi = \xi_1\xi_2\xi_3\dots$ such that $\xi_j = x_d$ for all $j> i$. The sets $X_{i,\infty}$ are preserved by every element in $\Delta_1(\Omega_1)$. An element $\tilde{g} \in C^{[1]}$ can move element  $\xi \in X_{i,\infty}$ to $X_{i+1,\infty}$ only if $\xi_{j}=x_d$ for all $j < i$ and $\xi_{i} \in \{x_1,x_2\}$. In other words, $\partial X_{i,\infty}$, the set of points incident to a point in $X_{i+1,\infty}$, contains at most $2$ elements. The series
%\[
%	\sum_{i=1}^\infty \frac{1}{\partial X_{i,\infty}} = \sum_{i =1}^\infty \frac{1}{2}
%\]
%diverges and hence the action is recurrent.
%\end{proof}

Let $F$ be a field and let $d \geq 4$. Let us consider the $\mathrm{SL}_{2d}(F)$-telescope $\mathcal{SL}_d(F) = ((\Omega_n)_{n\in \N}, (\phi_n)_{n\geq 2})$ defined in Section~\ref{ex:elementary}. Recall that $\Omega_{n} = \SL_{X^n}(F) \cong \SL_{d^n}(F)$, where the matrix coordinates are indexed by $X^n$
 with $X = \{x_1,\dots,x_d\}$, i.e, $\SL_{X^n}(F)$ acts on the vector space $V_n := F[X^n]$ spanned by a basis identified with $X^n$.
 We note that $V_{n+k}$ is canonically isomorphic to $V_n \otimes_F F[X^k]$ and we will tacitly use this identification.
  \begin{example}\label{ex:action-sl}
Consider the injective linear map
 \[
 	\iota'_{i,i+n}\colon V_i  \to V_{i+n},
 \]
which maps a vector $u \in V_i$ to the vector $u\otimes x_{3}^n$.
The system $((V_i),\iota_{i,i+n})$ is compatible with the transition maps of the telescope $\mathcal{SL}_d(F)$. 
We claim that this defines an action of $\mathcal{SL}_d(F)$ on $(V_n)_{n\in \N}$. 
To this end, let $i \geq 2$ be given and fix some $n \geq 2$.
As we have seen in the proof of Proposition \ref{prop:B-telescope-E}, the group $B_{i+n}$ is supported on $Z(i+n,i+n)$.
Thus it remains to note that $Z(i+n,i+n)$ does not contain any of the vectors in $V_{i,i+n} = F[X^{i}\times\{x_{3}^n\}]$. 
%%%
%Similarly, let $V_i^*$ denote the dual space of $V_i$ and let $(X^{i})^* = \{ z^* \mid z \in (X^{i})^*\}$ denote the dual basis to $X^{i}$. $V_n^*$ is equipped with the natural action of $\GL_{X^n}(F)$.
%Then $V^*_{i+n} = V_i^*\otimes V_n^*$ canonically and the linear maps
%\[
% 	\iota'_{i,i+n}\colon V^*_i  \to V^*_{i+n},
% \]
% defined by $\iota'_{i,i+n}(u) = u \otimes (x_{3}^n)^*$ are equivariant with respect to the telescope $\mathcal{E}(d)$. One can check that this defines an action of $\mathcal{E}(d)$ on $((V^*_i),\iota_{i,i+n})$.
 \end{example}
\begin{example}\label{ex:action-projectivespace}
 The actions we have just defined give rise to actions of the telescope $\mathcal{PSL}_d(F) = ((P\Omega_n)_{n\in N}, (\phi_n)_{n\geq 2})$ from Example \ref{ex:PSL}.
We define $P_n$ to be the projective space of $V_n = F[X^n]$ and we note that $\PSL_{X^n}(F)$ acts on $P_n$. Using these actions, we can construct a faithful action of $Q_{\mathcal{PSL}_d(F)}$.
The maps $\iota'_{i,i+n}$ defined above induce a map $\iota_{i,i+n}\colon P_i \to P_{i+n}$ between the projective spaces. The system $((P_i),\iota_{i,j})$ is compatible with the transition maps of the telescope $\mathcal{PSL}_d(F)$. We claim that this defines an action of $\mathcal{PSL}_d(F)$ on $(P_i)_{i\in\N}$.   It suffices to check that $B_{i+n}$ acts trivially on $P_{i,i+n}$. As above, for 
 $i \geq 2$ and $n \geq 2$
 the group $B_{i+n}$ is supported on $Z(i+n,i+n)$; see Proposition \ref{prop:B-telescope-E}. 
 By construction $P_{i,i+n}$ contains all lines spanned by vectors in the subspace $F[X^{i}\times\{x_{3}^n\}]$.
Consequently, $B_{i+n}$ acts trivially on $P_{i,i+n}$.
 The action of $\PSL_{X^n}(F)$ on $P_n$ is $2$-transitive and in particular primitive.
This implies that the action of $Q_{\mathcal{PSL}_d(F)}$ on $P_\infty$ is also $2$-transitive and that stabilizers are maximal subgroups.
We will see later that the action is faithful.  
\end{example}
 
%%%%%%%%%%%%%%%%%%%%%%%%%%%%%%%
\section{Simplicity of the head for $\mathcal{A}(d,r)$}\label{sec:simplicity-alt}
Here we consider only the telescope $\mathcal{A}(d,r)$ for $d \geq 5$ and $2 \leq r \leq d-3$. We will prove the following result.
\begin{theorem}\label{thm:simplehead-alt}
Let $d \geq 5$ and let $2 \leq r \leq d-3$ be given.
The head $Q_{\mathcal{A}(d,r)} = G_{\mathcal{A}(d,r)}/\bigoplus_{\ell=1}^\infty \Alt(X^{\ell})$ is a $2$-generated infinite simple group.
\end{theorem}

We have seen in Theorem \ref{thm:frame-minimal-Alt} that $G_{\mathcal{A}(d,r)}$ is $2$-generated. It remains to prove the simplicity of the head.
The key step in doing so will be to show that for elements $g \in  G_{\mathcal{A}(d,r)}$ not contained $\bigoplus_{i=1}^\infty \Alt(X^{i})$ the projections $\pi_n(g)$ are non-trivial for all large $n$. In fact, we will show a bit more and we will see that such elements move a substantial fraction of the points in $X^n$.
Our tool to verify this claim is a weak normal form theorem, that we are going to explain first.

\smallskip

\noindent Let  $m \geq 0$ be given. For all $i \geq m+2$ we define the embedding 
\[\phi^{(m)}_i\colon \Alt(X^m\times \{x_1,\dots,x_r\} \times X) \to \Alt(X^{i})\] 
using the canonical identification of the sets $X^m\times\{x_1,\dots,x_r\} \times X$ and $X^m\times\{x_d^{i-m-2}\}\times \{x_1,\dots,x_r\} \times X$.
Note that for $m=0$ the homomorphisms $\phi_i^{(0)} = \phi_i$ are the structure maps of the telescope. We observe that the images of $\iota_{i,j}\circ\phi^{(m)}_i$ and $\phi^{(m)}_j$ have disjoint supports for $i < j$ and hence commute.
Now we imitate the definition of $\tilde{b}^{[n]}$ for elements $b \in \Alt(\{x_1,\dots,x_r\}\times X)$.
Let $n > m$. For $\sigma \in \Alt(X^m\times \{x_1,\dots,x_r\} \times X)$ we define $\tilde{\sigma}^{[n]}  \in \prod_{i=1}^\infty \Alt(X^{i})$ by $\pi_i(\tilde{\sigma}^{[n]}) = 1$ for all $i <n+1$ and by
\[
	\pi_i(\tilde{\sigma}^{[n]})  = \prod_{k=n+1}^{i} \iota_{k,i}\phi^{(m)}_k(\sigma)
\]
for $i \geq n+1$.
This gives the recursive description 
\begin{equation}\label{eq:tilde-recursive}
 \tilde{\sigma}^{[n]} = \Delta_{n+1}(\phi^{(m)}_{n+1}(\sigma)) \tilde{\sigma}^{[n+1]} =  \tilde{\sigma}^{[n+1]}\Delta_{n+1}(\phi^{(m)}_{n+1}(\sigma)).
\end{equation}

\begin{lemma}\label{lem:weak-normalform}
Let $g \in G_{\mathcal{A}(d,r)}$. There are $m \geq 0$, $\sigma \in \Alt(X^m\times\{x_1,\dots,x_r\}\times X)$, $\delta, \eta \in \Delta_{m+1}(\Alt(X^{m+1}))$ and $f \in \bigoplus_{i=1}^m \Alt(X^{i})$ 
such that
\[
	g = f \delta \tilde{\sigma}^{[m+1]} \eta.
\]
The number $m$ is at most the word length of $g$ with respect to the generating set $\Delta_1(\Alt(X)) \cup C^{[1]}$.
\end{lemma}
\begin{proof}
Let $g \in G_{\mathcal{A}(d,r)}$ and let $\ell$ denote the word length of $g$ with respect to $\Delta_1(\Alt(X)) \cup C^{[1]}$.
We proceed by induction on the word length. In every step the number $m$ increases by at most $1$.
For $\ell = 0$ there is nothing to do. Assume that $\ell = 1$. If $g = \Delta_1(\omega)$, then we set $m=0$ and $\delta = g$ and $\eta = f = \sigma = 1$.
If $g = \tilde{b}^{[1]}$, then we set $m=0$, $\sigma = b$ and $\delta = \eta = f = 1$.

Assume that $\ell > 1$. We write $g = h g'$ with $h \in \Delta_1(\Alt(X)) \cup C^{[1]}$ and $g'$ of length $\ell -1$.
By induction hypothesis there is $m \leq \ell-1$ such that we may write $g' = f \delta \tilde{\sigma}^{[m+1]} \eta$ as above.
We note that $\bigoplus_{i=1}^m \Alt(X^{i})$ is normal, so $hf = f_1h$ with $f_1\in \bigoplus_{i=1}^m \Alt(X^{i})$.

If $h = \Delta_1(\omega)$, then $h = f_2 \Delta_{m+1}(\omega)$ with $f_2 \in \bigoplus_{i=1}^m \Alt(X^{i})$ and we have
\[
	g = hg' = f_1f_2 \underbrace{\Delta_{m+1}(\omega)\delta}_{\in \Delta_{m+1}(\Alt(X^{m+1}))} \tilde{\sigma}^{[m+1]}\eta
\]
and this proves the claim.

Suppose now that $h = \tilde{b}^{[1]}$ for some $b \in \Alt(\{x_1,\dots,x_r\}\times X)$.
We write 
\[
     \tilde{b}^{[1]} = \Delta_2(\phi_2(b)) \dots \Delta_{m+2}(\phi_{m+2}(b)) \tilde{b}^{[m+2]} = f_2 \delta' \tilde{b}^{[m+2]}
 \]
 for some $f_2 \in \bigoplus_{i=1}^{m+1}\Alt(X^{i})$ and $\delta' \in \Delta_{m+2}(\Alt(X^{m+2}))$.
Here is the crucial observation: By looking at the supports one sees that $(\tilde{b}^{[m+2]})^\delta = \tilde{\tau}^{[m+2]}$ for some $\tau \in \Alt(X^{m+1} \times \{x_1,\dots,x_r\} \times X)$.
In addition, by equation \eqref{eq:tilde-recursive} we have $\tilde{\sigma}^{[m+1]} = \tilde{\sigma}^{[m+2]} \eta'$ with $\eta' = \Delta_{m+2}(\phi^{(m)}_{m+2}(\sigma))$.
This gives
\[
	g = hg' = f_1f_2 \delta'\tilde{b}^{[m+2]} \delta \tilde{\sigma}^{[m+1]}\eta = f_1f_2 \delta'\delta \tilde{\tau}^{[m+2]}\tilde{\sigma}^{[m+2]} \eta'\eta
\]
and the assertion follows since $\tilde{\tau}^{[m+2]}\tilde{\sigma}^{[m+2]} = \widetilde{(\tau\sigma)}^{[m+2]}$ and moreover $\delta\delta', \eta'\eta$ are contained in  $\Delta_{m+2}(\Alt(X^{m+2}))\bigoplus_{i=1}^{m+1}\Alt(X^{i})$.
\end{proof}

Let $\sigma \in \Alt(X^{i})$. We define the \emph{support volume} of $\sigma$ on $X^{i}$ to 
 be $\vol(\sigma) =  \frac{|\supp_{X^i}(\sigma)|}{d^{i}}$. For $g \in G_{\mathcal{A}(d,r)}$ we write $\vol_i(g) = \vol(\pi_i(g))$.
 We say that $g \in G_{\mathcal{A}(d,r)}$ is \emph{consistent} at $w \in X^{i}$, if $\pi_{i+n}(g)(wy) = \pi_i(g)(w) y$ for all $n \geq 1$ and all $y \in X^{n}$. 
% The ratio of the number of consistent points in $X^{i}$ by $d^{i}$ will be called the \emph{consistency of $g$} at level $i$ and is denoted by $\cons_i(g)$. It is easy to check that support volume and consistency are invariant under conjugation in $G_{\mathcal{A}(d,r)}$.
 
 \begin{lemma}\label{lem:consistent-alt}
Let $g \in G_{\mathcal{A}(d,r)}$. 
%Then $\lim_{i\to \infty}\cons_i(g) = 1$. 
Assume that $g \not\in \bigoplus_{i\in \N} \Alt(X^{i})$. Then there is $k \in \N$ and a consistent point $w \in X^{k}$ with
 $g(w) \neq w$. In particular, $\vol_i(g)$ is bounded away from zero and $\pi_j(g) \neq 1$ for all $j \geq k$.
 \end{lemma}
 \begin{proof}
 %% Wird nicht mehr benötigt %%%%%
% We first observe that for all $g,h \in G_{\mathcal{A}(d,r)}$
% \[
% 	\cons_i(gh) \geq \cons_i(g)+ \cons_i(h) -1
% \]
% for all $i$. Indeed, if $w \in X^{i}$ is consistent for $h$ and $h(w)$ is consistent for $g$, then $w$ is consistent for $gh$.
% 
% By construction the elements in $\Delta_1(\Alt(X^{i}))$ are consistent at all points in $X^{i}$. For $c \in \Alt(rd)$ the element $\tilde{c}^{[1]}$ is
% consistent at all except for the $r+1$ points $x_d^{i}, x_d^{i-1}x_1, \dots, x_d^{i-1}x_r$. 
% In particular, $\lim_{i \to \infty} \cons(\tilde{c}^{[1]}) = 0$ and by induction on the word length the formula above shows that this holds true for all elements in $G_{\mathcal{A}(d,r)}$.
%
 By Lemma \ref{lem:weak-normalform} there is some $m$ such that we may write $g = f\delta\tilde{\sigma}^{[m+1]}\eta$. We will see that for some $k \leq m+3$ we can find a consistent~$w \in X^k$ in the support of~$g$. To prove the claim, we may work modulo $\bigoplus_{i=1}^m \Alt(X^{i})$ and ignore $f$. Also the assertion is invariant under conjugation with $\Delta_{m+1}(\Alt(X^{m+1}))$ and hence we can assume that 
 \[
 	g = \delta \tilde{\sigma}^{[m+1]}
 \]
 with $\delta \in \Delta_{m+1}(\Alt(X^{m+1}))$ and $\sigma \in \Alt(X^{m}\times\{x_1,\dots,x_r\}\times X)$.
 If $\sigma = 1$ then every point $w \in X^{m+1}$ in the support of  $\delta$ is consistent and we may pick $k=m+1$.
 %Assume that $\sigma \neq 1$.
 
% If $\delta = 1$, then we choose $k= m+2$ and take $w \in X^k$ in the support of $\phi^{(m)}_{m+2}(\sigma)$. The point $w$ does not lie in the support of $\phi^{(m)}_{j}(\sigma)$ for any $j \neq k$ and is hence moved consistently.
 
Suppose that $\sigma \neq 1$ and consider the case that for some $v \in X^m$ the point $vx_d$ lies in the support of $\delta$.
Then $w = vx_dx_{d-1} \in X^{m+2}$ lies in the support of $\delta$ but  does not belong to the support of $\tilde{\sigma}^{[m+1]}$. In particular, it is consistent and we can take $k = m+2$.
Finally, we assume that the support of $\delta$ is disjoint from $X^m\times \{x_d\}$. Then we pick $k = m+3$ and an element $w \in X^k$ in the support of $\phi_{k}^{(m)}(\sigma)$.
As $w$ is not in the support of $\phi^{(m)}_{j}(\sigma)$ for any $j \neq k$, it is moved consistently.
We conclude that $\vol_i(g) \geq \frac{1}{d^k}$ for all $i \geq k$.
 \end{proof}
 With these preparations, we can give two proofs of Theorem \ref{thm:simplehead-alt}. The first proof is shorter and explains more directly how consistently moved points imply simplicity. The second proof uses actions and contains more information about the structure of $Q_{\mathcal{A}(d,r)}$.
 
 \begin{proof}[First proof of Theorem \ref{thm:simplehead-alt}]
 Let $N \subseteq  G_{\mathcal{A}(d,r)}$ be a normal subgroup that contains $\bigoplus_{i\in \N} \Alt(X^{i})$ properly.
Pick $g \in N \setminus \bigoplus_{i\in \N} \Alt(X^{i})$. By Lemma \ref{lem:consistent-alt} there is some $k \geq 1$ and a consistent point $w \in X^k$ in the support of $g$. Say $g(w) = v$. Let $\tau \in \Alt(X^{k+1})$ be non-trivial with support in $\{w\}\times X$. Say $\tau$ acts as $\tau(wx) = w t(x)$ for some $t \in \Alt(X)$. Then for all $x \in X$ and $y \in X^{i}$  we obtain (using that $w$ is a consistent point for $g$)
\begin{align*}
		[\Delta_{k+1}(\tau),g](vxy) &= \Delta_{k+1}(\tau)g\Delta_{k+1}(\tau)^{-1}g^{-1}(vxy)\\  
		&= \Delta_{k+1}(\tau)g(\tau^{-1}(wx)y)
		= \Delta_{k+1}(\tau)g(wt^{-1}(x)y)\\
		 &= \Delta_{k+1}(\tau)(vt^{-1}(x)y) = vt^{-1}(x)y.
\end{align*}
Moreover, $[\Delta_{k+1}(\tau),g]$ acts trivially on all points that do not start with $v$.
This means that $[\Delta_{k+1}(\tau),g] \in \Delta_{k+1}(\Alt(X^{k+1}))$ and is non-trivial. Since $\Alt(X^{k+1})$ is simple, we deduce that $\Delta_{k+1}(\Alt(X^{k+1}))$ is contained in $N$. By Lemma \ref{lem:normal-closure-of-Delta} this implies that $N$ contains $K_{k+1}$. Since $G_{\mathcal{A}(d,r)}$ is generated by $K_{k+1}$ and
 $\bigoplus_{i\in \N} \Alt(X^{i})$ (see Corollary~\ref{cor:group-associated-to-shifted-B-telescope}) we deduce that $N = G_{\mathcal{A}(d,r)}$. It follows that $Q_{\mathcal{A}(d,r)}$ is simple.
 \end{proof}
\begin{proof}[Second proof of Theorem \ref{thm:simplehead-alt} ]
We consider the action of $ \mathcal{S} = \mathcal{A}(d,r)$ on $(X_i)_{i\in \N}$ defined in Example \ref{ex:alt-action}. It follows from Lemma \ref{lem:consistent-alt} that the kernel of this action is $\bigoplus_{i\in\N} \Alt(X^{i})$, i.e., the action induces a faithful action of $Q_{\mathcal{S}}$ on $X_\infty$.

We establish the simplicity of $Q_{\mathcal{S}}$ using Lemma \ref{lem:simplicity-criterion}.
To this end, it suffices to verify that $\mathcal{N}^\infty = \bigcup_{j=1}^\infty \bigcap_{k=j}^\infty N_{G_{\mathcal{S}}}(C^{[k]})$  is a maximal subgroup of $G_{\mathcal{S}}$.
We do so by proving that $\mathcal{N}^\infty$ is the stabilizer of $x = [x_d] \in X_\infty$ in $G_{\mathcal{S}}$.
Indeed, as the action is primitive, this is a maximal subgroup.

We consider this stabilizer in $Q_{\mathcal{S}}$ and we claim that $\St_{Q_{\mathcal{S}}}(x)$ is generated by $Z^\infty = \bigcup_{j=1}^\infty \bigcap_{k=j}^\infty Z_{G_\mathcal{S}}(C^{[k]})$ and the groups $C^{[k]}$ with $k \in \N$ modulo $\bigoplus_{i\in \N} \Omega_i$. 
In that case the desired equality follows from the maximality of the stabilizer.
%This implies that $\St_{G_\mathcal{S}}(x) \subseteq \mathcal{N}^\infty$ and since the stabilizer is a maximal subgroup, this gives 
By definition every element of $Q_\mathcal{S}$ can be written as a product of images of elements in $\bigcup_{k \in \N} \Delta_k(\Omega_k) \cup \bigcup_{j=1}^\infty C^{[j]}$. Recall that the image of $\Delta_k(\Omega_k)$ in $Q_\mathcal{S}$ is contained in the image of $\Delta_{k+1}(\Omega_{k+1})$.  Let $g \in \St_{Q_\mathcal{S}}(x)$. We prove the claim by induction on the word length with respect to this generating set. For the induction base we observe that an element $\Delta_k(\sigma)$ with $\sigma \in \Omega_k$ fixes $x$ exactly if $\sigma$ fixes $x_d^k$; in this case $\Delta_k(\sigma)$ centralizes $C^{[j]}$ for all $j > k$, since these groups have a disjoint support for the action on $X^n$ for all $n$. 

Let $g \in \St_{Q_\mathcal{S}}(x)$ of length $\ell > 1$. For the induction step there is nothing to do, if the first or last letter of $g$ is in the image of $C^{[j]}$. We assume that $g = \omega_1 c_2 \omega_3 \dots c_{\ell-1}\omega_\ell$ with $\omega_i$ in the image of $\Delta_{k_i}(\Omega_{k_i})$ and $c_i$ in the image of $C^{[k_i]}$. If $\omega_\ell$ fixes $x$, then it centralizes  $C^{[j]}$ for $j > k_\ell$ and there is nothing to do. 
Otherwise, we may replace $c_{\ell-1}$ by $\omega' c'_{\ell-1}$ to assume that $k_{\ell-1} \geq k_\ell +2$. Then $\omega_\ell^{-1}c'_{\ell-1}\omega_\ell$ fixes $x$. Since $\omega_\ell$ moves $x_d^{k_\ell}$, we have $z := \omega_\ell^{-1}(x_d^{k_\ell}) \neq x_d^{k_\ell}$. The element $\omega_\ell^{-1}c'_{\ell-1}\omega_\ell$ is supported in $\bigcup_{j \geq k_\ell+2} \{z x_d^{j-k_{\ell}-2}\}\times \{x_1,x_2,\dots,x_r\} \times X$ and hence centralizes $C^{[k_{\ell}+2]}$. Now the remaining word $\omega_1 c_2 \omega_3 \dots \omega' \omega_{\ell}$ fixes $x$ and the assertion follows by induction.
\end{proof}

Recall that a finitely generated group is $\LEF$ if it is isomorphic to a subgroup of an ultraproduct of finite groups.
A finitely presented $\LEF$ group is residually finite; see \cite{VershikGordon}.
Lemma \ref{lem:consistent-alt} shows that the homomorphism $G_{\mathcal{A}(d,r)} \to \prod_{\ell\in \N}\Alt(d^\ell)$ induces an injective homomorphism from the head
$Q_{\mathcal{A}(d,r)}$ into the ultraproduct 
$\prod_{\mathcal{U}} \Alt(d^\ell)$ for every free ultrafilter $\mathcal{U}$ on $\N$.
We deduce the following result:
\begin{corollary}
$Q_{\mathcal{A}(d,r)}$ is a simple $\LEF$ group and is not finitely presented.
\end{corollary}

%%%%%%%%%%%%%%%%%%%%%%%%%%%%%
\section{Simplicity of the head for $\mathcal{PSL}_d(F)$}\label{sec:simplicity-psl}

In this section we study the telescope $\mathcal{PSL}_d(F)$ over some fixed field $F$ defined in \S \ref{ex:PSL} and we establish the following result.

\begin{theorem}\label{thm:simplicity-psl}
Let $d \geq 4$ and let $F$ be a field. The head $Q_{\mathcal{PSL}_d(F)}$ is a simple group.
\end{theorem}
Our argument is similar to the one for $\mathcal{A}(d,r)$. The first step is a weak normal form theorem for the telescope $\mathcal{PSL}_d(F)$.
Let $Y$ be a finite set. Recall that for $g \in \SL_{Y}(F)$ the support $\supp(g)$ of $g$ is the set of basis elements $y \in Y$ such that $g(y) \neq y$ or $g(y^*) \neq y^*$ where $y^*$ is the associated element in the dual basis. In other words, a basis element $y \in Y$ does not lie in the support if and only if $g(y) = y$ and the complement $F[Y\setminus \{y\}]$ is $g$-stable.
Let  $m \geq 0$ be given. For all $j \geq 0$ we define
\[
Y(m,j) = X^m\times\{x_d^{j}\}\times \{x_1,x_2\} \times X \subseteq X^{m+j+2}.
\]
For $i \geq m+2$ let $\phi^{(m)}_i\colon \SL_{Y(m,0)}(F) \to \SL_{X^{i}}(F)$ be the embedding induced from the
canonical identification of the sets $Y(m,0)$ and $Y(m,i-m-2)$, i.e. the elements in the image are supported in $Y(m,i-m-2)$.
As for $\mathcal{A}(d,r)$, we note that for $m=0$ we have $\phi_i^{(0)} = \phi_i$ and that the images of $\iota_{i,j}\circ\phi^{(m)}_i$ and $\phi^{(m)}_j$ have disjoint supports for $i < j$ and hence commute.
Let $n > m$. For $s \in \SL_{Y(m,0)}(F)$ we define $\tilde{s}^{[n]} \in  \prod_{i=1}^\infty \SL_{X^{i}}(F)$ by $\pi_i(\tilde{s}^{[n]}) = 1$ for all $i <n+1$ and by
\[
\pi_i(\tilde{s}^{[n]})  = \prod_{k=n+1}^{i} \iota_{k,i}(\phi^{(m)}_k(s))
\]
for $i \geq n+1$.
As before, we obtain the recursive description 
\begin{equation}\label{eq:tilde-recursive-sl}
 \tilde{s}^{[n]} = \Delta_{n+1}(\phi^{(m)}_{n+1}(s)) \tilde{s}^{[n+1]} =  \tilde{s}^{[n+1]}\Delta_{n+1}(\phi^{(m)}_{n+1}(s)).
\end{equation}
%%%%
\begin{lemma}\label{lem:weak-normalform-sl}
For every $g \in G_{\mathcal{SL}_d(F)}$ there are $m \geq 0$, $\sigma \in \SL_{Y(m,0)}$, $\delta, \eta \in \Delta_{m+1}(\SL_{X^{m+1}}(F))$ and $f \in \bigoplus_{i=1}^m \SL_{X^{i}}(F)$ 
such that
\[
	g = f \delta \tilde{\sigma}^{[m+1]} \eta.
\]
The number $m$ is at most the word length of $g$ with respect to the generating set $\Delta_1(\SL_{X}(F)) \cup C^{[1]}$.
\end{lemma}
\begin{proof}
The proof of Lemma \ref{lem:weak-normalform} carries over to this situation mutatis mutandis.
\end{proof}
In the proof of Theorem \ref{thm:simplicity-psl} we use the following elementary lemma.
\begin{lemma}\label{lem:linear-non-commute}
 Let $V = U \oplus W$ be a finite dimensional vector space over $F$ and assume $\dim U \geq 2$. Let $\alpha\colon V \to V$ be an automorphism such that $\alpha|_U$ is not a scalar multiple of $\id_U$. 
 Then there is $\tau \in \SL(U)$ such that the commutator $[\alpha,\tau \oplus \id_W]$ is not a scalar multiple of the identity. 
 \end{lemma}
 \begin{proof}
Splitting $U$ into two components, we may assume $\dim U = 2$.
Choose bases of $U$ and $W$. Then $\alpha$ and $\tau$ are represented by matrices
 \[
 	A = \begin{pmatrix} A_1 & B \\ C & A_2 \end{pmatrix}, T = \begin{pmatrix} T_1 & 0 \\ 0 & 1 \end{pmatrix}
 \]
 where $A_1, T_1$ are $2\times 2$-matrices. We calculate
  \[
 	T^{-1}AT = \begin{pmatrix} T_1^{-1}A_1T_1 & T_1^{-1}B \\ CT_1 & A_2 \end{pmatrix}.
 \]
If $A_1$ is not a scalar matrix, then it is easy to find $T_1 \in \SL_2(F)$ such that $T_1^{-1}A_1T_1$ is not a scalar multiple of $A_1$.
 
 Assume that $A_1 = \lambda I$ is a scalar matrix. We note that $C \neq 0$ is this case.
Suppose that $C$ has rank $2$ (this is in particular the case if $\lambda = 0$, since $\alpha$ is an automorphism), then pick $x,y \in F^2$ to be an arbitrary basis. Suppose that $C$ has rank $1$. In this case, we pick a basis $x,y \in F^2$ with $Cx \neq 0$ and $Cy = 0$. We define $T_1$ with $T_1x = y$ and $T_1y = -x$, then $CT_1$ is not a scalar multiple of $C$.
\end{proof}

\begin{proof}[Proof of Theorem \ref{thm:simplicity-psl}]
An element of the form $\Delta_n(\lambda I)$ for  $n\in \N$ and $\lambda \in F$ in $G_{\mathcal{SL}_d(F)}$ will be briefly called a scalar in this proof.
Let $N \subseteq G_{\mathcal{SL}_d(F)}$ be a normal subgroup that contains the subgroup $\bigoplus_{i\in \N} \SL_{X^{i}}(F)$ and all scalars.
Let $g \in N$ be an element that is not congruent to a scalar  modulo $\bigoplus_{i\in \N} \SL_{X^{i}}(F)$. We show that $N$ contains a non-scalar element in $\Delta_n(\SL_{X^n}(F))$. Then the simplicity of $\PSL_{X^n}(F)$ implies $N = G_{\mathcal{SL}_d(F)}$ as in the first proof of Theorem \ref{thm:simplehead-alt}.

By Lemma \ref{lem:weak-normalform-sl} there is some $m$ such that we may write $g = f\delta\tilde{\sigma}^{[m+1]}\eta$. To prove the claim, we may work modulo $\bigoplus_{i=1}^m \SL_{X^{i}}(F)$ and ignore $f$. Also the assertion is invariant under conjugation with $\Delta_{m+1}(\SL_{X^{m+1}}(F))$ and we can assume that 
 \[
 	g = \delta \tilde{\sigma}^{[m+1]}
 \]
 with $\delta \in \Delta_{m+1}(\SL_{X^{m+1}}(F))$ and $\sigma \in \SL_{Y(m,0)}$.
If $\sigma = 1$, then $\delta$ is not a scalar and there is nothing left to do.
 
Assume now that $\sigma \neq 1$. Suppose there is some $v \in X^m$ such that $v\otimes x_d$ is not an eigenvector for $\delta = \Delta_{m+1}(\omega)$.
Then $F[X^{m+3}] = U \oplus W$ where $U = F[X^m \times \{x_d\} \times \{x_3\} \times X]$ and $W$ is the complement with respect to the basis $X^{m+3}$. We observe that $\tilde{\sigma}^{[m+1]}$ preserves this decomposition and acts trivially on $U$.
 By Lemma \ref{lem:linear-non-commute}, there is $\tau \in \SL_{X^{m+3}}(F)$ such that  $\tau$ commutes with $\tilde{s}^{[m+1]}$ but $[\omega\otimes\id_{F[X^2]},\tau]$ is not a scalar matrix. Thus $[g,\Delta_{m+3}(\tau)]$ lies in $N$ and is a non-scalar element from $\Delta_{m+3}(\SL_{X^{m+3}}(F))$.
 
Next, we assume that $\delta$ acts by a scalar $\lambda$ on $F[X^m\times \{x_d\}]$. Let $k = m+3$. Multiplying $g$ with $\Delta_{m+1}(\lambda^{-1} I)$ we may assume that $\lambda = 1$. 

We observe that for $\sigma\neq 1$ the element $\tilde{\sigma}^{[m+1]}$ does not act like a scalar on $F[X^m\times \{x_d\} \times \{x_1,x_2,x_3\} \times X]$
and only the component $\Delta_{m+3}(\phi^{(m)}_{m+3}(\sigma))$ acts non-trivially on this subspace.
Let $\tau \in \SL_{X^{m+3}}(F)$ be supported in $X^m\times\{x_d\}\times\{x_1,x_2,x_3\}\times X$ such that $[\tilde{\sigma}^{[m+1]},\Delta_{m+3}(\tau)] = [\Delta_{m+3}(\phi^{(m)}_{m+3}(\sigma)),\Delta_{m+3}(\tau)]$ is not a scalar matrix. Then
\begin{align*}
	[g,\Delta_{m+3}(\tau)] &= [\delta\tilde{\sigma}^{[m+1]},\Delta_{m+3}(\tau)] \\
	&= \delta[\tilde{\sigma}^{[m+1]},\Delta_{m+3}(\tau)]\delta^{-1}[\delta,\Delta_{m+3}(\tau)] \\
	&=  \delta[\Delta_{m+3}(\phi^{(m)}_{m+3}(s)),\Delta_{m+3}(\tau)]\delta^{-1}[\delta,\Delta_{m+3}(\tau)]
\end{align*}
and since $\delta$ acts trivially on vectors in $F[X^m\times\{x_d\}]$ we see that this is a non-scalar element from $\Delta_{m+3}(\SL_{X^{m+3}}(F))$ that is contained in $N$.
\end{proof}
%%%%%
From the simplicity of $Q_{\mathcal{PSL}_d(F)}$ it follows that all of its non-trivial actions are faithful. We deduce:
\begin{corollary}
The action of $Q_{\mathcal{PSL}_d(F)}$ on $P_\infty$ defined in Example \ref{ex:action-projectivespace} is faithful.
\end{corollary}
As for the alternating telescope we may deduce the following:
\begin{corollary}
Let $F$ be a finite field.
The head $Q_{\mathcal{PSL}_d(F)}$ is a simple $\LEF$ group and is not finitely presented.
\end{corollary}

\section{Embedding residually finite groups into simple groups}\label{sec:embedding-groups}

The goal of this section is to show that every finitely generated, residually finite, (amenable) group $G$ embeds into a finitely generated, simple (amenable) $\LEF$ group $Q$.
%, which proves Theorem~\ref{thm:embedding} from the introduction.
In fact we will see in Theorem~\ref{thm:embeddings-into-simple-groups} that $G$ embeds in another finitely generated, residually finite, (amenable) group $H$ that admits a projection onto $Q$, whose kernel intersects the image of $G$ in $H$ trivially.
%restriction to the image of $G$ in $H$ is an embedding.
%This generalizes a result of Bon~\cite[Theorem 1.1]{Bon17} who proved that every bounded automaton group embeds into a finitely generated, simple, amenable group.
We will see that it is sufficient to prove this result for perfect groups.
As the first step in the proof of Theorem \ref{thm:embeddings-into-simple-groups}, we start by embedding a finitely generated residually finite perfect group into the automorphism group of a spherically symmetric rooted tree. We take this perspective to illustrate the similarity between telescopes and groups acting on rooted trees. We briefly recall some standard terminology.

\subsection{Groups acting on rooted trees}
Let $X = (X_n)_{n \in \N}$ be a sequence of finite sets of cardinality $\abs{X_n} \geq 2$, which we will think of as alphabets.
For each $\ell \in \N_0$, we consider the set of words of length $\ell$ given by $X^{\ell} \defeq X_1 \times \ldots \times X_{\ell}$, where $X^{0} = \{\emptyset\}$ is defined to be the singleton consisting of the unique word $\emptyset$ of length $0$.
The \emph{spherically homogeneous rooted tree associated to $X$}, denoted by $\T = \T_X$, is the rooted tree with vertex set $X^{\ast} = \bigcup \limits_{\ell=0}^{\infty} X^{\ell}$ and root vertex $\emptyset$, where two vertices $v,w$ are connected by an edge if and only if there is a letter $x \in X_n$ for some $n \in \N$ such that either $v = wx$ or $w = vx$.
We write $\Aut(\T)$ for the group of all automorphisms of $\T$ that fix the root $\emptyset$.
Note that the distance of a vertex $v$ to the root, which we call the \emph{level} of $v$, is preserved under the action of $\Aut(\T)$ and coincides with the word length of $v$. It therefore follows that the sets $X^{\ell}$ of vertices of level $\ell$ are stable under the action of $\Aut(\T)$.

% A subgroup $G$ in $\Aut(\T)$ is said to act \emph{spherically transitively} on $\T$ if the action of $G$ on $X^{\ell}$ via $\pi_{\ell}$ is transitive for every $\ell \in \N_0$.
%For each such $\ell$, the \emph{level $\ell$ stabilizer subgroup} in $G$ is defined by
%\[
%\St_G(\ell) \defeq \bigcap \limits_{v \in X^{\ell}} \St_G(v),
%\]
%where $\St_G(v)$ denotes the stabilizer of $v$ in $G$.
%Note that $\St_G(\ell)$ coincides with the kernel of the restriction of $\pi_{\ell}$ to $G$, which makes it a normal subgroup of finite index in $G$.
%The \emph{rigid stabilizer} of $v$ in $G$, denoted by $\RiSt_G(v)$, is the subgroup of $\St_G(v)$ consisting of those elements that fix every word that does not contain $v$ as an initial subword.
%The subgroup of $\St_G(\ell)$ that is generated by the groups $\RiSt_G(v)$ with $v \in X^{\ell}$ is called the \emph{rigid level $\ell$ stabilizer subgroup} in $G$ and will be denoted by $\RiSt_G(\ell)$.
%It can be easily seen that $\RiSt_G(\ell)$ is a normal subgroup of $G$.
%If moreover $\RiSt_G(\ell)$ has finite index in $G$ for every $\ell \in \N$ and $G$ acts spherically transitively on $\T$, then $G$ is said to be a \emph{branch subgroup} of $\Aut(\T)$.
%More generally, we say that a group $G$ is a \emph{branch group} if it is isomorphic to a branch subgroup of the automorphism group of a spherically homogeneous rooted tree.
%
%\subsection{Directed automorphisms}\label{subsec:directed-autos}
%
Many important examples of groups acting on rooted trees are generated by tree automorphisms that are either rooted or directed.
Rooted automorphisms are easy to describe.
They are given by those tree automorphisms that only permute the first letter of a word in $X^{\ast}$.
Thus the group of rooted automorphisms of $\T$ is canonically isomorphic to $\Sym(X_1)$.
In order to define directed tree automorphisms of $\T$ we need to fix a direction in $\T$, which is given by a sequence $o = (o_i) \in X^{\infty} \defeq \prod \limits_{i \in \N} X_i$.
Let $Y_i$ denote the complement of $o_i$ in $X_i$ and let $\mathcal{S} \defeq \prod \limits_{i \in \N} Y_i$.
Consider now a group $G$ that acts on $X_i$ for every $i \geq 2$ and let $G_i \leq \Sym(X_i)$ denote the image of $G$ under this action.
Given $g \in G$, we write $g_i$ for the image of $g$ in $G_i$.
Let $\T_{i}$ be the spherically homogeneous rooted tree associated to the subsequence $(X_i, X_{i+1}, \ldots)$ of $X$.
Note that we obtain an embedding $\rho_i \colon G_i \rightarrow \Aut(\T_i)$ of $G_i$ as a group of rooted automorphisms of $\T_i$ by setting
\[
\rho_i(g_i)x_ix_{i+1}\cdots x_k = (g_ix_i)x_{i+1}\cdots x_k.
\]
In the following we will identify $G_i$ with its image under $\rho_i$.
For every $\alpha \in \mathcal{S}$ and every $g \in G$ we can recursively define a family of directed automorphisms $\tilde{g}_{[i]}^\alpha \in \Aut(\T_i)$ as follows.
Given a letter $x \in X_i$ and a word $v \in X_{i+1} \times \ldots \times X_k$ we define
\[
	\tilde{g}_{[i]}^\alpha(xv) =
	\begin{cases}
	x\tilde{g}_{[i+1]}^{\alpha}(v) & \text{ if } x = o_i\\
	xg_{i+1}(v) & \text{ if } x = \alpha_i\\
	xv & \text{ if } x \not\in \{o_i,\alpha_i\}.
	\end{cases}
\]
%Note in particular that each $\tilde{g}_{[i]}^\alpha$ acts trivially on the first level of $\T_i$.
We observe that $\tilde{\cdot}^\alpha_{[i]} \colon G \rightarrow \Aut(\T_i),\ g \mapsto \tilde{g}^\alpha_{[i]}$ defines a homomorphism with kernel $\bigcap_{k \geq i} \ker(G \to G_k)$, whose image we denote by $\tilde{G}_{[i]}^\alpha$.

%To this end, we use a variation of our $\Alt(rd)$-telescopes whose associated group can be combined with $\Gamma_{G,H}^{\alpha,\beta}$ in order to generate an amenable subgroup of $\prod \limits_{n} \Alt(X_n)$ that admits a simple quotient in which $\Gamma_{G,H}^{\alpha,\beta}$ embeds.

\subsection{Embedding perfect groups into groups of $B$-telescopes}\label{subsec:embedding-perfect-groups}
In this section we construct the telescope used in the proof of the embedding theorem.
We fix a non-trivial, finitely generated, residually finite, perfect group $G$ and a sequence $(N_i)_{i \geq 2}$ of normal subgroups of finite index in $G$.
Suppose that $\bigcap \limits_{i \geq 2} N_i = 1$ and that each $N_i$ has index at least $6$ in $G$.
We consider the canonical action of $G$ on the finite quotients $X_i \defeq G / N_i$, which we will think of as the alphabets of the previous subsection.
Let $X_1$ be a set of cardinality $6$ and let $o,\alpha,\delta,\varepsilon,y,z \in \prod \limits_{n \in \N} X_n$ be sequences of pairwise distinct elements in each coordinate $X_n$.
By viewing $o$ as a direction in the tree $\T = \T_X$, where $X = (X_n)_{n \in \N}$, we obtain a map $\tilde{\cdot}^\alpha_{[1]} \colon G \rightarrow \Aut(\T)$ as above.
Note that the condition $\bigcap \limits_{i \geq 2} N_i = 1$ guarantees that $\tilde{\cdot}^\alpha_{[1]}$ defines an isomorphism onto its image $\tilde{G}_{[1]}^\alpha$.
The main step in the proof of Theorem~\ref{thm:embeddings-into-simple-groups} is the construction of a flexible $B$-telescope $\mathcal{H} \defeq ((\Omega_{\ell})_{\ell \in \N},(\phi_{\ell})_{\ell \geq 2})$ whose associated group $G_{\mathcal{H}}$ contains $\tilde{G}_{[1]}^\alpha$.
To abbreviate some notation we write $o^i \defeq o_1o_2 \cdots o_i$.
Consider the group $B_0 \defeq \Alt(\{\delta,\varepsilon\} \times \{y,z,o\})$.
For each $n \geq 2$ let
\[
\psi_n^{0} \colon B_0 \rightarrow \Alt( \{o^{n-2}\} \times \{\delta_{n-1},\varepsilon_{n-1}\} \times \{y_{n},z_{n},o_{n}\})
\]
be the isomorphism that is induced by the obvious bijection between the sets $\{\delta,\varepsilon\} \times \{y,z,o\}$ and $\{o^{n-2}\} \times \{\delta_{n-1},\varepsilon_{n-1}\} \times \{y_{n},z_{n},o_{n}\}$.
Moreover we define the map
\[
\psi_n^{\alpha} \colon G \rightarrow \Alt( \{o^{n-2}\} \times \{\alpha_{n-1}\} \times X_n)
\]
by setting $\psi_n^{\alpha}(g)(o^{n-2} \alpha_{n-1} x) = o^{n-2} \alpha_{n-1} g(x)$ for every $n \geq 2$.
We can now define the group $B$ for our $B$-telescope as the direct product $B = B_0 \times G$.
For each $n \in \N$ let $\Omega_n = \Alt(X^n)$. We write $P_\mathcal{H} = \prod_{n} \Omega_n$.
As in the case of our $\Alt(rd)$-telescopes we consider the embeddings $\iota_{i,j} \colon \Omega_i \rightarrow \Omega_j$ given by $\iota_{i,j}(\sigma)(vw) = \sigma(v)w$, where $v$ has word length $i$.
%By combining $\psi_n^{0}$ and $\psi_n^{\alpha}$
Next we define the maps
\[
\phi_n \colon B \rightarrow \Omega_n,\ (b,g) \mapsto \psi_n^0(b) \cdot \psi_n^{\alpha}(g)
\]
for every $n \geq 2$.
Since the supports of the groups $\psi_n^0(B_0)$ and $\psi_n^{\alpha}(G)$ are disjoint by construction, it follows that $\phi_n$ is a well-defined homomorphism.
In order to see that $\tilde{G}^\alpha_{[1]}$ is contained in $G_{\mathcal{H}}$, we apply the embedding
\[
\iota \colon \SAut(\T) \rightarrow P_{\mathcal{H}},\ \gamma \mapsto (\pi_n(\gamma))_{n \in \N},
\]
where $\SAut(\T)$ denotes the subgroup of tree automorphisms that act as even permutations on each level of $\T$.
From the definitions it is evident that the restriction of $\tilde{\cdot}^{[1]} \colon B \rightarrow P_{\mathcal{H}}$ to $G$ has the same image as $\iota \circ \tilde{\cdot}^\alpha_{[1]}$.
Using the fact that $\tilde{G}^\alpha_{[1]}$ consists of automorphisms of $\T$, it moreover follows that the intersection $\tilde{G}^\alpha_{[1]} \cap \bigoplus \limits_{n \in \N} \Omega_n$ is trivial.
Once we have proven that $\mathcal{H}$ is a flexible $B$-telescope, this will allow us to view $\tilde{G}^\alpha_{[1]}$ as a subgroup of $Q_{\mathcal{H}}$.

\begin{lemma}\label{lem:gluing-alternating-groups}
Let $X$ be a finite set and let $Y,Z \subseteq X$ be subsets with $Y \cap Z \neq \emptyset$.
Suppose that $\abs{Y},\abs{Z} \geq 3$.
Then the subgroup of $\Alt(X)$ generated by $\Alt(Y)$ and $\Alt(Z)$ coincides with $\Alt(Y \cup Z)$.
\end{lemma}
\begin{proof}
The claim is a direct consequence of the O'Nan–Scott theorem. 
\end{proof}

\begin{proposition}\label{prop:H-is-flexible-B-system}
The telescope $\mathcal{H}$ defined above is a flexible $B$-telescope.
\end{proposition}
\begin{proof}
The verification of~\eqref{eq:commutator} and the flexibility axioms (F1) - (F3) follows along the lines of the verification of the corresponding axioms for $\mathcal{A}(d,r)$, which was carried out in
%Lemma~\ref{lem:supports} and
Proposition~\ref{prop:B-telescope} by looking at the supports.
The only difference is that one has to replace the $3$-cycles $(x_d^{\ell-1}x_{d-2},x_d^{\ell-1}x_{d-1},x_d^{\ell})$ from the flexibility of $\mathcal{A}(d,r)$ with the $3$-cycles $(o^{\ell-1} y_{\ell},o^{\ell-1} z_{\ell},o^{\ell}) \in \Omega_{\ell}$.
%so we omit it here.
What remains to prove is~\eqref{eq:generation}, i.e.\ that the group $K \defeq \langle \{ hB_{n+1} h^{-1} \mid h \in \Omega_{n,n+1} \} \rangle$ coincides with $\Omega_{n+1}$ for every $n \in \N$.
Recall that $\Omega_n = \Alt(X^n)$.
By definition we have
\[
B_{n+1} = \phi_{n+1}(B_0 \times G)
= \langle \psi_{n+1}^{0}(B_0), \psi_{n+1}^{\alpha}(G) \rangle.
\]
Consider the group $\psi_{n+1}^{0}(B_0) = \Alt( \{o^{n-1}\} \times \{\delta_{n},\varepsilon_{n}\} \times \{y_{n+1},z_{n+1},o_{n+1}\})$.
Since $\Omega_n$ acts $2$-transitively on $X^n$, it follows by conjugating $\psi_{n+1}^{0}(B_0)$ with the $\iota_{n,n+1}$-image of the stabilizer of $o^{n-1} \delta_{n}$ in $\Omega_{n}$ that
$\Alt( \{v,o^{n-1}\delta_n\} \times \{y_{n+1},z_{n+1},o_{n+1}\})$ is contained in $K$ for every $v \in X^n$.
In view of Lemma~\ref{lem:gluing-alternating-groups} this implies that $K$ contains the group $\Alt(X^{n} \times \{y_{n+1},z_{n+1},o_{n+1}\})$.
Since $G$ acts transitively on $X_{n+1}$ we deduce that $\psi_{n+1}^{\alpha}(G)$ acts transitively on $\{o^{n-1}\} \times \{\alpha_n\} \times X_{n+1}$.
Given $v \in X^n$, we can therefore choose an appropriate $\Omega_{n,n+1}$-conjugate $Q_v \leq \Alt(\{v\} \times X_{n+1})$ of $\psi_{n+1}^{\alpha}(G)$ that acts transitively on $\{v\} \times X_{n+1}$.
Since the intersection
\[
\Alt(X^{n} \times \{y_{n+1},z_{n+1},o_{n+1}\}) \cap \Alt(X^{n} \times \{y_{n+1},z_{n+1},o_{n+1}\})^q
\]
contains $\Alt((X^{n} \setminus \{v\}) \times \{y_{n+1},z_{n+1},o_{n+1}\})$ for each $q \in Q_v$, we may apply Lemma~\ref{lem:gluing-alternating-groups} to deduce that $\Alt(X^{n} \times \{y_{n+1},z_{n+1},o_{n+1}\} \cup \{v\} \times X_{n+1})$ is contained in $K$.
Using the fact that $X^{n+1}$ is covered by subsets of the form $X^{n} \times \{y_{n+1},z_{n+1},o_{n+1}\} \cup \{v\} \times X_{n+1}$, a final application of Lemma~\ref{lem:gluing-alternating-groups} shows that $K = \Omega_{n+1}$, which completes the proof.
\end{proof}

\begin{remark}\label{rem:phi-n-not-injective}
Unlike in our previous examples of $B$-telescopes, the maps $\phi_n$ in $\mathcal{H}$ are far from being embeddings.
Indeed, the group $B$ contains $G$ and hence is typically infinite, while each $\Omega_n$ is finite.
\end{remark}

\subsection{Telescopes and Cantor sets}
The ultimate goal of this section is to prove that
 the groups of the telescopes $\mathcal{A}(d,r)$ and  $\mathcal{H}$ are amenable (provided that $G$ is amenable).
This will be done by applying a celebrated criterion of Juschenko, Nekrashevych and de la Salle~\cite{JuschenkoNekrashevychdelaSalle16} for detecting amenability of groups acting on a Cantor set.
The first step is therefore to define appropriate actions of the groups $G_{\mathcal{H}}$ and $G_{\mathcal{A}(d,r)}$ on a Cantor set.
It will be convenient to put these two groups on a common ground as follows.
Consider again a sequence $X = (X_n)_{n \in \N}$ of finite sets $X_n$ of cardinality $\abs{X_n} \geq 2$ and let $\mathfrak{C} \defeq \prod \limits_{n=1}^{\infty} X_n$ denote the Cantor set given by the product of these sets.
We fix a point $o = (o_n) \in \mathfrak{C}$.
Consider now a flexible $B$-telescope $\mathcal{S} = ((\Omega_n)_{n \in \N},  (\phi_n)_{n \geq 2})$ consisting of transitive subgroups $\Omega_n \leq \Sym(X^n)$, where
the transition maps $\iota_{n,m} \colon \Omega_n \rightarrow \Omega_m$ are defined by letting $\iota_{n,m}(\omega)$ act on words of length $m$ by permuting their initial subwords of length $n$.
We assume that 
\begin{equation}\label{eq:assumption-for-well-def-action}
\supp(B_i) \subseteq \{o^{i-2}\}\times (X_{i-1}\setminus\{o_{i-1}\}) \times X_i.
\end{equation}

We show that $\mathfrak{C}$ carries an action of $Q_\mathcal{S}$ by homeomorphisms.
For every $n \in \N$ we have an action of $\Omega_n$ on $X^n = \prod_{i=1}^n X_i$ and these actions induce actions of $\widetilde{G}_{\mathcal{S}}$ on $\mathfrak{C}$ by acting on the first $n$ coordinates. More precisely, for $g \in \widetilde{G}_{\mathcal{S}}$ and $\xi = (\xi_1,\xi_2,\dots)$ we define
\[
  g._n \xi = \Delta_n(\pi_n(g))(\xi).
\]

\begin{lemma}\label{lem:well-defined-action}
For every $g \in \widetilde{G}_\mathcal{S}$ and every $\xi \in \mathfrak{C}$ the sequence
$(g._n\xi)_{n\in \N}$ stabilizes and 
\[
	g.\xi = \lim_{n \to \infty} g._n \xi
\]
defines an action of $\widetilde{G}_\mathcal{S}$ on $\mathfrak{C}$ by homeomorphisms.
\end{lemma}
\begin{proof}
As in the proof of Lemma \ref{lem:action} it is sufficient to verify the assertions on a set of generators.
The huge group $\widetilde{G}_\mathcal{S}$ is generated by elements in $\Delta_j(\Omega_j)$ and $C^{[k]}$.
For $g  \in \Delta_j(\Omega_j)$ the action at $\xi$ is stable for all $n \geq j$ by definition of the transition maps and, as they only act on the first $n$ entries, it is clear that they are homeomorphisms.

Now we consider the elements of $C^{[k]}$. By assumption the supports of the actions of the groups $B_i$ on $\mathfrak{C}$ are disjoint clopen sets and every $\xi \in \mathfrak{C}$ can live in at most one of these support sets. On each of these support sets, the action agrees with the action of an element in $\Delta_{k+j}(\Omega_{k+j})$.
This implies that the action of each $g \in C^{[k]}$ stabilizes and is continuous.
\end{proof}

Let us now study the action of $G_{\mathcal{S}}$ on $\mathfrak{C}$ that is provided by Lemma~\ref{lem:well-defined-action}.
We start by determining the kernel of this action.

\begin{lemma}\label{lem:kernel-is-direct-sum}
Suppose that the head $Q_{\mathcal{S}}$ of $\mathcal{S}$ is simple.
Then the kernel of the action of $\widetilde{G}_{\mathcal{S}}$ on $\mathfrak{C}$ is given by $\bigoplus \limits_{n=1}^{\infty} \Omega_{\ell}$.
In particular, $Q_{\mathcal{S}}$ acts faithfully on $\mathfrak{C}$.
\end{lemma}
\begin{proof}
Let $N$ denote the kernel of the action of $\widetilde{G}_{\mathcal{S}}$ on $\mathfrak{C}$.
We know that $\bigoplus \limits_{\ell=1}^{\infty} \Omega_{\ell}$ is a subgroup of $\widetilde{G}_{\mathcal{S}}$ and by construction it acts trivially on $\mathfrak{C}$.
This shows that $\bigoplus \limits_{\ell=1}^{\infty} \Omega_{\ell} \subseteq N$.
Using our assumption that $Q_{\mathcal{S}} = \widetilde{G}_{\mathcal{S}} / \bigoplus \limits_{\ell=1}^{\infty} \Omega_{\ell}$ is simple and (noting that the action is not trivial), it follows that $N$ coincides with $\bigoplus \limits_{\ell=1}^{\infty} \Omega_{\ell}$.
\end{proof}

\begin{remark}\label{rem:QS-not-finitely-presented}
Assume that $G_\mathcal{S}$ is finitely generated and contains $\bigoplus \limits_{\ell \in \N} \Omega_{\ell}$.
Note that $\bigoplus \limits_{\ell \in \N} \Omega_{\ell}$ is not finitely generated as a normal subgroup of $G_{\mathcal{S}}$.
As a consequence, $Q_{\mathcal{S}} = G_{\mathcal{S}} / \bigoplus \limits_{\ell \in \N} \Omega_{\ell}$ is not finitely presented, which is not obvious if $Q_{\mathcal{S}}$ would have been defined directly via its action on $\mathfrak{C}$.
\end{remark}

\subsection{Amenability of $G_{\mathcal{S}}$}\label{subsec:amenability-of-embeddings}

Now that we have a faithful action of $Q_{\mathcal{S}}$ on $\mathfrak{C}$, we want to formulate conditions for this action that ensure the amenability of $Q_{\mathcal{S}}$.
Since $\bigoplus \limits_{\ell \in \N} \Omega_{\ell}$ is locally finite and hence amenable, Lemma~\ref{lem:kernel-is-direct-sum} implies that these conditions then also ensure the amenability of $G_{\mathcal{S}}$.
For a definition of amenability we refer the reader to \cite{Juschenko22}.
%Let us first recall F\text{\o}lners~\cite{Folner55} characterization of (discrete) amenable groups.
%\begin{definition}\label{def:amenability-discrete}
%A group $\Gamma$ is \emph{amenable} if for every $\varepsilon > 0$ and every finite subset $A \subseteq \Gamma$ there is a non-empty finite subset $F \subseteq \Gamma$ such that $\abs{(F \cdot A) \setminus F} \leq \varepsilon \cdot \abs{F}$.
%\end{definition}
We need to introduce some notation that will be used to formulate the amenability criterion in~\cite{JuschenkoNekrashevychdelaSalle16} for our present setting.
To do so, we consider for each $v \in X^{\ast}$ the open subset $U_v \subseteq \mathfrak{C}$ consisting of all sequences starting with $v$.
Using these sets we can conveniently generalize our previous notion of consistency of an element $\alpha \in \Homeo(\mathfrak{C})$ at $v \in X^{\ell}$ by imposing that there is some $\omega \in \Omega_{\ell}$ such that $\alpha$ coincides with $\Delta_{\ell}(\omega)$ on $U_v$.

\begin{definition}\label{def:bounded-type}
We say that $\alpha \in \Homeo(\mathfrak{C})$ is of \emph{bounded type} if there is some $C \geq 0$ such that
\[
\abs{\Set{v \in X^{\ell}}{ \alpha \text{ is not consistent at } v}} \leq C
\]
for every $\ell \in \N$. 
\end{definition}

\begin{definition}\label{def:germ}
Let $\Gamma$ be a group that acts on $\mathfrak{C}$.
For each $\xi$ the group of \emph{germs of $\Gamma$ at $\xi$} is defined as $\mathcal{O}_{\xi}(\Gamma) = \St_{\Gamma}(\xi) / \sim$,
%where $\St_{\Gamma}(\xi)$ denotes the stabilizer of $\xi$ in $\Gamma$ and
where $g,h \in \St_{\Gamma}(\xi)$ satisfy $g \sim h$ if and only if $g$ and $h$ coincide on some neighbourhood of $\xi$ in $\mathfrak{C}$.
\end{definition}

The following result is a special case of~\cite[Theorem 16]{JuschenkoNekrashevychdelaSalle16}.

\begin{theorem}[Juschenko, Nekrashevych, de la Salle]\label{thm:amenability-criterion}
Let $G$ be a group that acts faithfully
by homeomorphisms of bounded type on $\mathfrak{C}$.
If the group of germs $\mathcal{O}_{\xi}(G)$ is amenable for every $\xi \in \mathfrak{C}$, then $G$ is amenable.
\end{theorem}

The following lemma provides us with a normal form for elements in $Q_{\mathcal{S}}$ that will be suitable for verifying the conditions of Theorem~\ref{thm:amenability-criterion}.
For simplicity, we will denote the elements of $Q_\mathcal{S}$ in that lemma and its proof with representatives from $G_\mathcal{S}$.
% since the Lemma and its proof only concerns.
Recall that in $Q_\mathcal{S}$ we have $\Delta_n(\Omega_n) \subseteq \Delta_{n+1}(\Omega_{n+1})$ for all $n$.

\begin{lemma}\label{lem:normalform-short-proof}
For every $q \in Q_{\mathcal{S}}$ there are $m,n \in \N$ such that $q$ is represented by an element of the form $\varepsilon c_1^{\delta_1} \cdots c_m^{\delta_m} \eta$, where $c_i \in C^{[n+1]}$, $\delta_i \in \Delta_{n}(\Omega_{n})$, $\varepsilon, \eta \in \Delta_{n+1}(\Omega_{n+1})$, and
$\pi_n(\delta_i^{-1})(o^{n}) \neq \pi_n(\delta_j^{-1})(o^{n})$ for all $i \neq j$.
\end{lemma}
\begin{proof}
We prove the lemma by induction on the word length of $q \in Q_{\mathcal{S}}$ with respect to the generating set $\Delta_1(\Omega_1) \cup C^{[1]}$.
If the word length is $1$, then the claim is clear.
Suppose that $q = \varepsilon c_1^{\delta_1} \cdots c_m^{\delta_m} \eta$ as in the lemma.
Let $b_i \in B$, $\omega_i \in \Omega_n$, and $\sigma,\tau \in \Omega_{n+1}$ be such that $c_i = \tilde{b}_i^{[n+1]}$, $\delta_i = \Delta_{n}(\omega_i)$, $\varepsilon = \Delta_{n+1}(\sigma)$, and $\eta = \Delta_{n+1}(\tau)$.
By multiplying $q$ from right with an element of the form $\Delta_1(\rho)$ for some $\rho \in \Omega_1$ we obtain
\[
q \Delta_1(\rho)
= \varepsilon c_1^{\delta_1} \cdots c_m^{\delta_m} \eta \Delta_1(\rho)
= \varepsilon c_1^{\delta_1} \cdots c_m^{\delta_m} \eta',
\]
where $\eta' = \Delta_{n+1}(\tau \iota_{1,n+1}(\rho))$ and the equations are considered in $Q_{\mathcal{S}}$.
Thus $q \Delta_1(\rho)$ can be written in a normal form as in the lemma.
Let us next multiply $q$ from the right with an element of the form $\tilde{b}^{[1]}$ for some $b \in B$.
We first calculate
\begin{align*}
\eta \tilde{b}^{[1]}
&= \Delta_{n+1}(\tau) \tilde{b}^{[n+2]} \Delta_{n+2}(\iota_{2,n+2}(\phi_2(b)) \cdots \iota_{n+1,n+2}(\phi_{n+1}(b)) \phi_{n+2}(b))\\
&= (\tilde{b}^{[n+2]})^{\Delta_{n+1}(\tau^{-1})} \Delta_{n+2}(\iota_{n+1,n+2}(\tau) \cdot \iota_{2,n+2}(\phi_2(b)) \cdots \phi_{n+2}(b))\\
&= c_{m+1}^{\delta_{m+1}} \Delta_{n+2}(\rho),
\end{align*}
where $c_{m+1} = \tilde{b}^{[n+2]}$ and $\delta_{m+1} \in \Delta_{n+1}(\Omega_{n+1})$, $\rho \in \Delta_{n+2}(\Omega_{n+2})$ are chosen appropriately.
%$= \iota_{n+1,n+2}(\tau) \cdot \iota_{2,n+2}(\phi_2(b)) \cdots \phi_{n+2}(b)$
Let us next consider the element $c_1^{\delta_1} \cdots c_m^{\delta_m}$ more closely.
For every $k \geq 0$ we have
\begin{align*}
\supp_{X^{n+k+2}}(c_i^{\delta_i})
&= \delta_i^{-1}(\supp_{X^{n+k+2}}(c_i))\\
&\subseteq \delta_i^{-1}(\{o^{n}\} \times X_{n+1} \setminus \{o_{n+1}\} \times X_{n+2} \times \cdots \times X_{n+k+2})\\
&= \{\omega_i^{-1}(o^{n})\} \times X_{n+1} \setminus \{o_{n+1}\} \times X_{n+2} \times \cdots \times X_{n+k+2}).
\end{align*}
As a consequence we see that
\begin{equation}\label{eq:merging-conjugates}
c_i^{\delta_i}c_j^{\delta_j} = (\widetilde{b_ib_j}^{[n+1]})^{\delta_i}
\end{equation}
if $\omega_i^{-1}(o^{n}) = \omega_j^{-1}(o^{n})$ and $[c_i^{\delta_i},c_j^{\delta_j}] = 1$ otherwise.
In view of this, our calculation of $\eta \tilde{b}^{[1]}$, and the assumption
\[
\omega_i^{-1}(o^{n}) = \pi_n(\delta_i^{-1})(o^{n}) \neq \pi_n(\delta_j^{-1})(o^{n}) = \omega_j^{-1}(o^{n})
\]
for $i \neq j$, we obtain
\begin{align*}
q \tilde{b}^{[1]}
&= \varepsilon c_1^{\delta_1} \cdots c_m^{\delta_m} c_{m+1}^{\delta_{m+1}} \Delta_{n+2}(\rho)\\
&= \Delta_{n+1}(\sigma) (\tilde{b}_1^{[n+1]})^{\delta_1} \cdots (\tilde{b}_m^{[n+1]})^{\delta_m} c_{m+1}^{\delta_{m+1}} \Delta_{n+2}(\rho)\\
&= \Delta_{n+2}\bigg(\iota_{n+1,n+2}(\sigma) \cdot \prod \limits_{i=1}^m \phi_{n+2}(b_i)^{\iota_{n,n+2}(\omega_i)}\bigg) \cdot d_1^{\delta_1} \cdots d_m^{\delta_m} c_{m+1}^{\delta_{m+1}} \cdot \Delta_{n+2}(\rho),
\end{align*}
where $d_i = \tilde{b}_i^{[n+2]}$ for every $i$.
By `merging' $c_{m+1}^{\delta_{m+1}}$ with some $d_i^{\delta_i}$ as in~\eqref{eq:merging-conjugates} if necessary we see that $q \tilde{b}^{[1]}$ can be written in the desired normal form.
\end{proof}

\begin{proposition}\label{prop:amenability-of-GS}
Suppose that $B$ is a perfect, amenable group and that $Q_{\mathcal{S}}$ is simple.
Then the groups $G_{\mathcal{S}}$ and $Q_{\mathcal{S}}$ are amenable as well.
\end{proposition}
\begin{proof}
From Lemma~\ref{lem:kernel-is-direct-sum} we know that $G_{\mathcal{S}}$ is an extension of $Q_{\mathcal{S}}$ by the elementary amenable group $\bigoplus \limits_{n \in \N} \Omega_{n}$.
It therefore suffices to prove the claim for $Q_{\mathcal{S}}$, which we think of as a subgroup $\Homeo(\mathfrak{C})$.
From the definition of the generators in $\Delta_1(\Omega_1)  \cup C^{[1]}$ it directly follows that they act as homeomorphisms of bounded type on $\mathfrak{C}$.
Using this and the elementary fact that the set of homeomorphisms of bounded type of $\mathfrak{C}$ forms a group, we deduce that $Q_{\mathcal{S}}$ consists of homeomorphisms of bounded type.
Let us next verify that $\mathcal{O}_{\xi}(Q_{\mathcal{S}})$ is a homomorphic image of $B$, and hence amenable, for every $\xi \in \mathfrak{C}$.
To see this, let $q \in Q_{\mathcal{S}}$ be an element that fixes $\xi$.
Suppose first that $q$ is represented by $\Delta_{n}(\omega)$ for some $n \in \N$ and $\omega \in \Omega_{n}$.
Then $\omega$ fixes $v \defeq \xi_1\ldots\xi_n$ and hence acts as the identity on the open neighbourhood $U_v$ of $\xi$ that consists of all sequences that start with $v$.
As a consequence, $q$ is trivial in $\mathcal{O}_{\xi}(Q_{\mathcal{S}})$.
Suppose next that $q \in Q_{\mathcal{S}}$ is arbitrary.
From Lemma~\ref{lem:normalform-short-proof} we know that $q$ is represented by an element of the form $\varepsilon c_1^{\delta_1} \cdots c_m^{\delta_m} \eta$, where $c_i \in C^{[n+1]}$, $\delta_i \in \Delta_{n}(\Omega_{n})$, $\varepsilon, \eta \in \Delta_{n+1}(\Omega_{n+1})$, and
$\pi_n(\delta_i^{-1})(o^{n}) \neq \pi_n(\delta_j^{-1})(o^{n})$ for $i \neq j$.
Then $\varepsilon \cdot \eta$ fixes $\xi$ and $c_1^{\delta_1} \cdots c_m^{\delta_m}$ fixes $\eta(\xi)$.
Indeed, this follows from the fact that elements in $\Delta_{n+1}(\Omega_{n+1})$ only permute initial subwords of length $n+1$ while the elements in $C^{[n+1]}$ fix initial subwords of length $n+1$.
In this case we just have seen that $\varepsilon \cdot \eta$ represents the trivial element in $\mathcal{O}_{\xi}(Q_{\mathcal{S}})$.
Thus the germ of $q$ at $\xi$ can be written as $\eta^{-1} c_1^{\delta_1} \cdots c_m^{\delta_m} \eta$.
Since $\pi_n(\delta_i^{-1})(o^{n}) \neq \pi_n(\delta_j^{-1})(o^{n})$ for $i \neq j$ it follows that the supports of the elements $c_i^{\delta_i}$ lie in disjoint open subsets of $\mathfrak{C}$.
Thus there is some index $i_0$ such that the germ of $q$ at $\xi$ is represented by $c_{i_0}^{\delta_{i_0} \eta}$.
%For convenience we just write $c = c_{i_0}$ and $\delta = \delta_{i_0}$.
We can therefore assume that $q$ is represented by $c^{\delta}$ for appropriate elements $c \in C^{[n+1]}$ and $\delta \in \Delta_{n+1}(\Omega_{n+1})$.
%Since $c^{\delta}$ would represent the trivial element in $\mathcal{O}_{\xi}(Q_{\mathcal{S}})$ otherwise
Without loss of generality we may assume that the restriction of $c$ to every neighbourhood of $\delta(\xi)$ is non-trivial.
From the definition of $C^{[n+1]}$ it is evident that $\supp_{\mathfrak{C}}(c)$ is contained in $U_{o^{n-1}}$.
Thus $\delta(\xi)$ lies in $U_{o^{n-1}}$.
Let $b \in B$ be such that $c = \widetilde{b}^{[n+1]}$.
Since $\widetilde{b}^{[n+1]} = \Delta_{n+2}(\phi_{n+2}(b)) \cdot \widetilde{b}^{[n+2]}$ fixes $\xi$ and since the supports of $\Delta_{n+2}(\phi_{n+2}(b))$ and $\widetilde{b}^{[n+2]}$ in $\mathfrak{C}$ are disjoint we deduce that both of these elements fix $\delta(\xi)$.
As we have already seen, $\Delta_{n+2}(\phi_{n+2}(b))$ represents the trivial element in $\mathcal{O}_{\delta(\xi)}(Q_{\mathcal{S}})$.
Hence $\delta(\xi)$ lies in $\supp_{\mathfrak{C}}(\widetilde{b}^{[n+2]}) \subseteq U_{o^{n}}$.
Inductively we obtain $\delta(\xi) \in \bigcap \limits_{k \geq n} U_{o^{n}} = \{o^{\infty}\}$, i.e.\ $\delta(o^{\infty}) = \xi$.
This shows that there is some $\ell \in \N$ and a word $v \in X^{\ell}$ such that $\xi = v o_{\ell+1} o_{\ell+2} \cdots$.
Thus there is some $\omega \in \Omega_{\ell}$ with $\omega(\xi) = o^{\infty}$, from which we deduce that the images of $c^{\delta}$ and $c^{\Delta_{\ell}(\omega)}$ in $\mathcal{O}_{\xi}(Q_{\mathcal{S}})$ coincide.
Since $\omega$ only depends on $\xi$ this shows that $\mathcal{O}_{\xi}(Q_{\mathcal{S}})$ is a homomorphic image of $(C^{[n+1]})^{\Delta_{\ell}(\omega)}$ and therefore amenable.
\end{proof}

\begin{corollary}\label{cor:2-generated-amenable}
Given two integers $d \geq 5$ and $2 \leq r \leq d-3$,
the groups $G_{\mathcal{A}(d,r)}$ and $Q_{\mathcal{A}(d,r)}$ are amenable.
\end{corollary}
\begin{proof}
We have to show that $\mathcal{A}(d,r)$ satisfies the conditions of the $B$-telescope $\mathcal{S}$ in Proposition~\ref{prop:amenability-of-GS}.
Since $B \cong \Alt(rd)$ is perfect and $Q_{\mathcal{A}(d,r)}$ is simple by Theorem~\ref{thm:simplehead-alt}, it remains to show that $\mathcal{A}(d,r) = ((\Omega_n)_{n \in \N},  (\phi_n)_{n \geq 2})$ satisfies~\eqref{eq:assumption-for-well-def-action}, i.e.\ that
\[
\supp_{X^n}(B_n) \subseteq \{o^{n-2}\} \times (X_{n-1} \setminus \{o_{n-1}\}) \times X_n
\]
for every $n \geq 2$.
In view of Lemma~\ref{lem:supports} this becomes evident by setting $o_n = x_d$ for each $n \in \N$.
\end{proof}

The group $Q_{\mathcal{S}}$ in Proposition~\ref{prop:amenability-of-GS} yields one of the rare examples of finitely generated, infinite, simple amenable groups, whose existence was first established in~\cite{JuschenkoMonod13}. To the best of our knowledge $Q_{\mathcal{A}(d,r)}$ is the first $2$-generated example; in addition, it admits a $2$-generated residually finite amenable cover.

\begin{theorem}\label{thm:2-generated-amenable}
There exists an infinite, $2$-generated, residually finite, amenable group $G$ that admits an infinite, simple quotient $Q$.
\end{theorem}
\begin{proof}
Consider the $\Alt(rd)$-telescope $\mathcal{A}(d,r)$.
From Theorem~\ref{thm:frame-minimal-Alt} we know that $G_{\mathcal{A}(d,r)}$ is $2$-generated.
In particular the group $Q_{\mathcal{A}(d,r)}$, being a quotient of $G_{\mathcal{A}(d,r)}$ by Corollary~\ref{cor:first-levels}, is also $2$-generated.
Moreover $Q_{\mathcal{A}(d,r)}$ is simple by Theorem~\ref{thm:simplehead-alt}.
Now it remains to apply Corollary~\ref{cor:2-generated-amenable}, which tells us that the groups $G_{\mathcal{A}(d,r)}$ and $Q_{\mathcal{A}(d,r)}$ are amenable.
\end{proof}

Going back to our $B$-telescope $\mathcal{H}$, we can now prove the following.

\begin{theorem}\label{thm:embeddings-into-simple-groups}
Let $G$ be a finitely generated, residually finite (amenable) group.
There is a finitely generated, residually finite (amenable) group $H$ and an infinite, simple, (amenable) $\LEF$ group $Q$ such that
\begin{enumerate}
\item there is an embedding $\iota \colon G \rightarrow H$,
\item there is a projection $\pi \colon H \rightarrow Q$,
\item the composition $\pi \circ \iota$ is injective.
\end{enumerate}
\end{theorem}
\begin{proof}
In~\cite[Theorem 1.1]{KionkeSchesler22} the authors show that every finitely generated, residually finite, (amenable) group embeds in a finitely generated, residually finite, perfect (amenable) group.
We can therefore assume that $G$ is perfect.
In this case the group $B = B_0 \times G$ from our flexible $B$-telescope $\mathcal{H} = ((\Omega_{\ell})_{\ell \in \N},(\phi_{\ell})_{\ell \geq 2})$ is perfect since $B_0 \cong \Alt(6)$.
Moreover $B$ is amenable if and only if $G$ is amenable.
As noted in Subsection~\ref{subsec:embedding-perfect-groups}, the intersection $\tilde{G}^\alpha_{[1]} \cap \bigoplus \limits_{n \in \N} \Omega_n$ is trivial, from which it follows that $G \cong \tilde{G}^\alpha_{[1]}$ is contained in $Q_{\mathcal{H}}$.
To verify the conditions of the $B$-telescope $\mathcal{S}$ in Proposition~\ref{prop:amenability-of-GS} it remains to show that $Q_{\mathcal{H}}$ is simple and that $\mathcal{H}$ satisfies~\eqref{eq:assumption-for-well-def-action}.
The latter is a direct consequence of the definition of the maps $\phi_n$.
The proof of simplicity of $Q_{\mathcal{H}}$ follows the same lines as the first proof of Theorem~\ref{thm:simplehead-alt} provided in Section~\ref{sec:simplicity-alt}.
The only difference is that one has to replace the normal form in Lemma~\ref{lem:weak-normalform} with the normal form in Lemma~\ref{lem:normalform-short-proof}.
Finally, since $Q_{\mathcal{H}} = G_{\mathcal{H}} / \bigoplus \limits_{n \in \N} \Omega_n$ with finite groups $\Omega_n$, it follows that $Q_{\mathcal{H}}$ satisfies the $\LEF$ property.
\end{proof}

\section{An amenable-($\tau$) Grothendieck pair}\label{sec:grothendieck-pair}

%\begin{theorem}\label{thm:kassabov-nikolov-general}
%Let $(S_n)_{n \in \N}$ be a sequence of non-abelian finite simple groups that has essentially unbounded rank.
%Suppose that $P \defeq \prod \limits_{n \in \N} S_n$ has a $d$-generated dense subgroup.
%Then there is a $22(d+1)$-generated frame for $P$.
%\end{remark}

In this section we revisit the classical question of what can be said about a finitely generated, residually finite group $G$ if one knows its profinite completion $\widehat{G}$.
Let us investigate this question for two fairly oppositional properties in group theory: amenability and property $(\tau)$.
In order to underline how opposite these concepts are,
%of being amenable and having property $(\tau)$
we define them in terms of boundaries of subsets in their Cayley graphs.
To this end, it will be useful to write $\partial_{\Gamma(G,S)}(F) \defeq (F \cdot S \cup F \cdot S^{-1}) \setminus F$, where $S,F$ are two subsets of a group $G$.
The following characterization only applies to finitely generated, residually finite groups.
Many equivalent definitions of amenability and property $(\tau)$, which also apply in a more general setting, can be found in~\cite{CeccheriniCoornaert10}, respectively in~\cite{LubotzkyZuk05}.

\begin{definition}\label{def:amenable-and-tau}
Let $G$ be a finitely generated, residually finite group and let $S$ be a finite generating set of $G$.
Then $G$
\begin{enumerate}
\item is \emph{amenable} if for every $\varepsilon > 0$ there is a non-empty finite subset $F \subseteq G$ such that
\[
\abs{\partial_{\Gamma(G/N,S)}(F)} \leq \varepsilon \cdot \abs{F}
\]
for every finite index normal subgroup $N$ in $G$.
\item has \emph{property $(\tau)$} if there is some $\varepsilon > 0$ such that
\[
\abs{\partial_{\Gamma(G/N,S)}(F)} \geq \varepsilon \cdot \abs{F}
\]
for every finite index normal subgroup $N$ in $G$ and every non-empty subset $F \subseteq G/N$ of size $\abs{F} \leq \frac{\abs{G/N}}{2}$.
\end{enumerate}
\end{definition}

%In particular we see that both properties can be verified in the finite quotients of $G$ once a finite generating set is fixed.

Recall that a pair $(G,H)$ of finitely generated, residually finite groups is called a \emph{Grothendieck pair} if $G$ is a proper subgroup of $H$ and the inclusion $\iota \colon G \rightarrow H$ induces an isomorphism $\widehat{\iota} \colon \widehat{G} \rightarrow \widehat{H}$.
By combining strong results of Ershov, Jaikin-Zapirain, and Kassabov~\cite{ErshovJaikinKassabov17} on groups satisfying property $(\tau)$ with the properties of our groups $G_{\mathcal{S}}$ that we have seen so far, we can now easily deduce the following.

\begin{theorem}\label{thm:amenable-vs-tau}
There exists a Grothendieck pair $(G,H)$ consisting of an amenable group $G$ and a group $H$ with property $(\tau)$.
\end{theorem}
\begin{proof}
From Theorem~\ref{thm:frame-minimal-Alt} and Proposition~\ref{prop:amenability-of-GS} we know that $G \defeq G_{\mathcal{A}(d,r)}$ is a $2$-generated amenable frame for $P_{\mathcal{S}} = \prod \limits_{\ell \in \N} \Omega_{\ell} \cong \prod \limits \Alt(d^{\ell})$.
According to~\cite[p.3]{KassabovNikolov06} there is also a frame $K$ for $P_{\mathcal{S}}$ that has property $(\tau)$.
A detailed proof of this statement is given in~\cite[Section 9.6]{ErshovJaikinKassabov17}, where one has to replace $\N_{\geq 5}$ by $\Set{d^k}{k \in \N}$ in the construction.
See also~\cite{Kassabov08} for a slightly different proof.
By Lemma~\ref{lem:gluing} the group $H$ that is generated by the two subgroups $G,K \leq P$ is a frame for $P$ as well, which implies that the inclusion $\iota \colon G \rightarrow H$ induces an isomorphism $\widehat{\iota} \colon \widehat{G} \rightarrow \widehat{H}$.
It remains to observe that $H$ has property $(\tau)$.
To this end, we use the characterization of property $(\tau)$ in terms of representations with finite image.
More precisely, $H$ has property $(\tau)$ if the trivial representation is isolated in the space of all representations that factor through finite quotients, see~\cite[Definition 1.11]{LubotzkyZuk05}.
Using this, it directly follows that $H$ inherits property $(\tau)$ from $K$, as it has the same representations that factor through finite quotients and adding generators to a group with $(\tau)$ can only further isolate the trivial representation.
\end{proof}

\begin{remark}\label{rem:embeddings-of-countable-res-fin}
The proof of Theorem~\ref{thm:amenable-vs-tau} can be used to show that every countable residually finite group embeds into a finitely generated group with property $(\tau)$.
To see this, one only has to replace the group $G = G_{\mathcal{A}_{d,r}}$ in that proof with the group $G_{\mathcal{H}}$ from Section~\ref{sec:embedding-groups} and use the result of Wilson~\cite[Theorem A]{Wilson80} that every countable residually finite group embeds into a finitely generated residually finite group.
Thus, apart from the obvious ones, there are no obstructions for groups to embed into finitely generated groups with property $(\tau)$.
\end{remark}

A problem that is related to Theorem~\ref{thm:amenable-vs-tau} goes back to Lubotzky and Weiss~\cite[Conjecture 5.4]{LubotzkyWeiss92} who conjectured that every compact group $C$ that contains finitely generated dense subgroups $G,H$, where $G$ is amenable and $H$ has property $(\T)$, is finite.
This was disproven by Ershov and Jaikin-Zapirain~\cite[Section 6]{ErshovJaikin10} for the group $C = \prod \limits_{n=2}^{\infty} \SL_{3n}(\F_p)$, where $p > 2$ is a prime.
More precisely they showed that a certain dense subgroup of $C$, which was previously shown by Kassabov~\cite[Theorem 10]{Kassabov07a} to have property $(\tau)$, in fact has property $(\T)$.
On the other hand, it was proven by Lubotzky and Weiss that $C$ has a dense amenable subgroup.
Regarding all of this, the following question suggests itself.

\begin{question}\label{quest:grothendieck-T-amenable}
Does there exist a Grothendieck pair $(G,H)$, where $G$ is amenable and $H$ has property $(\T)$?
\end{question}

%Note however that Theorem~\ref{thm:amenable-vs-tau} says that this conjecture gets dramatically false if $(\T)$ is replaced by $(\tau)$.
Regarding the proof of Theorem~\ref{thm:amenable-vs-tau}, Question~\ref{quest:grothendieck-T-amenable} could be answered affirmatively if

1) there is a frame $K$ for $\prod \limits_{\ell=1}^{\infty} \Alt(d^{\ell})$ that has property $(\T)$, and

2) property $(\T)$ is a so-called weak-up profinite property, which means that for every Grothendieck pair $(K,H)$, where $K$ has property $(\T)$, it follows that $H$ has property $(\T)$.

\smallskip

\noindent As far as we know, none of these steps can be verified so far.
However, concerning the first step there are products of alternating groups known that contain dense subgroups with property $(\T)$.
See~\cite{CapraceKassabov22} for recent examples due to Caprace and Kassabov.
Regarding the second step, Cotton-Barratt~\cite[Theorem 1.3]{Cotton-Barratt13} has proven that a positive solution would imply the existence of hyperbolic groups that are not residually finite.
Without the restriction to Grothendieck pairs it was shown by Aka~\cite[Theorem 1]{Aka12} that property $(\T)$ is not profinite.

\section{Questions}\label{sec:questions}

We expect that there exist plenty of examples of telescopes.
%, especially of classical linear groups.
In addition to the examples given in this article, we are also aware of some telescopes of symplectic groups. We think it would be fruitful to have more telescopes available and pose the following

\begin{problem}
Construct telescopes of groups of Lie type that give rise to $2$-generated frames.
\end{problem}

To the best of the authors knowledge Theorem~\ref{thm:2-generated-amenable} and Theorem~\ref{thm:embeddings-into-simple-groups} provide the first examples of residually finite, amenable groups that admit infinite, simple quotients.
However the following question of Thom~\cite{ThomMO10} suggests that such groups might be abundant.

\begin{question}[Thom]\label{quest:thom}
Is every group the quotient of a residually finite group by a normal amenable subgroup?
\end{question}

Indeed, if Question~\ref{quest:thom} has a positive answer, then every finitely generated, simple, amenable group admits a residually finite cover, which then can be chosen to be finitely generated.

\smallskip

\noindent Guided by some of the breakthroughs that were achieved in the field of branch groups and their similarity of our $B$-telescopes $\mathcal{A}(r,d)$, we are tempted to formulate the following

\begin{question}
Is there a telescope $\mathcal{S}$ such that $G_{\mathcal{S}}$ or $Q_{\mathcal{S}}$ are finitely generated, infinite, and have
\begin{enumerate}
\item subexponential growth,
\item non-uniform exponential growth,
\item finite L-presentations, see~\cite{Bartholdi2003} for a definition of this concept.
\end{enumerate}
\end{question}

Judging from the theory of branch groups, the following problem might be easier.

\begin{problem}
Construct a telescope $\mathcal{S}$ whose associated group $G_{\mathcal{S}}$ is a finitely generated, infinite torsion group.
\end{problem}

According to a remarkable result of Brieussel~\cite[Theorem 3.3]{Brieussel14}, some of the classical examples of branch groups admit a F{\o}lner sequence that can be described in a surprisingly nice way.
We wonder if there is also a similarly nice sequence for $G_{\mathcal{A}(d,r)}$ or $Q_{\mathcal{A}(d,r)}$.

\begin{problem}
Find an easy description of a F{\o}lner sequence for $G_{\mathcal{A}(d,r)}$, and hence for the simple group $Q_{\mathcal{A}(d,r)}$.
\end{problem}

Regarding the lack of examples of simple amenable group, and in fact non-elementary amenable groups that are not defined via their actions on Cantor sets, it is intriguing to pose the following

\begin{problem}
Find new classes of telescopes whose associated groups are amenable but not elementary amenable.
E.g.\ the group $G_{\mathcal{SL}_d(\F_q)}$ can be easily seen to be not elementary amenable.
Is it amenable?
\end{problem}

%Finally, we would like to ask the following.
%
%\begin{problem}
%Is the embedding from Theorem~\ref{thm:2-generated-amenable} quasiisometric?
%\end{problem}

\bibliographystyle{amsplain}
\bibliography{literatur}

\end{document}